\documentclass[11pt]{article}

\usepackage[utf8]{inputenc} 
\usepackage[T1]{fontenc}    
\usepackage{hyperref}       
\usepackage{url}            
\usepackage{booktabs}       
\usepackage{amsfonts}       
\usepackage{nicefrac}       
\usepackage{tablefootnote}  

\usepackage[x11names]{xcolor}
\usepackage{multirow}
\usepackage{adjustbox}

\usepackage{makecell}
\usepackage{colortbl}

\usepackage{fullpage}
\usepackage{wrapfig}

\usepackage[protrusion=true,expansion=true]{microtype}

\usepackage{algorithm}
\usepackage{algorithmic}
\usepackage{amsthm} 
\usepackage{amsmath}
\usepackage{amssymb}
\usepackage{float}
\usepackage{mathtools}
\usepackage{url}
\usepackage{graphicx,times}
\usepackage{fancyheadings}
\usepackage{subfigure}
\usepackage{bm}

\usepackage{threeparttable}
\DeclareMathOperator*{\argmin}{arg\,min}

\usepackage{pifont}%

\allowdisplaybreaks

\newtheorem{theorem}{Theorem}
\newtheorem{proposition}{Proposition} 
\newtheorem{lemma}{Lemma} 
\newtheorem{assumption}{Assumption} 
\newtheorem{corollary}{Corollary} 
\newtheorem{definition}{Definition}

\usepackage{cleveref}

\title{Achieving $\mathcal{O}(\epsilon^{-1.5})$ Complexity in Hessian/Jacobian-free Stochastic Bilevel Optimization}

\author{%
  Yifan Yang, Peiyao Xiao and Kaiyi Ji
  \\
  Department of Computer Science and Engineering\\
  University at Buffalo\\
  \texttt{\{yyang99, peiyaoxi, kaiyiji\}@buffalo.edu} \\
  \date{May 30, 2023}
}

\begin{document}

\maketitle

\begin{abstract}
    In this paper, we revisit the bilevel optimization problem, in which the upper-level objective function is generally nonconvex and the lower-level objective function is strongly convex. Although this type of problem has been studied extensively, it still remains an open question how to achieve an $\mathcal{O}(\epsilon^{-1.5})$ sample complexity in Hessian/Jacobian-free stochastic bilevel optimization without any second-order derivative computation. To fill this gap, we propose a novel Hessian/Jacobian-free bilevel optimizer named FdeHBO, which features a simple fully single-loop structure, a projection-aided finite-difference Hessian/Jacobian-vector approximation, and momentum-based updates. Theoretically, we show that FdeHBO requires $\mathcal{O}(\epsilon^{-1.5})$ iterations (each using $\mathcal{O}(1)$ samples and only first-order gradient information) to find an $\epsilon$-accurate stationary point. As far as we know, this is the first Hessian/Jacobian-free method with an $\mathcal{O}(\epsilon^{-1.5})$
    sample complexity for nonconvex-strongly-convex stochastic bilevel optimization.  
\end{abstract}

\section{Introduction}
Bilevel optimization has 
drawn intensive attention due to its wide applications in meta-learning~\cite{franceschi2018bilevel,bertinetto2018meta,rajeswaran2019meta}, hyperparameter optimization~\cite{franceschi2018bilevel,shaban2019truncated,feurer2019hyperparameter}, reinforcement learning~\cite{konda2000actor,hong2020two}, signal process~\cite{kunapuli2008classification,flamary2014learning} and communication~\cite{ji2022network} and federated learning~\cite{tarzanagh2022fednest}. 
In this paper, we study the following stochastic bilevel optimization problem. 
\begin{align}\label{objective_total}
&\min_{x\in\mathbb{R}^{p}} \Phi(x)=f(x, y^*(x)) : =
\mathbb{E}_{\xi} \left[f(x,y^*(x);\xi)\right]  \nonumber
\\&\;\;\mbox{s.t.} \; y^*(x)= \argmin_{y\in\mathbb{R}^q} g(x,y):=\mathbb{E}_{\zeta} \left[g(x,y^*(x);\zeta)\right] 
\end{align}
where the upper- and lower-level objective functions   $f(x,y)$ and $g(x,y)$ take the expectation form w.r.t.~the random variables $\xi$ and $\zeta$, and are jointly continuously differentiable. In this paper, we focus on the nonconvex-strongly-convex bilevel setting, where the lower-level function $g(x,\cdot)$ is strongly convex and the upper-level function $\Phi(x)$ is nonconvex. This class of bilevel problems has been studied extensively from the theoretical perspective in recent years. 
Among them, \cite{ghadimi2018approximation,ji2021bilevel,arbel2022amortized,yang2021provably} proposed bilevel approaches with a double-loop structure, which update $x$ and $y$ in a nested manner. Single-loop bilevel algorithms have also attracted significant attention recently~\cite{hong2020two,yang2021provably,khanduri2021near,guo2021randomized,chen2021single,li2022fully,dagreou2022framework} due to the simple updates on all variables simultaneously. 
Among them, the approaches in \cite{yang2021provably,khanduri2021near,guo2021randomized} have been shown to achieve an $\mathcal{O}(\epsilon^{-1.5})$ sample complexity, but with expensive evaluations of Hessian/Jacobian matrices or Hessian/Jacobian-vector products.


Hessian/Jacobian-free bilevel optimization has received increasing attention due to its high efficiency and feasibility in practical large-scale settings. 
In particular, \cite{finn2017model,nichol2018first,vuorio2019multimodal}  directly ignored the computation of  
all second-order derivatives. 
However, such eliminations may lead to performance degeneration~\cite{antoniou2019train,fallah2020convergence}, and  can vanish the hypergradient for bilevel problems
with single-variable upper-level function, i.e., $\Phi(x)=f(y^*(x))$. \cite{song2019maml,gu2021optimizing} proposed zeroth-order approaches that approximate the hypergradient using only function values. These methods do not have a convergence rate guarantee. 
Recently, several Hessian/Jacobian-free bilevel algorithms were proposed by ~\cite{liu2022bome,sow2022constrained,shen2023penalty,chen2023bilevel} by reformulating the lower-level problem into the optimality-based constraints such as $g(x,y)\leq \min_yg(x,y)$. However, these approaches all focus on the deterministic setting, and their extensions to the stochastic setting remain unclear. 
In the stochastic case, \cite{sow2022on} proposed evolution strategies based bilevel method, which achieves a high sample complexity of $\mathcal{O}(p^2\epsilon^{-2})$, where $p$ is the problem dimension. Most recently, \cite{kwon2023fully} proposed two fully first-order (i.e., Hessian/Jacobian-free) value-function-based stochastic bilevel optimizer named F$^2$SA and its momentum-based version F$^3$SA with a single-loop structure, which achieves sample complexities of $\mathcal{O}(\epsilon^{-3.5})$ and $\mathcal{O}(\epsilon^{-2.5})$, respectively. However, there is still a large gap of $\epsilon^{-1}$, compared to the optimal complexity of $\mathcal{O}(\epsilon^{-1.5})$.  
Then, an important open question, as recently proposed by \cite{kwon2023fully}, is: 
 \begin{list}{$\bullet$}{\topsep=0.2ex \leftmargin=0.2in \rightmargin=0.in \itemsep =0.06in}
\item Can we achieve an $\mathcal{O}(\epsilon^{-1.5})$ sample/gradient complexity for nonconvex-strongly-convex bilevel optimization using only first-order gradient information? 
\end{list}

\begin{table*}[t]
 \centering
  \begin{tabular}{|l|l|l|l|c|}
   \hline
Algorithm & Samples & Batch size & \# of iterations &Loops per iteration 
\\ \hline \hline
PZOBO-S \cite{sow2022on}  &$ \mathcal{\widetilde O}(p^2\epsilon^{-3})$ & $\mathcal{O}(\epsilon^{-1})$& $\mathcal{\widetilde O}(p^2\epsilon^{-2})$&2
   \\ \hline 
   F$^2$SA \cite{kwon2023fully}  &$ \mathcal{\widetilde O}(\epsilon^{-3.5})$ & $\mathcal{O}(1)$& $ \mathcal{\widetilde O}(\epsilon^{-3.5})$&1
   \\ \hline 
   F$^3$SA \cite{kwon2023fully}  &$ \mathcal{\widetilde O}(\epsilon^{-2.5})$ & $\mathcal{O}(1)$& $ \mathcal{\widetilde O}(\epsilon^{-2.5})$&1
   \\ \hline 
      \cellcolor{blue!15}{FdeHBO (this paper)} & \cellcolor{blue!15}{$\mathcal{\widetilde O}(\epsilon^{-1.5})$} & \cellcolor{blue!15}{$\mathcal{O}(1)$ }  &\cellcolor{blue!15}{$\mathcal{\widetilde O}(\epsilon^{-1.5})$}&\cellcolor{blue!15}{1} 
   \\ \hline 
  \end{tabular}
   \vspace{0.2cm}
   \caption{Comparison of stochastic Hessian/Jacobian-free bilevel optimization algorithms. 
   } \label{tb:hessian-free}
\end{table*}

\subsection{Our Contributions}
In this paper, 
we provide an affirmative answer to the above question by proposing a new Hessian/Jacobian-free stochastic bilevel optimizer named FdeHBO with three main features. First, FdeHBO takes the fully single-loop structure with momentum-based updates on three variables $y,v$ and $x$ for optimizing the lower-level objective, the linear system (LS) of the Hessian-inverse-vector approximation, and the upper-level objective, respectively. Second, FdeHBO contains only a single matrix-vector product at each iteration, which admits a simple first-order finite-difference estimation. Third, FdeHBO involves an auxiliary projection on $v$ updates to ensure the boundedness of the Hessian-vector approximation error, the variance on momentum-based iterates, and the smoothness of the LS loss function. 
Our detailed contributions are summarized below. 
 \begin{list}{$\bullet$}{\topsep=0.2ex \leftmargin=0.2in \rightmargin=0.in \itemsep =0.06in}
 \item Theoretically, we show that FdeHBO achieves a sample/gradient complexity of $\mathcal{O}(\epsilon^{-1.5})$ and an iteration complexity of $\mathcal{O}(\epsilon^{-1.5})$ to achieve an $\epsilon$-accurate stationary point, both of which outperforms existing results by a large margin. As far as we know, this is the first-known method with an $\mathcal{O}(\epsilon^{-1.5})$ sample complexity for nonconvex-strongly-convex stochastic bilevel optimization using only first-order gradient information. 
 \item Technically, we show that the auxiliary projection can provide more accurate iterates on $v$ in solving the LS problem without affecting the overall convergence behavior, and in addition, provide a novel characterization of the gradient estimation error and the iterative progress during the $v$ updates, as well as the impact of the $y$ and $v$ updates on the momentum-based hypergradient estimation, all of which do not exist in previous studies. In addition, the finite-different approximations make the unbiased assumptions in the momentum-based gradients no longer hold, and hence a more careful analysis is required. 
\item As a byproduct, we further propose a fully single-loop momentum-based method named FMBO in the small-dimensional case with matrix-vector-based hypergradient computations. Differently from existing momentum-based bilevel methods with $\mathcal{O}(\log\frac{1}{\epsilon})$ Hessian-vector evaluations per iteration, FMBO contains only a single Hessian-vector computation per iteration with the same $\mathcal{O}(\epsilon^{-1.5})$ sample complexity.    
 
 \end{list}

We also want to emphasize our technical differences from previous works as below. 

\noindent
{\bf Comparison to existing momentum-based methods.} Previous momentum-based methods~\cite{yang2021provably,khanduri2021near} solve the linear system (LS) to a high accuracy of $\mathcal{O}(\epsilon)$, whereas our algorithm includes a new estimation error by the single-step momentum update on LS, and this error is also correlated with the lower-level updating error and the hypergradient estimation error. In addition, due to the finite-difference approximation, the stochastic gradients in all three updates on $y,v,x$ are no longer unbiased. 
Non-trivial efforts need to be taken to deal with such challenges and derive the optimal complexity. 

\noindent
{\bf Comparison to existing fully single-loop methods.} 
The analysis of the single-step momentum update in solving the LS requires the smoothness of the LS loss function and the boundedness of LS gradient variance, both of which may not be satisfied. 
To this end, we include an auxiliary projection and show it not only guarantees these crucial properties, but also, in theory, provides an improved per-iteration progress. As a comparison, existing works on fully single-loop stochastic bilevel optimization such as SOBA/SABA~\cite{dagreou2022framework} and  FLSA~\cite{li2022fully} with a new time scale to update the LS problem often assume that the iterates on $v$ are bounded during the process. We do not require such assumptions. In addition, an $\mathcal{O}(\epsilon^{-1.5})$ complexity has not been established for fully single-loop bilevel algorithms yet. 

\subsection{Related Work}
\noindent{\bf Bilevel optimization methods.}
 Bilevel optimization, which was first introduced by~\cite{bracken1973mathematical}, has been studied for decades. By replacing the lower-level problem with its optimality conditions,
\cite{hansen1992new, gould2016differentiating, shi2005extended,sinha2017review} reformulated the bilevel problem to the single-level problem. Gradient-based bilevel methods have shown great promise recently, which can be  divided into approximate implicit differentiation  
(AID)~\cite{domke2012generic, pedregosa2016hyperparameter, liao2018reviving,arbel2022amortized} and iterative differentiation (ITD)~\cite{maclaurin2015gradient, franceschi2017forward,finn2017model, shaban2019truncated, grazzi2020iteration} based approaches. Recently, a bunch of stochastic bilevel algorithms has been proposed via Neumann series~\cite{chen2021single,ji2021bilevel}, recursive momentum~\cite{yang2021provably,huang2021biadam,guo2021randomized} and variance reduction~\cite{yang2021provably,dagreou2022framework}.
Theoretically, the convergence of bilevel optimization has been analyzed by \cite{franceschi2018bilevel,shaban2019truncated,liu2021value,ghadimi2018approximation,ji2021bilevel,hong2020two,arbel2022amortized,dagreou2022framework}. 
Among them, \cite{ji2021lower} provides the lower complexity bounds for deterministic bilevel optimization with (strongly-)convex upper-level functions. \cite{guo2021randomized,chen2021single,yang2021provably,khanduri2021near} achieved the near-optimal sample complexity with second-order derivative computations. 
Some works studied deterministic bilevel optimization with convex or Polyak-Lojasiewicz (PL) lower-level problems via mixed gradient aggregation~\cite{sabach2017first,liu2020generic,li2020improved}, log-barrier regularization~\cite{liu2021value}, primal-dual method~\cite{sow2022constrained} and dynamic barrier~\cite{ye2022bome}.  More results and details can be found in the survey by  \cite{liu2021investigating}. 


\noindent{\bf Hessian/Jacobian-free bilevel optimization.} Some Hessian/Jacobian-free bilevel optimization methods have been proposed recently by~\cite{sow2022on,liu2018darts,finn2017model,gu2021optimizing,song2019maml,nichol2018first}. Among them,  FOMAML~\cite{finn2017model,nichol2018first} and MUMOMAML~\cite{vuorio2019multimodal} directly ignore the computation of  
all second-order derivatives. 
Several Hessian/Jacobian-free bilevel algorithms were proposed by ~\cite{liu2022bome,sow2022constrained,shen2023penalty,chen2023bilevel} by replacing the lower-level problem with the optimality conditions as the constraints.  However, these approaches focus only on the deterministic setting. 
Recently, zeroth-order stochastic approaches have been proposed for the hypergradient estimation~\cite{song2019maml,gu2021optimizing,sow2022on}. Theoretically,  
\cite{sow2022on} analyzed the convergence rate for their method. \cite{kwon2023fully} proposed fully first-order stochastic bilevel optimization algorithms based on the value-function-based lower-level problem reformulation. 
This paper proposes a new Hessian/Jacobian-free stochastic bilevel algorithm that for the first time achieves an $\mathcal{O}(\epsilon^{-1.5})$ sample complexity. 

\noindent
{\bf Momentum-based bilevel approaches.}
The recursive momentum technique was first introduced by~\cite{cutkosky2019momentum,tran2019hybrid} for minimization problems to improve the SGD-based updates in theory and in practice. This technique has been incorporated in stochastic bilevel optimization~\cite{khanduri2021near,chen2021single,guo2021stochastic,guo2021randomized,yang2021provably}. These approaches involve either Hessian-inverse matrix computations or a subloop of a number of iterations in the Hessian-inverse-vector approximation. As a comparison, our proposed method 
takes the simpler fully single-loop structure, and only uses the first-order gradient information. 

\noindent
{\bf Finite-difference matrix-vector  approximation.}
The finite-difference matrix-vector estimation has been studied extensively in the problems of escaping from saddle points~\cite{allen2018neon2} \cite{carmon2018accelerated} (some other related works can be found therein), neural architecture search (NAS)~\cite{liu2018darts} and meta-learning~\cite{fallah2020convergence}. However, such finite-different estimation can be sensitive to the selection of the smoothing constant, and may suffer from some numerical issues in practice \cite{jin2017escape}\cite{jorge2006numerical}, such as rounding errors. It is interesting but still open to developing a fully first-order stochastic bilevel optimizer without the finite-different matrix-vector estimation. We would like to leave it for future study.  


\section{Algorithms}
In this section, we first describe the hypergradient computation in bilevel optimization, and then
present the proposed Hessian/Jacobian-free bilevel method.  

\subsection{Hypergradient Computation}
One major challenge in bilevel optimization
lies in computing the hypergradient $\nabla\Phi(x)$ due to the implicit and complex dependence of the lower-level minimizer $y^*$ on $x$. To see this, if $g$ is twice differentiable, $\nabla_yg$ is continuously differentiable  
and the Hessian $\nabla_{yy}^2g(x,y^*(x))$ 
is invertible, using the implicit function theorem (IFT)~\cite{griewank2008evaluating,blondel2021efficient}, the hypergradient $\nabla \Phi(x)$ takes the form of 
\begin{align}\label{def:Phi}
    \nabla \Phi(x) &= \nabla_xf(x,y^*) - \nabla_{xy}^2g(x,y^*)\big[\nabla_{yy}^2g(x,y^*)\big]^{-1}\nabla_y f(x,y^*).
\end{align}
Note that the hypergradient in \cref{def:Phi} requires computing the exact solution $y^*$ and the expensive Hessian inverse $[\nabla_{yy}^2g(x,y^*)]^{-1}$. To approximate this hypergradient efficiently, we define the following (stochastic) hypergradient surrogates as 
\begin{align}\label{def:fffform}
    \bar{\nabla}f(x, y, v) =& \nabla_xf(x,y) - \nabla_{xy}^2g(x,y)v, \nonumber
    \\\bar{\nabla}f(x, y, v;\xi) =&\nabla_x f(x, y; \xi) 
- \nabla_{xy}^2 g(x, y; \xi)v,
\end{align}
where $v\in\mathbb{R}^q$ is an auxiliary vector to approximate the Hessian-inverse-vector product in \cref{def:Phi}, and $\bar{\nabla}f(x, y, v;\xi)$ can be regarded as a stochastic version of $ \bar{\nabla}f(x, y, v)$. Based on \cref{def:fffform}, one needs to find an efficient estimate $y$ of $y^*$, e.g., via an iterative optimization procedure, as well as a feasible estimate $v$ of the solution $v^* = [\nabla_{yy}^2g(x,y)]^{-1}\nabla_y f(x,y)$ of a linear system (LS) (equivalently quadratic programming) whose generic loss function is given by    
\begin{align}\label{def:R}
   \text{(Linear system loss:)}\quad  R(x, y, v) = \frac{1}{2}v^T\nabla_{yy}^2g(x,y)v - v^T\nabla_y f(x,y),
\end{align}
where the gradient of  $R(x,y,v)$ w.r.t.~$v$  is given by 
\begin{align}\label{def:gradR} 
   \nabla_v R(x, y, v) &= \nabla_{yy}^2g(x,y)v - \nabla_y f(x,y).
\end{align}
Similarly to \cref{def:fffform}, we also define $\nabla_v R(x, y, v;\psi) = \nabla_{yy}^2g(x,y;\psi)v - \nabla_y f(x,y;\psi)$ over any sample $\psi$ as a stochastic version of $ \nabla_v R(x, y, v)$ in \cref{def:gradR}. It can be seen from \cref{def:fffform}, \cref{def:R} and \cref{def:gradR} that the updates on the LS system involve the Hessian- and Jacobian-vector products, which can be computationally intractable in the high-dimensional case.   
In the next section, we propose a novel stochastic Hessian/Jacobian-free bilevel algorithm.

\begin{algorithm}[t]
	\caption{Hessian/Jacobian-free Bilevel Optimizer via Projection-aided Finite-difference Estimation}   
	\small
	\label{alg:main_free}
	\begin{algorithmic}[1]
		\STATE {\bfseries Input:}  $\left\{\alpha_t, \beta_t,\lambda_t\right\}_{t=0}^{T-1}$ and $r_v$. 
		\STATE {\bfseries Initialize:}
		\FOR{$t=0,1,2,...,T-1$}
		\STATE{Compute the gradient estimator $h_t^g$ by~\cref{def:htg} and update $y_{t+1} = y_t-\beta_t h_t^g$.}
		\STATE{Compute the gradient estimator $h_t^R$ by~\cref{def:htR_free} and update $w_{t+1} = v_t-\lambda_t \widetilde{h}_t^R$.}
		\STATE{Set $v_{t+1} = \begin{cases} w_{t+1},& \|w_{t+1}\|\leq r_v;\\ \frac{r_v w_{t+1}}{\|w_{t+1}\|},& \|w_{t+1}\|> r_v. \end{cases}$}
		\STATE{Compute the gradient estimator $h_t^f$ by~\cref{def:htf_free} and update $x_{t+1} = x_t-\alpha_t \widetilde{h}_t^f$.}
		\ENDFOR
	\end{algorithmic}
\end{algorithm}

\subsection{Hessian/Jacobian-free Bilevel Optimizer via Projection-aided Finite-difference Estimation}\label{sec:hfbo}
As shown in \Cref{alg:main_free}, we propose a fully single-loop stochastic Hessian/Jacobian-free bilevel optimizer named FdeHBO via projection-aided finite-difference estimation. 
It can be seen that 
FdeHBO first minimizes the lower-level objective function $g(x,y)$ w.r.t.~$y$ by running a single-step momentum-based update as 
    $y_{t+1} = y_t-\beta_t h_t^g$,
where $\beta_t$ is the stepsize and $h_t^g$ is the momentum-based gradient estimator that takes the form of  
\begin{align}\label{def:htg}
h_t^g = \eta_t^g&\nabla_y g(x_t,y_t;\zeta_t) 
+ (1-\eta_t^g)\big(h_{t-1}^g + \nabla_y g(x_t,y_t;\zeta_t) - \nabla_y g(x_{t-1},y_{t-1};\zeta_t)\big)
\end{align}
where $\eta^g_t \in [0,1]$ is a tuning parameter. 
The next key step is to deal with the LS problem via solving the quadratic problem \cref{def:R}  as 
$w_{t+1} = v_t-\lambda_t \widetilde{h}_t^R$, with the momentum-based gradient  $\widetilde{h}_t^R$ given by 
\begin{align}\label{def:htR_free}
\widetilde{h}_t^R =& \eta_t^R\widetilde{\nabla}_v R(x_t,y_t, v_t,\delta_{\epsilon};\psi_t) + (1-\eta_t^R)\big(h_{t-1}^R+\widetilde{\nabla}_v R(x_t,y_t, v_t,\delta_{\epsilon};\psi_t) \nonumber
\\&- \widetilde{\nabla}_v R(x_{t-1},y_{t-1},v_{t-1},\delta_{\epsilon};\psi_t)\big),
\end{align}
 where $\widetilde {\nabla}_vR $ is a Hessian-free version of the LS gradient ${\nabla}_vR$ in \cref{def:gradR},  given by 
 \begin{align}\label{def:gradR_free}
   \text{(First-order LS gradient:)} \;\;\widetilde{\nabla}_v &R(x_t, y_t, v_t, \delta_{\epsilon};\psi_t) = \widetilde{H}(x_t, y_t, v_t, \delta_{\epsilon};\psi_t) - \nabla_y f(x_t,y_t;\psi_t).
\end{align}
Note that in the above \cref{def:gradR_free},  $\widetilde{H}(x_t, y_t, v_t, \delta_{\epsilon};\psi_t)$ is the finite-difference estimation of the Hessian-vector product   $\nabla_{yy}^2g(x_t,y_t;\psi_t)v_t$, which takes the form of 
\begin{align}\label{def:H}
\widetilde{H}(x_t, y_t, v_t, \delta_{\epsilon};\psi_t) = \frac{\nabla_y g(x_t, y_t+\delta_{\epsilon} v_t;\psi_t) - \nabla_y g(x_t, y_t-\delta_{\epsilon} v_t;\psi_t)}{2\delta_{\epsilon}},
\end{align}
where $\delta_\epsilon>0$ is a small constant. Note that in \cref{def:H}, if the iterative $v_t$ is unbounded, the approximation error between $\widetilde H$ and  $\nabla_{yy}^2g(x_t,y_t;\psi_t)v_t$ can be uncontrollable as well. 
We further prove \cref{lm:boundofeJH} in \cref{appd:auxiLemma} that the bound of this gap relies on $\|v_t\|$ and $\delta$ but it is independent of the dimension of $y_t$.
To this end, 
after obtaining $w_{t+1}$, our key step in line 6 introduces an auxiliary projection on a ball (which can be generalized to any convex and bounded domain) with a radius of $r_v$ as  
\begin{align*}
  \text{(Auxiliary projection)} \quad \ v_{t+1} = \begin{cases} w_{t+1},& \|w_{t+1}\|\leq r_v;\\ \frac{r_v w_{t+1}}{\|w_{t+1}\|},& \|w_{t+1}\|> r_v. \end{cases}
\end{align*}
This auxiliary projection guarantees the boundedness of $v_t,t=0,...,T-1$, which serves {\bf three} important purposes. First, it ensures the smoothness of the LS loss function $R(x,y,v)$ in \cref{def:gradR}  w.r.t.~all $x,y$ and $v$, which is crucial in the convergence analysis of the momentum-based updates. Second, the boundedness of $v_t$ also ensures that the estimation variance of the stochastic LS gradient $\nabla_v R(x_t, y_t, v_t;\psi_t)$ does not explode. Third, it guarantees the error of the finite-difference Hessian-vector approximation to be sufficiently small with proper $\delta_\epsilon$. 
We will show later that under a proper choice of the radius $r_v$, this auxiliary projection provides a better per-step progress, and the proposed algorithm achieves a stronger convergence performance. 
Finally, for the upper-level problem, the momentum-based hypergradient estimate $\widetilde{h}_t^f$ is designed as 
\begin{align}\label{def:htf_free}
\widetilde{h}_t^f = &\eta_t^f\widetilde{\nabla} f(x_t,y_t,v_t,\delta_{\epsilon};\bar{\xi}_t) + (1-\eta_t^f)\big(h_{t-1}^f+\widetilde{\nabla} f(x_t,y_t,v_t,\delta_{\epsilon};\bar{\xi}_t)  \nonumber
\\&- \widetilde{\nabla} f(x_{t-1},y_{t-1},v_{t-1},\delta_{\epsilon};\bar{\xi}_t)\big),
\end{align}
where  $\widetilde\nabla f(x, y, v, \delta_{\epsilon};\bar \xi_t)$ is the fully first-order hypergradient estimate evaluated at two consecutive  iterates  $(x_t,y_t,v_t)$ and $(x_{t-1},y_{t-1},v_{t-1})$ is given by 
\begin{align*}
    \widetilde{\nabla}f(x, y, v, \delta_{\epsilon};\bar\xi_t) =&\nabla_x f(x, y; \bar \xi_t) 
- \widetilde{J}(x, y, v, \delta_{\epsilon}; \bar \xi_t),
\end{align*}
and $\widetilde{J}(x, y, v, \delta_{\epsilon};\bar\xi_t) $ is the finite-difference Jacobian-vector  approximation given by 
\begin{align}\label{def:J}
\widetilde{J}(x, y, v, \delta_{\epsilon};\bar\xi_t):=\frac{\nabla_x g(x, y+\delta_{\epsilon} v;\bar\xi_t) - \nabla_x g(x, y-\delta_{\epsilon} v;\bar\xi_t)}{2\delta_{\epsilon}}.
\end{align}
Note that $\widetilde\nabla_v R$ and $\widetilde \nabla f$ are {\bf biased} estimators of the gradients $\nabla_v R$ and $\bar \nabla f$, which further complicates the convergence analysis on the momentum-based updates because the conventional analysis on the recursive momentum requires the unbiased gradient estimation to ensure the variance reduction effect. 
By controlling the perturbation $\delta_\epsilon$ properly, we will show that FdeHBO can achieve an $\mathcal{O}(\epsilon^{-1.5})$ convergence and complexity performance without any second-order derivative computation.

\subsection{Extension to Small-Dimensional Case}\label{sec:FNBO}
As a byproduct of our proposed FdeHBO, we further propose a fully single-loop momentum-based bilevel optimizer (FMBO), which is more suitable in the small-dimensional case without finite-difference approximation. As shown in \Cref{alg:main}, FMBO first takes the same lower-level updates on $y_t$ as in \cref{def:htg}. Then, it solves the LS problem  as 
 $w_{t+1} = v_t-\lambda_t h_t^R$,
where the momentum-based gradient estimator is given by 
\begin{align}\label{def:htR}
h_t^R = & \eta_t^R\nabla_v R(x_t,y_t, v_t;\psi_t) + (1-\eta_t^R)\big(h_{t-1}^g+ \nabla_v R(x_t,y_t, v_t;\psi_t) \nonumber
\\&- \nabla_v R(x_{t-1},y_{t-1},v_{t-1};\psi_t)\big),
\end{align}
where differently from FdeHBO, we here use the precise gradient $\nabla_v R$ without finite-difference approximation. Similarly to FdeHBO, we add an auxiliary projection on the $v_t$ updates to ensure the LS smoothness and bounded variance. Finally, for the upper-level problem, we optimize $x_t$ based on a momentum-based update as $x_{t+1} = x_t-\alpha_t h_t^f$ with the hypergradient estimator 
\begin{align}\label{def:htf}
h_t^f =& \eta_t^f\bar{\nabla} f(x_t,y_t,v_t;\bar{\xi}_t) + (1-\eta_t^f)(h_{t-1}^f+ \bar{\nabla} f(x_t,y_t,v_t;\bar{\xi}_t) - \bar{\nabla} f(x_{t-1},y_{t-1},v_{t-1};\bar{\xi}_t))
\end{align}
where $\eta_t^f\in [0,1]$ is a tuning parameter. Similarly, we 
directly use the hypergradient estimate in \cref{def:fffform} without the finite-difference estimation. We note that compared to existing momentum-based algorithms~\cite{yang2021provably,khanduri2021near} that contains $\mathcal{O}(\log\frac{1}{\epsilon})$ steps in solving the LS problem, FMBO takes the fully single-loop structure with a single-step momentum-based acceleration on the LS updates. 

\begin{algorithm}[t]
	\caption{Fully Single-loop Momentum-based Bilevel Optimizer (FMBO)}   
	\small
	\label{alg:main}
	\begin{algorithmic}[1]
		\STATE {\bfseries Input:}  $\left\{\alpha_t, \beta_t,\lambda_t\right\}_{t=0}^{T-1}$, and $r_v$.
		\STATE {\bfseries Initialize:}
		\FOR{$t=0,1,2,...,T-1$}
		\STATE{Compute the gradient estimator $h_t^g$ by~\cref{def:htg} and update  $y_{t+1} = y_t-\beta_t h_t^g$.}
		\STATE{Compute the gradient estimator $h_t^R$ by~\cref{def:htR} and update  $w_{t+1} = v_t-\lambda_t h_t^R$.}
		\STATE{Set $v_{t+1} = \begin{cases} w_{t+1},& \|w_{t+1}\|\leq r_v;\\ \frac{r_v w_{t+1}}{\|w_{t+1}\|},& \|w_{t+1}\|> r_v. \end{cases}$}
		\STATE{Compute the gradient estimator $h_t^f$ by~\cref{def:htf} and update  $x_{t+1} = x_t-\alpha_t h_t^f$.}
		\ENDFOR
	\end{algorithmic}
\end{algorithm}

\section{Main Results}

\subsection{Assumptions and Definitions}
We make the following standard assumptions for the upper- and lower-level objective functions, as also adopted by \cite{ji2021bilevel,chen2021single,khanduri2021near}. The following assumption imposes the Lipschitz 
condition on the upper-level function~$f(x,y)$. 
\begin{assumption}
\label{as:ulf}
For any $x \in \mathbb{R}^{d_{x}}$ and $y \in \mathbb{R}^{d_{y}}$, there exist positive constants $L_{f_x}$, $L_{f_y}$, $C_{f_x}$ and $C_{f_y}$ such that 
$\nabla_x f(x,y)$ and $\nabla_y f(x,y)$ are   $L_{f_x}$- and $L_{f_y}$-Lipschitz continuous w.r.t. $(x,y)$, and  
$\|\nabla_x f(x,y)\|^2\leq C_{f_x}$, $\|\nabla_y f(x,y)\|^2\leq C_{f_y}$. 
\end{assumption}\noindent
The following assumption imposes the Lipschitz condition on the lower-level function~$g(x,y)$.
\begin{assumption}
\label{as:llf}
For any $x \in \mathbb{R}^{d_{x}}$ and $y \in \mathbb{R}^{d_{y}}$, 
there exist positive constants 
$\mu_g$, $L_g$, $L_{g_{xy}}$, $L_{g_{yy}}$, $C_{g_{xy}}, C_{g_{y y}}$ such that 
 \begin{list}{$\bullet$}{\topsep=0.1ex \leftmargin=0.2in \rightmargin=0.in \itemsep =0.01in}
\item  Function $g(x, y)$ is twice continuously differentiable;
\item Function $g(x, \cdot)$ is $\mu_g$-strongly-convex;
\item The derivatives $\nabla_y g(x,y)$,~$\nabla^2_{xy}g(x,y)$ and $\nabla^2_{yy}g(x,y)$ are $L_g$-, $L_{g_{xy}}$- and $L_{g_{yy}}$-Lipschitz continuous w.r.t.~$(x,y)$;
\item $\|\nabla^2_{xy}g(x,y)\|^2 \leq C_{g_{xy}}$ and $\|\nabla^2_{yy}g(x,y)\|^2 \leq C_{g_{yy}}$.
\end{list}
\end{assumption}
\noindent
The following assumption is adopted for the stochastic functions $f(x,y;\xi)$ and $g(x,y;\zeta)$.
\begin{assumption}
\label{as:sf}
Assumptions~\ref{as:ulf} and ~\ref{as:llf} hold for $f(x, y; \xi)$ and $g(x, y; \zeta)$  for $\forall\,\xi$ and $\zeta$. Moreover, we assume that  there exist positive constants $ \sigma_{f_x}$, $\sigma_{f_y}$, $\sigma_{g}$, $\sigma_{g_{xy}}$ and $\sigma_{g_{yy}}$ such that 
\begin{align*}
   & \mathbb{E}\left[\|\nabla_x f(x,y)-\nabla_x f(x,y;\xi)\|^2\right] \leq \sigma^2_{f_x}, \quad 
    \mathbb{E}\left[\|\nabla_y f(x,y)-\nabla_y f(x,y;\xi)\|^2\right] \leq \sigma^2_{f_y}, \\
   & \mathbb{E}\left[\|\nabla_y g(x,y) - \nabla_y g(x,y;\zeta)\|^2\right] \leq \sigma_g^2, \quad\;\;\; 
    \mathbb{E}\left[\|\nabla_{xy}^2 g(x,y)-\nabla_{xy}^2 g(x,y;\xi)\|^2\right] \leq \sigma^2_{g_{xy}}, \\ 
    &\mathbb{E}\left[\|\nabla_{yy}^2 g(x,y)-\nabla_{yy}^2 g(x,y;\xi)\|^2\right] \leq \sigma^2_{g_{yy}}. 
\end{align*}
\end{assumption}

\begin{definition}
We say $\bar x$ is an $\epsilon$-accurate stationary point of a function $\Phi(x)$ if $\mathbb{E}\|\nabla \Phi(\bar x)\|^2\leq \epsilon$, where $\bar x$ is the output of an optimization algorithm. 
\end{definition}

\subsection{Convergence and Complexity Analysis of FdeHBO}\label{sec:analysisHF}
We further provide the convergence analysis for the proposed Hessian/Jacobian-free FdeHBO algorithm. We first characterize several estimation properties of FdeHBO. Let $e^f_t :=\widetilde{h}^f_t - \nabla f(x_t,y_t,v_t)-\Delta(x_t,y_t,v_t)$ denote  the hypergradient estimation error. 
\begin{proposition}\label{prop:HFetf} 
Under Assumption~\ref{as:sf}, the iterates of the outer problem by~\Cref{alg:main_free} satisfy
\begin{align*}
    \mathbb{E}\|e^f_{t+1}\|^2 
    \leq& \Big[(1-\eta^f_{t+1})^2+4L_{g_{xy}}r_v^2\delta_{\epsilon}\Big]\mathbb{E}\|e_t^f\|^2  + 4(\eta^f_{t+1})^2\sigma_f^2+ \big(4L_{g_{xy}}r_v^2\delta_{\epsilon} + 16L^2_{g_{xy}}r_v^4\delta_{\epsilon}^2\big) \\
    &+6(1-\eta^f_{t+1})^2\Big[L_F^2\alpha_t^2\mathbb{E}\|\widetilde{h}_t^f\|^2 + 2L_F^2\beta_t^2\big(\mathbb{E}\|e_t^g\|^2 + \|\nabla_y g(x_t, y_t)\|^2\big)
    \\&\qquad \qquad \qquad \ \ \ + 2C_{g_{xy}}\lambda^2_t\big(\mathbb{E}\|e_t^R\|^2 + L_g^2\mathbb{E}\|v_t - v_t^*\|^2\big)\Big],  
\end{align*}
for all $t \in \{0, . . . , T-1\}$ with $L_F^2 = 2\big(L_{f_x}^2 + L^2_{g_{xy}}r_v^2\big)$.
\end{proposition}
\noindent
The hypergradient estimator error {\small$\mathcal{O}(\mathbb{E}\|e_{t+1}^f\|^2)$} contains three main components. The first term {\small$[(1-\eta^f_{t+1})^2+4L_{g_{xy}}r_v^2\delta_{\epsilon}]\mathbb{E}\|e_t^f\|^2$}  indicates the per-iteration improvement induced by the momentum-based update, the error term $\alpha_t^2\mathbb{E}\|h_t^f\|^2$ is caused by the $x_t$ updates, the error term 
{\small$\mathcal{O}(\beta_t^2\mathbb{E}(\|e_t^g\|^2+\|\nabla_yg(x_t, y_t)\|^2))$}  
is caused by solving the lower-level problem, and the 
new error term {\small$\mathcal{O}(\lambda^2_t\mathbb{E}(\|e_t^R\|^2 + L_g^2\|v_t - v_t^*\|^2))$} is induced by the one-step momentum update on the LS problem, which does not exist in previous momentum-based bilevel methods~\cite{yang2021provably,khanduri2021near,guo2021randomized} that solve the LS problem to a high accuracy. 
Also note that the errors $4L_{g_{xy}}r_v^2\delta_{\epsilon}\mathbb{E}\|e_t^f\|^2$ and  $4L_{g_{xy}}r_v^2\delta_{\epsilon} + 16L^2_{g_{xy}}r_v^4\delta_{\epsilon}^2$ are caused by the finite-difference approximation error. Fortunately, by choosing the perturbation level $\delta_\epsilon$ in these two terms to be properly small, it can guarantee the descent factor $(1-\eta^f_{t+1})^2+4L_{g_{xy}}r_v^2\delta_{\epsilon}$ to be at an order of $(1-\mathcal{O}(\eta_{t+1}^f))^2$, and hence the momentum-based variance reduction effect is still applied. 
\begin{proposition}\label{prop:hfetR}
    For $\forall\,\psi$, define $e_t^R := \widetilde{h}_t^R - \nabla_v R(x_t, y_t, v_t)$. Under Assumptions~\ref{as:ulf},~\ref{as:llf},~\ref{as:sf}, we have 
\begin{align*}
\mathbb{E}\|e^R_{t+1}\|^2 \leq& \big[(1-\eta^R_{t+1})^2(1+96L_g^4\lambda^2_t) + 4L_{g_{yy}}r^2_v\delta_{\epsilon}\big] \mathbb{E}\|e_t^R\|^2 + \big(4L_{g_{yy}}r^2_v\delta_{\epsilon} 
+ 8L^2_{g_{yy}}r_v^4\delta_{\epsilon}^2\big)
\\&+8(\eta_{t+1}^R)^2(\sigma^2_{g_{yy}}r_v^2+ \sigma^2_{f_y}) + 96(1-\eta_{t+1}^R)^2L^2_g\lambda^2_t \left(\mathbb{E}\|e_t^R\|^2 + L_g^2\mathbb{E}\|v_t-v_t^*\|^2\right)
\\ 
& + 96(1-\eta_{t+1}^R)^2\big(L^2_{g_{yy}}r^2_v + L^2_{f_y}\big) \Big[\alpha_t^2\mathbb{E}\|\widetilde{h}_{t}^f\|^2 + 2\beta_t^2(\mathbb{E}\|e_t^g\|^2 + \mathbb{E}\|\nabla_y g(x_t,y_t)\|^2)\Big] 
\end{align*}
for all $t\in \{0,1,...,T-1\}$. 
\end{proposition}
\noindent
As shown in \Cref{prop:hfetR}, the LS gradient estimation error $e_{t+1}^R$ contains an 
iteratively improved error component $\big[(1-\eta^R_{t+1})^2(1+96L_g^4\lambda^2_t) + 4L_{g_{yy}}r^2_v\delta_{\epsilon}\big] \mathbb{E}\|e_t^R\|^2 $ for the stepsize $\lambda_t$ and the approximation factor $\delta_\epsilon$ sufficiently small, a finite-difference approximation error $\mathcal{O}(\delta_\epsilon)$  as well as an approximation error $\mathcal{O}(\lambda^2_t\mathbb{E}\|v_t-v_t^*\|^2)$ for solving the LS problem. The next step is to upper-bound  $\mathbb{E}\|v_{t}- v_{t}^*\|^2$.

\begin{proposition}\label{prop:hessian_free_v}
Under the Assumption~\ref{as:ulf},~\ref{as:llf}, the iterates of the LS problem  by~\Cref{alg:main_free} satisfy
\begin{align*}
\mathbb{E}\|v_{t+1} &- v_{t+1}^*\|^2  \\
&\leq (1+\gamma_t')\Big(1+\delta_t'\Big)\Big[\Big(1-2\lambda_t\frac{(L_g+L_g^3)\mu_g}{\mu_g + L_g} + \lambda_t^2L_g^2\Big)\mathbb{E}\|v_t - v^*_t\|^2\Big] \\
&\quad \ + (1+\gamma_t')\Big(1+\frac{1}{\delta_t'}\Big)\lambda_t^2\mathbb{E}\|e_t^R\|^2 \\
&\quad \ + (1+\frac{1}{\gamma_t'})\Big( \frac{2L^2_{f_y}}{\mu^2_g} + \frac{2C^2_{f_y}L^2_{g_{yy}}}{\mu^4_g} \Big)\left[ \alpha_t^2\mathbb{E}\|\widetilde{h}_t^f\|^2 + \beta_t^2\left(2\mathbb{E}\|e_t^g\|^2 + 2\mathbb{E}\|\nabla_y g(x_t, y_t)\|^2 \right) \right].
\end{align*}
for all $t \in \{0, . . . , T-1\}$ with some $\gamma_t'>0$ and $ \delta_t' > 0$.
\end{proposition}
\noindent
Based on the above important properties, we now provide the general convergence theorem for FdeHBO. 
\begin{theorem}\label{th:hessianfree}
Suppose Assumptions~\ref{as:ulf},~\ref{as:llf}, ~\ref{as:sf} are satisfied. Choose $r_v \geq \frac{C_{f_y}}{\mu_g}$ and 
set
\begin{align*}
\alpha_t = \frac{1}{(w+t)^{1/3}},\quad \beta_t = c_\beta\alpha_t, &\quad \lambda_t = c_\lambda\alpha_t, \quad
\eta_t^f = c_{\eta_f}\alpha_t^2,\quad \eta_t^R = c_{\eta_R}\alpha_t^2,\quad \eta_t^g = c_{\eta_g}\alpha_t^2,
\end{align*}
and $\delta_{\epsilon} \leq \frac{\min\{c_{\eta_f},c_{\eta_R}\}}{8(L_{g_{xy}}r_v^2(w+T-1)^{2/3})}$, 
where the constants $w$, $c_\beta,c_\lambda,c_{\eta_f},c_{\eta_R}$ and  $c_{\eta_g}$ are defined in~\cref{def:parameters_free} in the appendix. Then, the iterates generated by~\Cref{alg:main_free} satisfy
\begin{align*}
    \mathbb{E}\|\nabla \Phi\big(x_a(T)\big)\|^2 \leq \widetilde{\mathcal{O}}\bigg(& \frac{\Phi(x_0) - \Phi^*}{T^{2/3}}
    + \frac{\|y_0 - y^*(x_0)\|^2}{T^{2/3}}
    + \frac{\|v_0 - v^*(x_0, y_0)\|^2}{T^{2/3}} \\& + \frac{1}{T^{2/3}} + \frac{\sigma_f^2}{T^{2/3}}
    +  \frac{\sigma_g^2}{T^{2/3}}
    +  \frac{\sigma_R^2}{T^{2/3}}\bigg) .
\end{align*}
\end{theorem}

\begin{corollary}\label{co:hftribo}
  Under the same setting of \Cref{th:hessianfree}, FdeHBO requires $\mathcal{\widetilde O}(\epsilon^{-1.5})$ samples and gradient evaluations, respectively, to achieve an $\epsilon$-accurate stationary point.  
\end{corollary}
\noindent
It can be seen from \Cref{co:hftribo} that the proposed FdeHBO achieves an $\mathcal{\widetilde O} (\epsilon^{-1.5})$ sample complexity without any second-order derivative computation. As far as we know, this is the first Hessian/Jacobian-free stochastic bilevel optimizer with an $\mathcal{\widetilde O}(\epsilon^{-1.5})$ sample complexity. 

\subsection{Convergence and Complexity Analysis of FMBO}
In this section, we analyze the convergence and complexity of the simplified FMBO method.
\begin{theorem}\label{th:main}
Suppose Assumptions~\ref{as:ulf},~\ref{as:llf} and ~\ref{as:sf} are satisfied. 
Choose $r_v \geq \frac{C_{f_y}}{\mu_g}$ and 
set parameters 
\begin{align*}
\alpha_t = \frac{1}{(w+t)^{1/3}},\quad \beta_t &= c_\beta\alpha_t, \quad \lambda_t = c_\lambda\alpha_t, \\
\eta_t^f = c_{\eta_f}\alpha_t^2,\quad \eta_t^R &= c_{\eta_R}\alpha_t^2,\quad \eta_t^g = c_{\eta_g}\alpha_t^2
\end{align*}
where $w$, $c_\beta,c_\lambda,c_{\eta_f},c_{\eta_R}$ and  $c_{\eta_g}$ are defined in~\cref{def:parameters} in the appendix. The iterates generated by~\Cref{alg:main} satisfy
\begin{align*}
    \mathbb{E}\|\nabla \Phi(x_a(T))\|^2 \leq &\widetilde{\mathcal{O}}\Bigg( \frac{\Phi(x_0) - \Phi^*}{T^{2/3}}
    + \frac{\|y_0 - y^*(x_0)\|^2}{T^{2/3}} 
    \\&\quad \ \ + \frac{\|v_0 - v^*(x_0, y_0)\|^2}{T^{2/3}}+ \frac{\sigma_f^2}{T^{2/3}} + \frac{\sigma_g^2}{T^{2/3}} 
    +\frac{\sigma_R^2}{T^{2/3}} \Bigg).
\end{align*} 
\end{theorem}
\noindent
\Cref{th:main} shows that the proposed fully single-loop FMBO achieves a convergence rate of $\frac{1}{T^{2/3}}$, which further yields the following complexity result.   
\begin{corollary}\label{Co:complexitytribo}
   Under the same setting of \Cref{th:main}, FMBO requires totally $\mathcal{\widetilde O}(\epsilon^{-1.5})$ data samples, gradient and matrix-vector evaluations, respectively, to achieve an $\epsilon$-accurate stationary point. 
\end{corollary}
\noindent
\Cref{Co:complexitytribo} shows that FMBO requires a total number $\mathcal{\widetilde O}(\epsilon^{-1.5})$ of data samples, which matches the best sample complexity in \cite{khanduri2021near,yang2021provably,huang2021biadam}. More importantly, each iteration of FMBO contains only one Hessian-vector computation due to the simple fully single-loop implementation, whereas other momentum-based approaches require $\mathcal{O}(\log\frac{1}{\epsilon})$ Hessian-vector computations in a nested manner per iteration. Also, note that FMBO is the first fully single-loop bilevel optimizer that achieves the $\mathcal{\widetilde O}(\epsilon^{-1.5})$ sample complexity. 

\section{Experiments}
In this section, we test the performance of the proposed FdeHBO and FMBO on two applications: hyper-representation and data hyper-cleaning, respectively. 

\subsection{Hyper-representation on MNIST Dataset}
We now compare the performance of our Hessian/Jacobian-free FdeHBO with the relevant 
Hessian/Jacobian-free methods PZOBO-S~\cite{sow2022on}, F$^2$SA \cite{kwon2023fully} and F$^3$SA \cite{kwon2023fully}. We perform the hyper-representation with the $7$-layer LeNet network \cite{lecun1998gradient}, which aims to solve the following bilevel problem. 
\begin{align*}
&\min_{\lambda}L_\nu(\lambda) := \frac{1}{|S_\nu|}\sum_{(x_i, y_i)\in S_\nu}L_{CE}(w^*(\lambda)f(\lambda; x_i), y_i)\\
&s.t.\;\;w^*(\lambda)=\argmin_{w}L_{in}(\lambda, w), \;\;\; L_{in}(\lambda, w):=\frac{1}{|S_\tau|}\sum_{(\tau,y_i)\in 
S_\tau}L_{CE}(wf(\lambda, x_i),y_i),
\end{align*}
where $L_{CE}$ denotes the cross-entropy loss, $S_{\nu}$ and $S_{\tau}$ denote the training data and validation data, and $f(\lambda; x_i)$ denotes the features extracted from the data $x_i$. More details of the experimental setups are specified in~\Cref{suply-HR}. 

As shown in~\Cref{fig:hr}, our FdeHBO converges much faster and more stably than PZOBO-S, F$^2$SA and F$^3$SA, while achieving a higher training accuracy. This is consistent with our theoretical results, and validates the momentum-based approaches in reducing the variance during the entire training.  

\begin{figure}[t]
\centering    
 \vspace{-1cm}
\includegraphics[width=80mm]{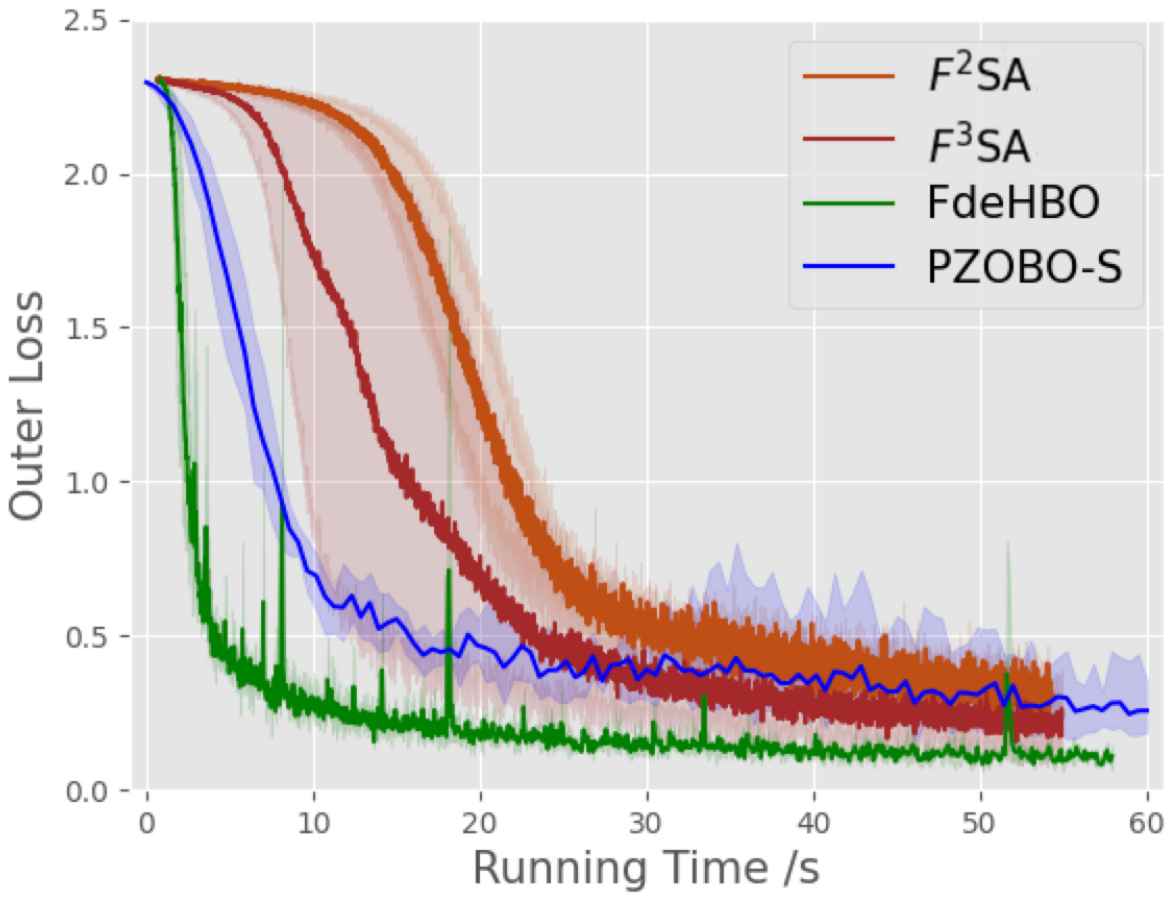}
\hspace{-0.8cm}
\includegraphics[width=80mm]{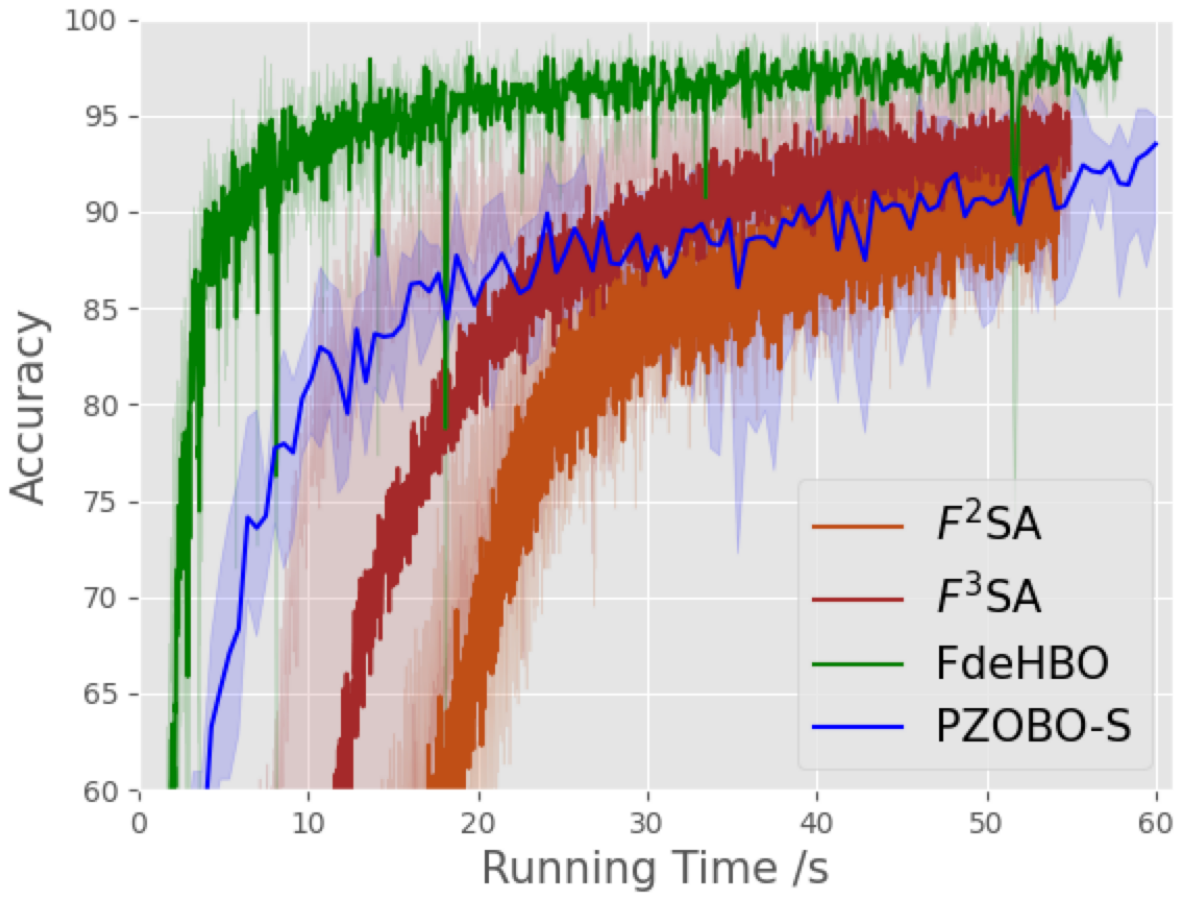}
 \vspace{-0.5cm}
\caption{Comparison on hyper-representation with the LeNet neural network. Left plot: outer loss v.s.~running time; right plot: accuracy v.s.~running time.}\label{fig:hr}
\end{figure}

\subsection{Hyper-cleaning on MNIST Dataset}
We compare the performance of our FMBO to various bilevel algorithms including AID-FP~\cite{grazzi2020iteration}, reverse\cite{franceschi2017forward}, SUSTAIN~\cite{khanduri2021near}, MRBO and VRBO~\cite{yang2021provably}, BSA~\cite{ghadimi2018approximation}, stocBiO~\cite{ji2021bilevel}, FSLA~\cite{li2022fully} and SOBA~\cite{dagreou2022framework}, on a low-dimensional data hyper-cleaning problem with a linear classifier on MNIST dataset, which takes the following formulation.
\begin{align}\label{data_cleaning}
    &\min_{\lambda}L_{\nu}(\lambda, w^*) = \frac{1}{|S_{\nu}|}\sum_{(x_i, y_i)\in S_{\nu}} L_{CE}((w^*)^Tx_i, y_i) \nonumber \\
    & s.t. \quad w^* = \arg\min_{w}L(\lambda, w) := \frac{1}{|S_{\tau}|}\sum_{(x_i, y_i)\in S_{\tau}}  \sigma(\lambda_i)L_{CE}(w^Tx_i, y_i) + C\|w\|^2,
\end{align}
where $L_{CE}$ denotes the cross-entropy loss, $S_{\nu}$ and $S_{\tau}$ denote the training data and validation data, whose sizes are set to 20000 and 5000, respectively, $\lambda = \{\lambda_i\}_{i \in S_{\tau}}$ and $C$ are the regularization parameters, and $\sigma(\cdot)$ is the sigmoid function.  AmIGO~\cite{arbel2022amortized} is not included in the figures because it performs similarly to stocBiO. 
The experimental details can be found in~\Cref{app:hc}.

As shown in \Cref{fig1:a}, FMBO, stocBiO and AID-FP converge much faster and more stable than other algorithms. Compared to stocBiO and AID-FP, FMBO achieves a lower training loss. This demonstrates the effectiveness of momentum-based variance reduction in finding more accurate iterates. It can be seen from \Cref{fig1:b} that FMBO converges faster than existing fully single-loop FSLA and SOBA algorithms with a lower training loss.


\begin{figure*}[t]
	\subfigure[noise p = 0.1]{\label{fig1:a}\includegraphics[width=55mm]{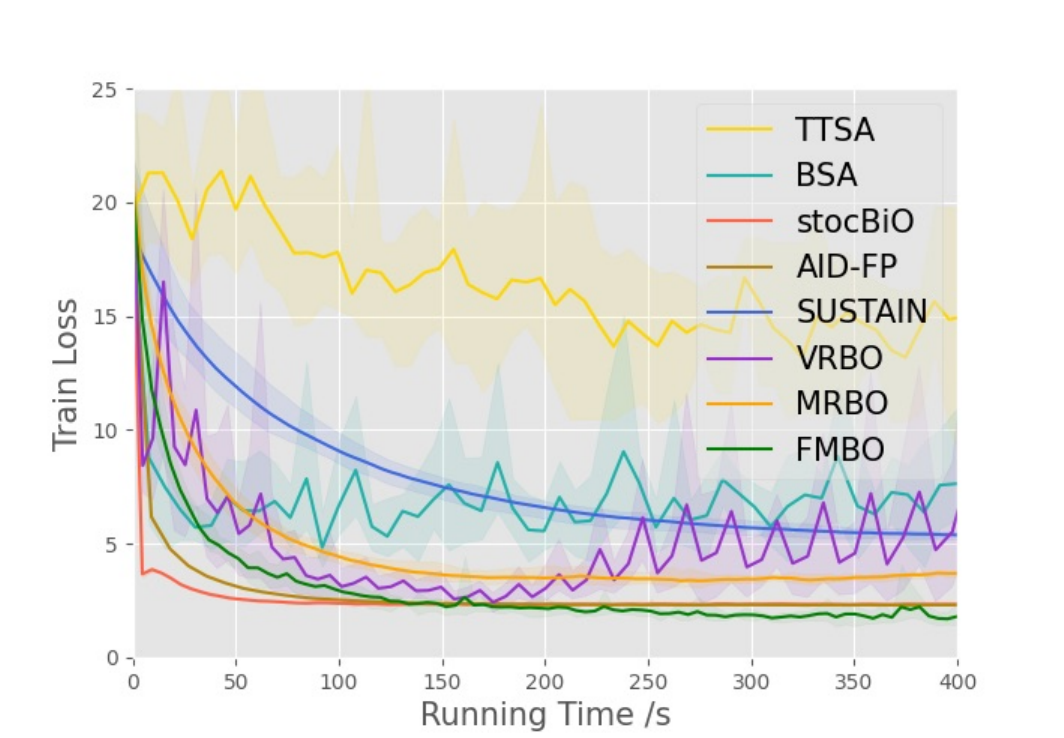}\includegraphics[width=55mm]{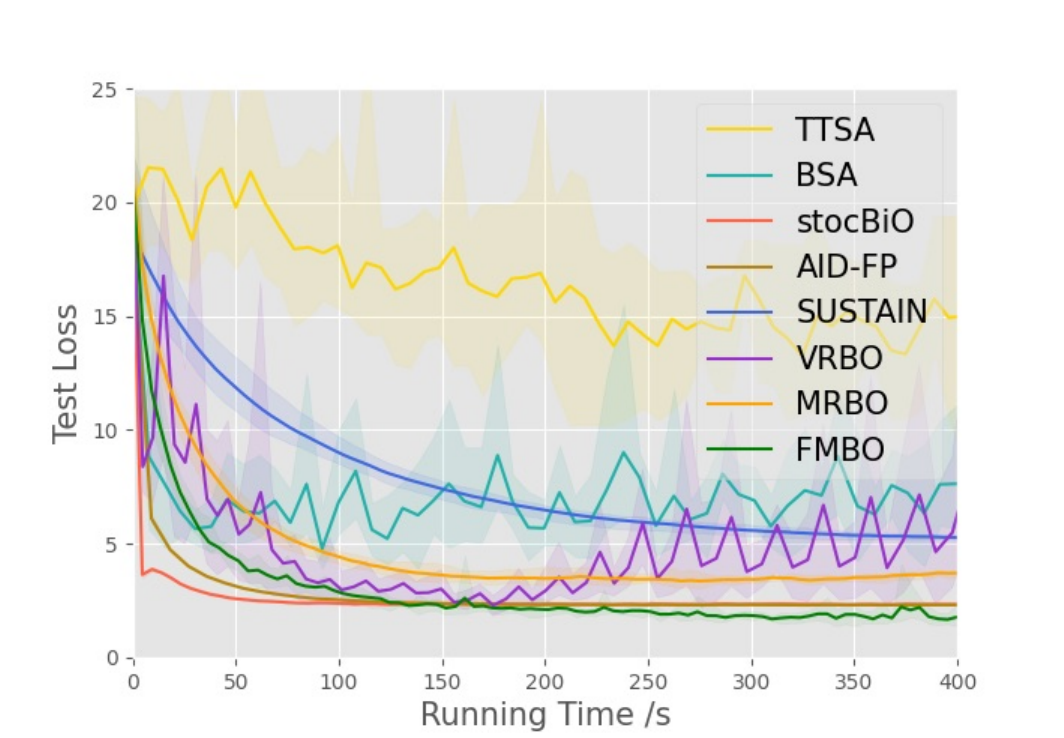}}
	\subfigure[Single-loop methods]{\label{fig1:b}\includegraphics[width=55mm]{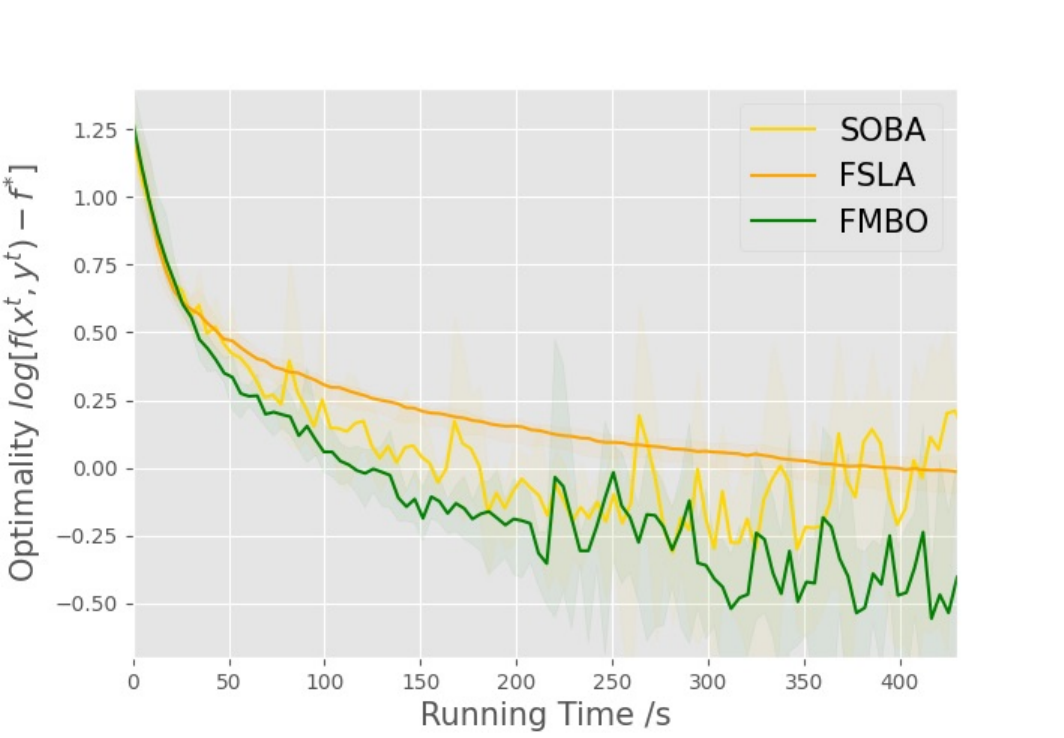}}
	\caption{(a) Comparison of different algorithms on data hyper-cleaning with noise $p=0.1$. Left plot: test loss v.s.~running time; right plot: train loss v.s.~running time. (b) Comparison among different single-loop algorithms: training loss v.s.~running time.}\label{fig:hc}
\end{figure*}

\section{Conclusion} 
In this paper, we propose a novel Hessian/Jacobian-free bilevel optimizer named FdeHBO. We show that FdeHBO achieves an $\mathcal{O}(\epsilon^{-1.5})$ sample complexity, which outperforms existing algorithms of the same type by a large margin. Our experiments validate the theoretical results and the effectiveness of the proposed algorithms.
We anticipate that the developed analysis will shed light on developing provable Hessian/Jacobian-free bilevel optimization algorithms and the proposed algorithms may be applied to other applications such as fair machine learning.

\section*{Acknowledgement}
The work is supported in part by NSF under grants 2326592 and 2311274.

\bibliography{ref}
\bibliographystyle{abbrv}

\newpage
\appendix
\allowdisplaybreaks
\noindent
{\huge \bf Supplementary material}

\section{Specifications of Experiments}
\subsection{Hyper-representation}\label{suply-HR}
The hyper-representation problem follows the same problem setup as in  \cite{sow2022on}, and is given by
\begin{align*}
&\min_{\lambda}L_\nu(\lambda) := \frac{1}{|S_\nu|}\sum_{(x_i, y_i)\in S_\nu}L_{CE}(w^*(\lambda)f(\lambda; x_i), y_i)\\
&s.t.\;\;w^*(\lambda)=\argmin_{w}L_{in}(\lambda, w), \;\;\; L_{in}(\lambda, w):=\frac{1}{|S_\tau|}\sum_{(\tau,y_i)\in 
S_\tau}L_{CE}(wf(\lambda, x_i),y_i),
\end{align*}
where $L_{CE}$ denotes the cross-entropy loss, $S_{\nu}$ and $S_{\tau}$ denote the training data and the validation data, and $f(\lambda; x_i)$ donates the features extracted from the data $x_i$. 
We perform the hyper-representation with the 7-layer LeNet network to solve the bilevel problem.
We take the last two layers as the lower-level parameters $w$, and all remaining layers as the upper-level parameters $\lambda$. 
In our experiments, the dimension of $w$ is 850, and the dimension of $\lambda$ is 60856.
For the choices of hyperparameters in \Cref{fig:hr}, all $\eta$ in \cref{def:htg}, \cref{def:htR_free}, and \cref{def:htf_free} are all chosen as 0.9. The stepsize $\beta_t$ for the lower-level update is chosen as 0.8 and the stepsize $\alpha_t$  for the upper-level update is chosen as 0.008. The stepsize $\lambda_t$ is 0.05 and all the $\delta_{\epsilon}$ used are 0.1. The batch size is 256 for both lower- and upper-level processes and the number of lower iterations is 1. For PZOBO-S, $F^2$SA, and $F^3$SA, we use the hyperparameter configurations suggested in their papers and implementations.  
We run the experiment 3 times with different seeds where the solid lines show the average accuracy or loss and the transparent area indicates the variance filled by the max and min values. We run the comparison algorithms using their repositories. The repository of PZOBO-S is available at  \href{https://github.com/sowmaster/esjacobians}{https://github.com/sowmaster/esjacobians}. 

\subsection{Data Hyper-Cleaning}\label{app:hc}
The formulation of the data hyper-cleaning problem is given as follows:
\begin{align*}
    &\min_{\lambda}L_{\nu}(\lambda, w^*) = \frac{1}{|S_{\nu}|}\sum_{(x_i, y_i)\in S_{\nu}} L_{CE}((w^*)^Tx_i, y_i) \nonumber \\
    & s.t. \quad w^* = \arg\min_{w}L(\lambda, w) := \frac{1}{|S_{\tau}|}\sum_{(x_i, y_i)\in S_{\tau}}  \sigma(\lambda_i)L_{CE}(w^Tx_i, y_i) + C\|w\|^2,
\end{align*}
where $L_{CE}$ denotes the cross-entropy loss, $S_{\nu}$ and $S_{\tau}$ denote the training data and the validation data, whose sizes are set to 20000 and 5000, respectively, $\lambda = \{\lambda_i\}_{i \in S_{\tau}}$ and $C$ are the regularization parameters, and $\sigma(\cdot)$ is the sigmoid function.  In the experiments, we set $C = 0.001$ and use 10000 images for testing. In addition, we set the number of iterations in solving the linear system to 3 and set $\eta = 0.5$ in the hypergradient estimation. For our FMBO, we set both inner stepsize (which is the stepsize to update $w$) and outer stepsize (which is the  stepsize to update $\lambda$) to 0.3 and set the batchsize to 100. We also use a decaying $\eta$ for more stable training. 
Following the hyperparameter configurations suggested in the codes and papers of other comparison methods, we set their batchsizes to 1000 for MRBO, stocBiO and set the batchsize to 500 for VRBO. We also set the period number to 3 for VRBO. 
We set the number of inner-loop iterations to 200 for stocBiO, AID-FP and BSA, and to 20 for VRBO, to achieve the best performance. Furthermore, we choose 0.1 as the both inner and outer stepsize for all algorithms except FMBO and SUSTAIN. For SUSTAIN, we set the inner stepsize to 0.03 and the outer stepsize to 0.1 for stability. We run the experiments at two noise rates of 0.1 and 0.15. All results are repeated with 5 random seeds and we use Macbook Pro with a 2.3 GHz Quad-Core Intel Core i5 CPU for training without the requirement of GPU. However, our code also supports GPU.

\section*{Technical Proofs}
\noindent
{\bf Proof outline.} Recall that our proposed FMBO algorithm is a simplification of the proposed Hessian/Jacobian-free FdeHBO method, which uses the Hessian matrix and Jacobian matrix vector directly without finite-difference approximation. For this reason, we first provide auxiliary lemmas for proving the main theorems in \Cref{appd:auxiLemma}. Next, we start the proof of FMBO in \Cref{sec:FMBOproof}, and then extend the analysis to the more complex FdeHBO in \Cref{sec:FdeHBOproof}. 

\section{Proofs of Preliminary Lemmas}\label{appd:auxiLemma}

\begin{lemma}[Boundedness of $v^*$]\label{lm:boundofv} 
Under Assumptions~\ref{as:ulf} and~\ref{as:llf}, we have  
for $v^*$ in~\cref{def:gradR},
    $\|v^*\|^2 \leq \frac{C_{f_y}}{\mu^2_g}.$
\end{lemma}
\begin{proof}
From ~\cref{def:gradR}, we have
\begin{align*}
    \|v^*\|^2 = \|\big[\nabla_{yy}^2g(x,y)\big]^{-1}\nabla_y f(x,y)\|^2 \leq \|\big[\nabla_{yy}^2g(x,y)\big]^{-1}\|^2\|\nabla_y f(x,y)\|^2 \overset{(a)}{\leq} \frac{C_{f_y}}{\mu^2_g},
\end{align*}
where (a) follows Assumptions~\ref{as:ulf} and~\ref{as:llf}. Then, the proof is complete. 
\end{proof}
\begin{lemma}\label{lm:boundofgradf}
Under Assumptions~\ref{as:ulf} and~\ref{as:llf}, for any $\|v\|\leq r_v$, we have that 
$\bar\nabla{f(x,y,v)}$ is Lipschitz continuous w.r.t. $(x,y) \in \mathbb{R}^{d_{x}} \times \mathbb{R}^{d_{y}}$ with constant $L_{F}$ given by 
$$
L_F^2 = 2\Big(L_{f_x}^2 + L^2_{g_{xy}}r_v^2\Big).
$$
\end{lemma}
\begin{proof}
Note from~\cref{def:fffform} that  
 $\bar\nabla{f(x,y,v)} = \nabla_xf(x,y) - \nabla_{xy}^2g(x,y)v$. Then, we have 
\begin{align*}
    \|\bar\nabla{f(x_1,y_1,v)} &- \bar\nabla{f(x_2,y_2,v)}\|^2 \\
    & = \|\big[\nabla_xf(x_1,y_1) - \nabla_xf(x_2,y_2)\big] - \big[\nabla_{xy}^2g(x_1,y_1) - \nabla_{xy}^2g(x_2,y_2)\big]v\|^2 \\
    & \overset{(a)}{\leq} 2\|\nabla_xf(x_1,y_1) - \nabla_xf(x_2,y_2)\|^2 + 2\|\big[\nabla_{xy}^2g(x_1,y_1) - \nabla_{xy}^2g(x_2,y_2)\big]v\|^2 \\ 
    & \overset{(b)}{\leq} 2\|\nabla_xf(x_1,y_1) - \nabla_xf(x_2,y_2)\|^2 + 2r_v^2\|\nabla_{xy}^2g(x_1,y_1) - \nabla_{xy}^2g(x_2,y_2)\|^2 \\ 
    & \overset{(c)}{\leq} 2\Big(L_{f_x}^2 + L^2_{g_{xy}}r_v^2\Big)\big(\|x_1 - x_2\|^2+\|y_1 - y_2\|^2\big) \\
    & = L_F^2\big(\|x_1 - x_2\|^2+\|y_1 - y_2\|^2\big)
\end{align*}
where (a) uses Cauchy–Schwartz inequality; (b) follows from $\|v\|\leq r_v$;  
(c) follows from Assumption~\ref{as:ulf} and Assumption~\ref{as:llf}. Then the proof is complete. 
\end{proof}

\begin{lemma}\label{lm:B}
Under Assumption~\ref{as:llf} and ~\ref{as:sf}, the estimation $\bar{\nabla}f(x, y, v; \xi)$ is unbiased. And the gradient estimate of the upper-level objective satisfies
\begin{align*}
\mathbb{E}_{\bar{\xi}}\|\bar{\nabla}f(x,y,v;\bar{\xi}) - \bar{\nabla}f(x,y,v)\|^2 \leq \sigma_f^2,
\end{align*}
where $\sigma_f^2 = 2(\sigma_{f_x}^2 + r_v^2\sigma_{g_{xy}}^2)$. 
\end{lemma}
\begin{proof}
Based on the definition of $\bar{\nabla}f(x, y, v; \xi)$ in \cref{def:fffform}, we have 
\begin{align*}
\mathbb{E}&\left[ \bar{\nabla}f(x, y, v; \xi) \right] - \bar{\nabla}f(x, y, v) \\
&= \mathbb{E}\left[ \nabla_x f(x, y; \xi) - \nabla_x \nabla_y g(x, y; \xi)v \right] - \left[ \nabla_x f(x, y) - \nabla_x \nabla_y g(x, y)v \right] \\
&= \mathbb{E}\left[\nabla_x f(x, y; \xi) - \nabla_x f(x, y) \right] - \mathbb{E}\left[\nabla_y g(x, y; \xi) - \nabla_y g(x, y) \right]v \\
&= 0.
\end{align*}
By taking norm, we have
\begin{align*}
    \mathbb{E}\| \bar{\nabla}f&(x, y, v; \xi)  - \bar{\nabla}f(x, y, v)\|^2 \\
    &= 
   \mathbb{E} \|\left[ \nabla_x f(x, y; \xi) - \nabla_x \nabla_y g(x, y; \xi)v \right] - \left[ \nabla_x f(x, y) - \nabla_x \nabla_y g(x, y)v \right]\|^2 \\
    &\leq 2\mathbb{E}\| \nabla_x f(x, y; \xi) - \nabla_x f(x, y)\|^2 + 2\mathbb{E}\|\nabla_x \nabla_y g(x, y; \xi)v - \nabla_x \nabla_y g(x, y)v\|^2 \\
    &\overset{(a)}{\leq} 2\mathbb{E}\| \nabla_x f(x, y; \xi) - \nabla_x f(x, y)\|^2 + 2r_v^2\mathbb{E}\|\nabla_x \nabla_y g(x, y; \xi) - \nabla_x \nabla_y g(x, y)\|^2 \\
    &\overset{(b)}{\leq} 2\sigma_{f_x}^2 + 2r_v^2\sigma_{g_{xy}}^2.
\end{align*}
where (a) follows from step 6 in \cref{alg:main}, (b) follows from Assumption~\ref{as:sf}. Then, the proof is complete. 
\end{proof}
\begin{lemma}[\cite{ghadimi2018approximation} lemma 2.2]\label{lm:3ieq}
Under Assumptions~\ref{as:ulf},~\ref{as:llf} and ~\ref{as:sf}, we have, for all $x, x_1 , x_2 \in \mathbb{R}^{d_{y}}$ and $y \in \mathbb{R}^{d_{y}}$,  
\begin{align*}
    \|\bar{\nabla}f(x,y,v) - \nabla \Phi(x)\| &\leq L\|y^*(x) - y\| \\
    \|y^*(x_1) - y^*(x_2)\| &\leq L_y\|x_1 - x_2\| \\
    \|\nabla \Phi(x_1) - \nabla \Phi(x_2)\| &\leq L_f\|x_1 - x_2\|, 
\end{align*}
where the Lipschitz constants are given by 
$$
L = L_{f_x} + \frac{L_{f_y}C_{g_{xy}}}{\mu_g} + C_y\left(\frac{L_{g_{xy}}}{\mu_g}+\frac{L_{g_{yy}}C_{g_{xy}}}{\mu_g^2}\right), \quad L_f = L+\frac{LC_{g_{xy}}}{\mu_g}, \quad L_y = \frac{C_{g_{xy}}}{\mu_g}.
$$
\end{lemma}

\begin{lemma}\label{lm:boundofeJH}
Under Assumption~\ref{as:llf}, for any $\xi$ and $\psi$, define $e_t^J:= \widetilde{J}(x_t, y_t, v_t, \delta_{\epsilon};\xi) - \nabla^2_{xy}g(x_t,y_t;\xi)v_t$, $e_t^H:= \widetilde{H}(x_t, y_t, v_t, \delta_{\epsilon};\psi) - \nabla^2_{yy}g(x_t,y_t;\psi)v_t$, then we have the bound of $e_t^J$ and $e_t^H$ that
\begin{align*}
    \|e_t^J\|^2 \leq C_J^2\delta_{\epsilon}^2, \quad \|e_t^H\|^2 \leq C_H^2\delta_{\epsilon}^2,
\end{align*}
where $C_J := L_{g_{xy}}r_v^2$, $C_H := L_{g_{yy}}r_v^2$. 
\end{lemma}
\begin{proof}

According to definition of $\widetilde{H}(x_t, y_t, v_t, \delta_{\epsilon};\psi)$ in \cref{def:H}, we have
\begin{align}
    \|e_t^H\| &= \|\widetilde{H}(x_t, y_t, v_t, \delta_{\epsilon};\psi) - \nabla^2_{yy} g(x_t, y_t;\psi)v_t\| \nonumber \\
    &= \frac{1}{2\delta_\epsilon}\|\nabla_{y}g(x_t, y_t + \delta_{\epsilon}v;\psi) - \nabla_{y}g(x_t, y_t;\psi) - \nabla^2_{yy} g(x_t, y_t;\psi)(\delta_\epsilon v_t)\| \nonumber \\
    & \quad + \frac{1}{2\delta_\epsilon}\|\nabla_{y}g(x_t, y_t;\psi) - \nabla_{y}g(x_t, y_t - \delta_{\epsilon}v;\psi) - \nabla^2_{yy} g(x_t, y_t;\psi)(\delta_\epsilon v_t)\| \nonumber \\
    & \overset{(a)}{\leq} L_{g_{yy}}\delta_\epsilon\| v_t\|^2 \leq L_{g_{yy}}r_v^2\delta_\epsilon. \nonumber
\end{align}
where (a) uses the Lemma A.1 in \cite{allen2018neon2}. 
Similarly, according to definition of $\widetilde{J}(x_t, y_t, v_t, \delta_{\epsilon};\psi)$ in \cref{def:J}, we have 
\begin{align}
    \|e_t^J\| &= \|\widetilde{J}(x_t, y_t, v_t, \delta_{\epsilon};\psi) - \nabla^2_{xy} g(x_t, y_t;\psi)v_t\| \nonumber \\
    &= \frac{1}{2\delta_\epsilon}\|\nabla_{x}g(x_t, y_t + \delta_{\epsilon}v;\psi) - \nabla_{x}g(x_t, y_t;\psi) - \nabla^2_{xy} g(x_t, y_t;\psi)(\delta_\epsilon v_t)\| \nonumber \\
    & \quad + \frac{1}{2\delta_\epsilon}\|\nabla_{x}g(x_t, y_t;\psi) - \nabla_{x}g(x_t, y_t - \delta_{\epsilon}v;\psi) - \nabla^2_{xy} g(x_t, y_t;\psi)(\delta_\epsilon v_t)\| \nonumber \\
    & \overset{(a)}{\leq} L_{g_{xy}}\delta_\epsilon\| v_t\|^2 \leq L_{g_{xy}}r_v^2\delta_\epsilon. \nonumber
\end{align}
where (a) uses the Lemma A.1 in \cite{allen2018neon2}.

\end{proof}

\section{Proof of~\Cref{th:main}  {\color{blue}(FMBO without first order approximation)}}\label{sec:FMBOproof}
\subsection{Some existential lemmas}
The following lemmas provide characterizations on descents of (1) function value (by~\Cref{lm:corefunction}); (2) iterates of the lower level problem (by~\Cref{lm:yy}); and (3) gradient estimation error of the outer-level function (by~\Cref{lm:errorf}).
\begin{lemma}\label{lm:corefunction}
For non-convex and smooth $\Phi(\cdot)$,with $e^f_t$ defined as: $e_t^f:=h_t^f - \bar{\nabla}f(x_t,y_t,v_t)$, the consecutive iterates of~\Cref{alg:main} satisfy:
\begin{align*}
    \mathbb{E}\big[\Phi(x_{t+1})\big] &\leq \mathbb{E}\big[\Phi(x_{t}) - \frac{\alpha_t}{2}\|\nabla \Phi(x_t)\|^2 - \frac{\alpha_t}{2}(1-\alpha_tL_f)\|h_t^f\|^2 + \alpha_t\|e_t^f\|^2 \nonumber \\
   &\quad + 4\alpha_t\Big(L_{f_x}^2 + \frac{C_{f_y}L_{g_{xy}}^2}{\mu_g^2}\Big)\|y_t - y_t^*\|^2 + 2\alpha_tC_{g_{xy}}\|v_t - v_t^*\|^2\big],
\end{align*} 
for all $t \in \{0,1,...,T-1\}$.
\end{lemma}
\begin{proof}
Using the Lipschitz smoothness of the objective function from~\Cref{lm:3ieq}, we have
\begin{align}\label{eq:obj1}
   \Phi(x_{t+1}) &\leq \Phi(x_{t}) + \langle \nabla\Phi(x_t), x_{t+1} - x_t\rangle + \frac{L_f}{2}\|x_{t+1} - x_t\|^2 \nonumber \\
   & = \Phi(x_{t}) + \langle \nabla\Phi(x_t), h_t^f\rangle + \frac{L_f}{2}\|h_t^f\|^2 \nonumber \\
   & = \Phi(x_{t}) - \frac{\alpha_t}{2}\|\nabla \Phi(x_t)\|^2 - \frac{\alpha_t}{2}(1-\alpha_tL_f)\|h_t^f\|^2 + \frac{\alpha_t}{2}\|h_t^f - \nabla \Phi(x_t)\|^2, 
\end{align}
where the last term of the right-hand side of \cref{eq:obj1} can be show as 
\begin{align}\label{eq:obj2}
    \|h_t^f &- \nabla \Phi(x_t)\|^2 \nonumber \\
    &= \|h_t^f - \bar{\nabla}f(x_t, y_t, v_t) + \bar{\nabla}f(x_t, y_t, v_t) - \nabla \Phi(x_t)\|^2 \nonumber \\
    &\leq 2\|e_t^f\|^2 + 2\|\bar{\nabla}f(x_t, y_t, v_t) - \bar{\nabla}f(x_t, y_t^*, v_t^*)\|^2 \nonumber \\
    &\leq 2\|e_t^f\|^2 + 4\|\bar{\nabla}f(x_t, y_t, v_t) - \bar{\nabla}f(x_t, y_t, v_t^*)\|^2 + 4\|\bar{\nabla}f(x_t, y_t, v_t^*) - \bar{\nabla}f(x_t, y_t, v_t^*)\|^2  \nonumber \\
    &\overset{(a)}{\leq} 2\|e_t^f\|^2 + 4C_{g_{xy}}\|v_t - v_t^*\|^2 + 8\Big(L_{f_x}^2 + \frac{C_{f_y}L_{g_{xy}}^2}{\mu_g^2}\Big)\|y_t - y_t^*\|^2. 
    \end{align}
By applying \cref{eq:obj2} to \cref{eq:obj1}, we have 
\begin{align}\label{eq:obj3}
   \Phi(x_{t+1}) &\leq \Phi(x_{t}) - \frac{\alpha_t}{2}\|\nabla \Phi(x_t)\|^2 - \frac{\alpha_t}{2}(1-\alpha_tL_f)\|h_t^f\|^2 + \alpha_t\|e_t^f\|^2 \nonumber \\
   &\quad + 4\alpha_t\Big(L_{f_x}^2 + \frac{C_{f_y}L_{g_{xy}}^2}{\mu_g^2}\Big)\|y_t - y_t^*\|^2 + 2\alpha_tC_{g_{xy}}\|v_t - v_t^*\|^2, 
\end{align}
then we take the expectation of both sides and the proof is complete. 
\end{proof}
\begin{lemma}[\cite{khanduri2021near} Lemma C.2]\label{lm:yy}
Define $e^g_t := h^g_t -\nabla_y g(x_t,y_t)$. Then the iterates of the inner problem generated by~\Cref{alg:main} satisfy
\begin{align*}
    \mathbb{E}\|y_{t+1} &- y^*_{t+1}\|^2 \nonumber
    \\
    &\leq (1+\gamma_t)(1+\delta_t)\left(1-2\beta_t\frac{\mu_g Lg}{\mu_g + Lg}\right)\mathbb{E}\|y_t-y^*(x_t)\|^2 + \left(1+\frac{1}{\gamma_t}\right)L_y^2\alpha_t^2\mathbb{E}\|h_t^f\|^2 \nonumber
    \\
    &\quad - (1+\gamma_t)(1+\delta_t)\left(\frac{2\beta_t}{\mu_g + Lg}-\beta_t^2 \right)\mathbb{E}\|\nabla_y g(x_t,y_t)\|^2 + (1+\gamma_t)(1+\frac{1}{\delta_t})\beta^2_t\mathbb{E}\|e_t^g\|^2
\end{align*}
for all $t \in \{0, . . . , T-1\}$ with some $\gamma_t, \delta_t > 0$.
\end{lemma}


\subsection{Descent in the gradient estimation error of the upper function}
\begin{lemma}\label{lm:errorf}
Define $e^f_t :=h^f_t - \bar{\nabla}f(x_t,y_t,v_t)$. 
Under \Cref{lm:B}, the iterations of the outer problem generated by~\Cref{alg:main} satisfy
\begin{align*}
    \mathbb{E}\|e^f_{t+1}\|^2 
    &\leq (1-\eta^f_{t+1})^2\mathbb{E}\|e_t^f\|^2 + 2(\eta^f_{t+1})^2\sigma_f^2\\
    &\quad \ +6(1-\eta^f_{t+1})^2\bigg[L_F^2\Big(\alpha_t^2\mathbb{E}\|h_t^f\|^2+\beta_t^2\big(2\mathbb{E}\|e_t^g\|^2+2\mathbb{E}\|\nabla_yg(x_t, y_t)\|^2\big)\Big) \\
    &\qquad \qquad \qquad \qquad + 2C_{g_{xy}}\lambda^2_t\big(\mathbb{E}\|e_t^R\|^2 + L_g^2\mathbb{E}\|v_t - v_t^*\|^2\big)\bigg]   
\end{align*}
for all $t \in \{0, . . . , T-1\}$ with $L_F$ 
in~\Cref{lm:boundofgradf}.
\end{lemma}
\begin{proof}
From the definition of $e_t^f$, we have 
\begin{align}\label{eq:etf}
    &\mathbb{E}\|e_{t+1}^f\|^2 \nonumber\\
    &= \|h^f_{t+1} - \bar{\nabla}f(x_{t+1},y_{t+1},v_{t+1})\|^2 \nonumber\\
    &\overset{(a)}{=}\mathbb{E}\|\eta^f_{t+1}\bar{\nabla}f(x_{t+1},y_{t+1},v_{t+1}; \xi_{t+1}) + (1-\eta^f_{t+1})\big(h_t^f+\bar{\nabla}f(x_{t+1},y_{t+1},v_{t+1}; \xi_{t+1}) \nonumber\\
    &\quad \quad \ \ - \bar{\nabla}f(x_{t},y_{t},v_{t}; \xi_{t+1})\big) - \bar{\nabla}f(x_{t+1},y_{t+1},v_{t+1})\|^2 \nonumber\\
    &\overset{(b)}{=}\mathbb{E}\Big\|(1-\eta^f_{t+1})e_t^f + \eta^f_{t+1}\big(\bar{\nabla}f(x_{t+1},y_{t+1},v_{t+1}; \xi_{t+1}) - \bar{\nabla}f(x_{t+1},y_{t+1},v_{t+1})\big)\nonumber\\
    &\quad \quad \ \ \ +(1-\eta^f_{t+1})\Big(\big(\bar{\nabla}f(x_{t+1},y_{t+1},v_{t+1}; \xi_{t+1}) - \bar{\nabla}f(x_{t+1},y_{t+1},v_{t+1})\big) \nonumber
    \\&\quad \quad \ \ \ -\big(\bar{\nabla}f(x_{t},y_{t},v_{t}; \xi_{t+1}) - \bar{\nabla}f(x_{t},y_{t},v_{t})\big)\Big)\Big\|^2
    \nonumber\\
    &\overset{(c)}{=}(1-\eta^f_{t+1})^2\mathbb{E}\|e_t^f\|^2 + \mathbb{E}\|\eta^f_{t+1}\big(\bar{\nabla}f(x_{t+1},y_{t+1},v_{t+1}; \xi_{t+1}) - \bar{\nabla}f(x_{t+1},y_{t+1},v_{t+1})\big) \nonumber\\
    &\quad\ +(1-\eta^f_{t+1})\Big(\big(\bar{\nabla}f(x_{t+1},y_{t+1},v_{t+1}; \xi_{t+1}) - \bar{\nabla}f(x_{t+1},y_{t+1},v_{t+1})\big) \nonumber
    \\&\quad \quad \ \ \ -\big(\bar{\nabla}f(x_{t},y_{t},v_{t}; \xi_{t+1}) - \bar{\nabla}f(x_{t},y_{t},v_{t}) \big)\Big)\|^2 \nonumber\\
    &\leq(1-\eta^f_{t+1})^2\mathbb{E}\|e_t^f\|^2 + 2(\eta^f_{t+1})^2 \mathbb{E}\|\bar{\nabla}f(x_{t+1},y_{t+1},v_{t+1}; \xi_{t+1}) - \bar{\nabla}f(x_{t+1},y_{t+1},v_{t+1}) \|^2 \nonumber\\ 
    &\quad \ + 2(1-\eta^f_{t+1})^2\mathbb{E}\|\big(\bar{\nabla}f(x_{t+1},y_{t+1},v_{t+1}; \xi_{t+1}) - \bar{\nabla}f(x_{t+1},y_{t+1},v_{t+1})\big) \nonumber
    \\&\quad \quad \ \ \ -\big(\bar{\nabla}f(x_{t},y_{t},v_{t}; \xi_{t+1}) - \bar{\nabla}f(x_{t},y_{t},v_{t})\big)\|^2\nonumber\\
    &\overset{(d)}{\leq} (1-\eta^f_{t+1})^2\mathbb{E}\|e_t^f\|^2 +  2(\eta^f_{t+1})^2\sigma_f^2 \nonumber\\ 
    &\quad \ + 2(1-\eta^f_{t+1})^2\mathbb{E}\|\big(\bar{\nabla}f(x_{t+1},y_{t+1},v_{t+1}; \xi_{t+1}) - \bar{\nabla}f(x_{t+1},y_{t+1},v_{t+1})\big) \nonumber
    \\&\quad \quad \ \ \ -\big(\bar{\nabla}f(x_{t},y_{t},v_{t}; \xi_{t+1}) - \bar{\nabla}f(x_{t},y_{t},v_{t})\big)\|^2 
\end{align}
where 
(a) uses the definition of $h^f_{t+1}$ in~\cref{def:htf}, 
(b) uses the definition that $e^f_t := h^f_t - \bar{\nabla}f(x_t,y_t,v_t)$,
(c) follows because for $\Sigma_{t+1} = \sigma\{y_0, x_0,v_0,...,y_t, x_t, v_t,  y_{t+1}, x_{t+1}, v_{t+1}\}$,
\begin{align}\label{eq:etfinnerproduct}
    &\mathbb{E}\Big\langle e_t^f, \big(\bar{\nabla} f(x_{t+1}, y_{t+1}, v_{t+1}; \xi_{t+1}) - \bar{\nabla} f(x_{t+1}, y_{t+1}, v_{t+1}) \nonumber\\
    &\quad\quad\quad-(1-\eta_{t+1}^f)\big(\bar{\nabla} f(x_{t}, y_{t}, v_{t}; \xi_{t+1}) - \bar{\nabla} f(x_{t}, y_{t}, v_{t}) \big)\Big\rangle \nonumber\\
    &\quad \ =\mathbb{E}\bigg\langle e_t^f, \mathbb{E}\Big[\big(\bar{\nabla} f(x_{t+1}, y_{t+1}, v_{t+1}; \xi_{t+1}) - \bar{\nabla} f(x_{t+1}, y_{t+1}, v_{t+1}) \nonumber\\
    &\qquad\quad\quad\quad\quad \ \ \  -(1-\eta_{t+1}^f)\big(\bar{\nabla} f(x_{t}, y_{t}, v_{t}; \xi_{t+1}) - \bar{\nabla} f(x_{t}, y_{t}, v_{t}) \big)|\Sigma_{t+1}\Big]\bigg\rangle = 0,
\end{align}
which follows from the fact that the second term in the inner product of~\cref{eq:etfinnerproduct} is zero mean as a result of Assumption~\ref{as:sf},
(d) follows from Assumption~\ref{as:sf}.

\noindent Next, we bound the last term of~\cref{eq:etf}
\begin{align}\label{eq:2ndpartofetf}
    &2(1-\eta^f_{t+1})^2\mathbb{E}\|\big(\bar{\nabla}f(x_{t+1},y_{t+1},v_{t+1}; \xi_{t+1}) - \bar{\nabla}f(x_{t},y_{t},v_{t}; \xi_{t+1})\big) \nonumber\\&\quad \quad - \big(\bar{\nabla}f(x_{t+1},y_{t+1},v_{t+1}) - \bar{\nabla}f(x_{t},y_{t},v_{t})\big)\|^2 \nonumber \\
    &\ \ \overset{(a)}{\leq}2(1-\eta^f_{t+1})^2\mathbb{E}\|\bar{\nabla}f(x_{t+1},y_{t+1},v_{t+1}; \xi_{t+1}) - \bar{\nabla}f(x_{t},y_{t},v_{t}; \xi_{t+1})\|^2 \nonumber \\
    &\ \ \leq 6(1-\eta^f_{t+1})^2\mathbb{E}\|\bar{\nabla}f(x_{t+1},y_{t+1},v_{t+1}; \xi_{t+1}) - \bar{\nabla}f(x_{t},y_{t+1},v_{t+1}; \xi_{t+1})\|^2 \nonumber \\
    &\ \ \quad \ + 6(1-\eta^f_{t+1})^2\mathbb{E}\|\bar{\nabla}f(x_{t},y_{t+1},v_{t+1}; \xi_{t+1}) - \bar{\nabla}f(x_{t},y_{t},v_{t+1}; \xi_{t+1})\|^2 \nonumber \\
    &\ \ \quad \ + 6(1-\eta^f_{t+1})^2\mathbb{E}\|\bar{\nabla}f(x_{t},y_{t},v_{t+1}; \xi_{t+1}) - \bar{\nabla}f(x_{t},y_{t},v_{t}; \xi_{t+1})\|^2 \nonumber \\
    &\ \ \overset{(b)}{\leq} 6(1-\eta^f_{t+1})^2L_F^2\mathbb{E}\|x_{t+1} - x_t\|^2 + 6(1-\eta^f_{t+1})^2L_F^2\mathbb{E}\|y_{t+1} - y_t\|^2 \nonumber \\
    &\ \ \quad \ + 6(1-\eta^f_{t+1})^2\mathbb{E}\|\nabla^2_{xy} g(x_t, y_t)(v_{t+1} - v_t)\|^2 \nonumber \\
    &\ \ \overset{(c)}{\leq} 6(1-\eta^f_{t+1})^2L_F^2\mathbb{E}\|x_{t+1} - x_t\|^2 + 6(1-\eta^f_{t+1})^2L_F^2\mathbb{E}\|y_{t+1} - y_t\|^2 \nonumber
    \\&\ \ \quad \ + 6(1-\eta^f_{t+1})^2C_{g_{xy}}\mathbb{E}\|v_{t+1} - v_t\|^2 \nonumber \\
    &\ \ \overset{(d)}{\leq} 6(1-\eta^f_{t+1})^2L_F^2\mathbb{E}\|x_{t+1} - x_t\|^2 + 6(1-\eta^f_{t+1})^2L_F^2\mathbb{E}\|y_{t+1} - y_t\|^2 \nonumber
    \\&\ \ \quad \ + 6(1-\eta^f_{t+1})^2C_{g_{xy}}\mathbb{E}\|w_{t+1} - v_t\|^2 \nonumber \\
    &\ \ \overset{(e)}{\leq} 6(1-\eta^f_{t+1})^2\Big[L_F^2\big(\alpha_t^2\mathbb{E}\|h_t^f\|^2+\beta_t^2\mathbb{E}\|h_t^g\|^2\big) + C_{g_{xy}}\lambda^2_t\mathbb{E}\|h_t^R\|^2\Big] \nonumber \\
    &\ \ \overset{(f)}{\leq}6(1-\eta^f_{t+1})^2\Big[L_F^2\big(\alpha_t^2\mathbb{E}\|h_t^f\|^2+\beta_t^2\mathbb{E}\|h_t^g\|^2\big) + 2C_{g_{xy}}\lambda^2_t\big(\mathbb{E}\|e_t^R\|^2 + \mathbb{E}\|\nabla R_v(x_t, y_t, v_t)\|^2\big)\Big] \nonumber \\
    &\ \ \overset{(g)}{\leq}6(1-\eta^f_{t+1})^2\Big[L_F^2\big(\alpha_t^2\mathbb{E}\|h_t^f\|^2+\beta_t^2\mathbb{E}\|h_t^g\|^2\big) + 2C_{g_{xy}}\lambda^2_t\big(\mathbb{E}\|e_t^R\|^2 + L_g^2\mathbb{E}\|v_t - v_t^*\|^2\big)\Big] \nonumber \\
    &\ \ \overset{(h)}{\leq}6(1-\eta^f_{t+1})^2\bigg[L_F^2\Big(\alpha_t^2\mathbb{E}\|h_t^f\|^2+\beta_t^2\big(2\mathbb{E}\|e_t^g\|^2+2\mathbb{E}\|\nabla_yg(x_t, y_t)\|^2\big)\Big) \nonumber 
    \\&\ \ \quad \ + 2C_{g_{xy}}\lambda^2_t\big(\mathbb{E}\|e_t^R\|^2 + L_g^2\mathbb{E}\|v_t - v_t^*\|^2\big)\bigg],
\end{align}
where (a) follows from the mean-variance inequality: For a random variable $Z$ we have $\mathbb{E}\|Z - \mathbb{E}[Z]\|^2 \leq \mathbb{E}\|Z\|^2$ with $Z$ defined as $Z:= \bar{\nabla}f(x_{t+1},y_{t+1},v_{t+1}; \xi_{t+1}) - \bar{\nabla}f(x_{t},y_{t},v_{t}; \xi_{t+1})$, 
(b) follows from~\Cref{lm:boundofgradf} and~\cref{def:gradR}, 
(c) uses Assumption~\ref{as:llf}, 
(d) follows from the nonexpansiveness of projection that $\|v_{t+1} - v^*_t\| \leq \|\text{Proj}_B(w_{t+1})- \text{Proj}_B(v_t^*)\| \leq \|w_{t+1} - v^*_t\|$ in convex ball $B(0, r_v^2)$,
(e) uses the definition of $h_t^f$ in~\cref{def:htf}, $h_t^g$ in~\cref{def:htg} and $h_t^R$ in~\cref{def:htR}, 
(f) follows from the definition that $e_t^R := h_t^R - \nabla_v R(x_t, y_t, v_t)$, 
(g) uses that result the
\begin{align}\label{eq:boundofnablaR}
    \mathbb{E}\|\nabla_v R(x_t, y_t, v_t)\|^2 &= \mathbb{E}\|\nabla_{yy}^2g(x_t, y_t)v_t - \nabla_yf(x_t, y_t)\|^2 \nonumber\\ 
    &= \mathbb{E}\|\nabla_{yy}^2g(x_t, y_t)(v_t-v_t^*)\|^2 \leq L_g^2\mathbb{E}\|v_t - v_t^*\|^2.
\end{align}
(h) uses the definition that $e_t^g := h_t^g - \nabla_y g(x_t, y_t)$.
Finally, substituting~\cref{eq:2ndpartofetf} in~\cref{eq:etf}, we have the statement in the lemma. 

\noindent Then, the proof is complete.
\end{proof}

\subsection{Descent in the gradient estimation error of the inner function}
\begin{lemma}\label{lm:errorg}
Define $e_t^g := h_t^g - \nabla_y g(x_t, y_t)$. Under Assumption~\ref{as:llf} and~\ref{as:sf}, the iterates generated from~\Cref{alg:main} satisfy 
\begin{align*}
    \mathbb{E} \|e_{t+1}^g\|^2 &\leq \left((1-\eta_{t+1}^g)^2 \mathbb{E} \|e_{t}^g\|^2 + 32(1-\eta_{t+1}^g)^2 L_g^2\beta_t^2 \right)\mathbb{E} \|e_{t}^g\|^2 + 2(\eta_{t+1}^g)^2\sigma_g^2 \nonumber\\ &+ 16(1-\eta_{t+1}^g)^2L_g^2 \alpha_t^2\mathbb{E} \|h_{t}^f\|^2 + 32(1-\eta_{t+1}^g)^2L_g^2 \beta^2_t \mathbb{E}\|\nabla_y g(x_t, y_t)\|^2
\end{align*}
for all $t\in \{0,1,...,T-1\}$.
\end{lemma}
\begin{proof}
From the definition of the $e_t^g$ we have 
\begin{align*}
    \mathbb{E}\|&e_{t+1}^g\|^2 = \mathbb{E}\|h_{t+1}^g - \nabla_y g(x_{t+1}, y_{t+1})\|^2 \\
    &\overset{(a)}{=} \mathbb{E}\|\nabla_y g(x_{t+1}, y_{t+1}; \zeta_{t+1}) + (1-\eta_{t+1}^g)^2(h_t^g - \nabla_y g(x_t, y_t; \zeta_{t+1})) - \nabla_y g(x_{t+1}, y_{t+1})\|^2\\
    &\overset{(b)}{=} \mathbb{E}\|(1-\eta_{t+1}^g)e_t^g + (\nabla_y g(x_{t+1}, y_{t+1}; \zeta_{t+1}) - \nabla_y g(x_{t+1}, y_{t+1}))\\
    &\quad \quad \ \ - (1-\eta_{t+1}^g)(\nabla_y g(x_{t}, y_{t}; \zeta_{t+1})) - \nabla_y g(x_{t}, y_{t})\|^2 \\
    &\overset{(c)}{=}(1-\eta_{t+1}^g)^2\mathbb{E}\|e_t^g\|^2 + \mathbb{E}\|(\nabla_y g(x_{t+1}, y_{t+1}; \zeta_{t+1}) - \nabla_y g(x_{t+1}, y_{t+1}))\\ 
    &\quad \ - (1-\eta_{t+1}^g)(\nabla_y g(x_{t}, y_{t}; \zeta_{t+1})) - \nabla_y g(x_{t}, y_{t})\|^2 \\
    &\overset{(d)}{\leq}(1-\eta_{t+1}^g)^2\mathbb{E}\|e_t^g\|^2 + 2(\eta_{t+1}^g)^2\sigma_g^2\\
    &\quad \ + 2(1-\eta_{t+1}^g)^2\mathbb{E}\|g(x_{t+1}, y_{t+1}; \zeta_{t+1}) - g(x_{t+1}, y_{t+1}) - g(x_{t}, y_{t}; \zeta_{t+1}) + g(x_{t}, y_{t})\|^2 \\
    &\overset{(e)}{\leq}(1-\eta_{t+1}^g)^2\mathbb{E}\|e_t^g\|^2 + 2(\eta_{t+1}^g)^2\sigma_g^2
    + 4(1-\eta_{t+1}^g)^2\mathbb{E}\|g(x_{t+1}, y_{t+1}) - g(x_{t}, y_{t})\|^2\\
    &\quad \ + 4(1-\eta_{t+1}^g)^2\mathbb{E}\|g(x_{t+1}, y_{t+1}; \zeta_{t+1}) - g(x_{t}, y_{t}; \zeta_{t+1})\|^2 \\
    &\overset{(f)}{\leq}(1-\eta_{t+1}^g)^2\mathbb{E}\|e_t^g\|^2 + 2(\eta_{t+1}^g)^2\sigma_g^2
    + 8(1-\eta_{t+1}^g)^2\mathbb{E}\|g(x_{t+1}, y_{t+1}) - g(x_{t+1}, y_{t})\|^2\\
    &\quad \ + 8(1-\eta_{t+1}^g)^2\mathbb{E}\|g(x_{t+1}, y_{t}) - g(x_{t}, y_{t})\|^2 \\
    &\quad \ + 8(1-\eta_{t+1}^g)^2\mathbb{E}\|g(x_{t+1}, y_{t+1}; \zeta_{t+1}) - g(x_{t+1}, y_{t}; \zeta_{t+1})\|^2 \\
    &\quad \ + 8(1-\eta_{t+1}^g)^2\mathbb{E}\|g(x_{t+1}, y_{t}; \zeta_{t+1}) - g(x_{t}, y_{t}; \zeta_{t+1})\|^2 \\
    &\overset{(g)}{\leq}(1-\eta_{t+1}^g)^2\mathbb{E}\|e_t^g\|^2 + 2(\eta_{t+1}^g)^2\sigma_g^2 \\
    &\quad \ + 16(1-\eta_{t+1}^g)^2L_g^2\mathbb{E}\|x_{t+1}-x_t\|^2 + 16(1-\eta_{t+1}^g)^2L_g^2\mathbb{E}\|y_{t+1}-y_t\|^2 \\
    &\overset{(h)}{\leq}(1-\eta_{t+1}^g)^2\mathbb{E}\|e_t^g\|^2 + 2(\eta_{t+1}^g)^2\sigma_g^2 \\
    &\quad \ + 16(1-\eta_{t+1}^g)^2L_g^2\alpha_t^2\mathbb{E}\|h_t^f\|^2 + 16(1-\eta_{t+1}^g)^2L_g^2\beta_t^2\mathbb{E}\|h_t^g\|^2 \\
    &\overset{(i)}{\leq}(1-\eta_{t+1}^g)^2\mathbb{E}\|e_t^g\|^2 + 2(\eta_{t+1}^g)^2\sigma_g^2 
    + 16(1-\eta_{t+1}^g)^2L_g^2\alpha_t^2\mathbb{E}\|h_t^f\|^2 \\
    &\quad \ + 32(1-\eta_{t+1}^g)^2L_g^2\beta_t^2\mathbb{E}\|e_t^g\|^2 + 32(1-\eta_{t+1}^g)^2L_g^2\beta_t^2\mathbb{E}\|\nabla_y g(x_t, y_t)\|^2 \\
    &= (1-\eta_{t+1}^g)^2(1+32L_g^2\beta_t^2)\mathbb{E}\|e_t^g\|^2 + 2(\eta_{t+1}^g)^2\sigma_g^2 + 16(1-\eta_{t+1}^g)^2L_g^2\alpha_t^2\mathbb{E}\|h_t^f\|^2 \\
    &\quad \ + 32(1-\eta_{t+1}^g)^2L_g^2\beta_t^2\mathbb{E}\|\nabla_y g(x_t, y_t)\|^2,
\end{align*}
where (a) uses the definition of $h^g_{t}$ in~\cref{def:htg}, (b) uses the definition of $e^g_{t}$, (c) follows because for $\Sigma_{t+1} = \sigma\{y_0, x_0,...,y_t, x_t, y_{t+1}, x_{t+1}\}$, 
\begin{align*}
    \mathbb{E}\Big\langle e_t^g, \big(\nabla_y g(x_{t+1}, &y_{t+1}; \zeta_{t+1}) - \nabla_y g(x_{t+1}, y_{t+1}) \\
    &-(1-\eta_{t+1}^g)\big(\nabla_y g(x_{t}, y_{t}; \zeta_{t+1}) - \nabla_y g(x_{t}, y_{t})\big)\Big\rangle \\
    =\mathbb{E}\Big\langle e_t^g, \mathbb{E}\big[&\big(\nabla_y g(x_{t+1}, y_{t+1}; \zeta_{t+1}) - \nabla_y g(x_{t+1}, y_{t+1})\big) \\
    &-(1-\eta_{t+1}^g)\big(\nabla_y g(x_{t}, y_{t}; \zeta_{t+1}) - \nabla_y g(x_{t}, y_{t}) \big)|\Sigma_{t+1}\big]\Big\rangle = 0,
\end{align*}
where 
(d) follows from Cauchy–Schwartz inequality and Assumption~\ref{as:sf}, 
(e) and (f) use Cauchy–Schwartz inequality, 
(g) follows from Assumption~\ref{as:llf}; 
(h) follows from Steps 4 and 7 in~\Cref{alg:main}; 
(i) follows from the definition $e_t^g := h_t^g - \nabla_y g(x_t, y_t)$.  
\end{proof}

\subsection{Descent in the gradient estimation error of the $R$ function}
\begin{lemma}\label{lm:errorR}
Define $e_t^R := h_t^R - \nabla_v R(x_t, y_t, v_t)$. Under Assumption~\ref{as:ulf},~\ref{as:llf},~\ref{as:sf}, the iterates generated by~\Cref{alg:main} satisfy 
\begin{align*}
    \mathbb{E} \|e_{t+1}^R\|^2 &\leq (1-\eta_{t+1}^R)^2\left(1 + 48L_g^2\lambda_t^2 \right)\mathbb{E} \|e_{t}^R\|^2 
    + 4(\eta_{t+1}^R)^2\big(\sigma^2_{g_{yy}}r_v^2+\sigma^2_{f_y}\big) \\ 
    &\quad \ + 48(1-\eta_{t+1}^R)^2\left(L^2_{g_{yy}}r^2_v+L^2_{f_y}\right) \left[\alpha_t^2\mathbb{E}\|h_{t}^f\|^2 + 2\beta_t^2(\mathbb{E}\|e_t^g\|^2 + \mathbb{E}\|\nabla_y g(x_t,y_t)\|^2)\right] \\
    &\quad \ +48(1-\eta_{t+1}^R)^2L^4_g\lambda^2_t \mathbb{E}\|v_t-v_t^*\|^2
\end{align*}
for all $t\in \{0,1,...,T-1\}$. 
\end{lemma}
\begin{proof}
For the gradient estimation error of the R function, we have  
\begin{align}\label{ieq:etR1}
&\mathbb{E}\|e^R_{t+1}\|^2 \nonumber\\
&\quad = \mathbb{E}\|h^R_{t+1} - \nabla_v R(x_t, y_t, v_t)\|^2 \nonumber\\
&\quad \overset{(a)}{=} \mathbb{E}\|\nabla_v R(x_{t+1}, y_{t+1}, v_{t+1}; \psi_{t+1}) + (1-\eta^R_{t+1})h_t^R - (1-\eta^R_{t+1})\nabla_v R(x_{t}, y_{t}, v_{t+1}; \psi_{t+1}) \nonumber\\
&\quad \qquad \ \ - \nabla_v R(x_{t+1}, y_{t+1}, v_{t+1})\|^2 \nonumber\\
&\quad \overset{(b)}{=} \mathbb{E}\|(1-\eta^R_{t+1})e_t^R + (1-\eta^R_{t+1})\nabla_v R(x_{t}, y_{t}, v_{t}) \nonumber\\
&\quad \qquad \ \ + \nabla_v R(x_{t+1}, y_{t+1}, v_{t+1}; \psi_{t+1}) - (1-\eta^R_{t+1}) \nabla_v R(x_{t}, y_{t}, v_{t}; \psi_{t+1}) \nonumber
\\&\quad \qquad \ \ - \nabla_v R(x_{t+1}, y_{t+1}, v_{t+1})\|^2 \nonumber\\
&\quad \overset{(c)}{=}(1-\eta^R_{t+1})^2 \mathbb{E}\|e_t^R\|^2
+ \mathbb{E}\|(1-\eta^R_{t+1})\nabla_v R(x_{t}, y_{t}, v_{t}) - (1-\eta^R_{t+1})\nabla_v R(x_{t}, y_{t}, v_{t};\psi_{t+1}) \nonumber\\
&\quad \quad \ + \nabla_v R(x_{t+1}, y_{t+1}, v_{t+1};\psi_{t+1}) - \nabla_v R(x_{t+1}, y_{t+1}, v_{t+1}) \|^2
\end{align} 
where (a) follows from~\cref{def:htR},
(b) uses the definition of $e_t^R := h_t^R - \nabla_v R(x_t, y_t, v_t)$,
(c) follows from the fact that  
\begin{align*}
    \mathbb{E}\Big\langle e_t^R, \big(\nabla_v R(x_{t+1}, y_{t+1},& v_{t+1}; \psi_{t+1}) - \nabla_v R(x_{t+1}, y_{t+1}, v_{t+1}) \\
    &-(1-\eta_{t+1}^R)\big(\nabla_v R(x_{t}, y_{t}, v_{t}; \psi_{t+1}) - \nabla_v R(x_{t}, y_{t}, v_{t})\big)\Big\rangle \\
    =\mathbb{E}\Big\langle e_t^R, \mathbb{E}\big[&\big(\nabla_v R(x_{t+1}, y_{t+1}, v_{t+1}; \psi_{t+1}) - \nabla R(x_{t+1}, y_{t+1}, v_{t+1})\big) \\
    &-(1-\eta_{t+1}^R)\big(\nabla_v R(x_{t}, y_{t}, v_{t}; \psi_{t+1}) - \nabla_v R(x_{t}, y_{t}, v_{t}) \big)|\Sigma_{t+1}\big]\Big\rangle = 0
\end{align*}
for $\Sigma_{t+1} = \sigma\{y_0, v_0, x_0,...,y_t, v_t, x_t, y_{t+1}, v_{t+1}, x_{t+1}\}$. 
For the second part of the right-hand side of~\cref{ieq:etR1}, 
\begin{align}\label{ieq:etR2}
&\quad \ \mathbb{E}\|(1-\eta^R_{t+1})\nabla_v R(x_{t}, y_{t}, v_{t}) - (1-\eta^R_{t+1})\nabla_v R(x_{t}, y_{t}, v_{t};\psi_{t+1}) \nonumber\\
&\quad \quad \ + \nabla_v R(x_{t+1}, y_{t+1}, v_{t+1};\psi_{t+1}) - \nabla_v R(x_{t+1}, y_{t+1}, v_{t+1}) \|^2 \nonumber\\
&\qquad \qquad\quad\quad \overset{(a)}{\leq} 2(\eta^R_{t+1})^2 \mathbb{E}\|\nabla_v R(x_{t+1}, y_{t+1}, v_{t+1};\psi_{t+1}) - \nabla_v R(x_{t+1}, y_{t+1}, v_{t+1})\|^2 \nonumber\\
&\qquad \qquad \quad\quad\quad + 2(1-\eta^R_{t+1})^2 \mathbb{E}\|\nabla_v R(x_{t+1}, y_{t+1}, v_{t+1};\psi_{t+1}) - \nabla_v R(x_{t+1}, y_{t+1}, v_{t+1}) \nonumber\\
&\qquad \qquad \qquad\qquad\qquad\qquad\quad\quad\quad - \nabla_v R(x_{t}, y_{t}, v_{t};\psi_{t+1}) + \nabla_v R(x_{t}, y_{t}, v_{t})\|^2
\end{align}
where (a) uses Cauchy–Schwartz inequality. For the first term of the right-hand side of~\cref{ieq:etR2}, we have 
\begin{align}\label{eq:first_delta_t}
&\mathbb{E}\|\nabla_v R(x_{t+1}, y_{t+1}, v_{t+1};\psi_{t+1}) - \nabla_v R(x_{t+1}, y_{t+1}, v_{t+1})\|^2 \nonumber\\
&\overset{(a)}{=} \mathbb{E}\Big\|\left[\nabla_{yy}^2g(x_{t+1}, y_{t+1}; \psi_{t+1})-\nabla_{yy}^2g(x_{t+1}, y_{t+1})\right]v_{t+1} \nonumber
\\&\qquad\qquad- [\nabla_y f(x_{t+1}, y_{t+1}; \psi_{t+1})-\nabla_y f(x_{t+1}, y_{t+1})]\Big\|^2 \nonumber
\\&\leq 2\mathbb{E}\|\left[\nabla_{yy}^2g(x_{t+1}, y_{t+1}; \psi_{t+1})-\nabla_{yy}^2g(x_{t+1}, y_{t+1})\right]v_{t+1}\|^2 \nonumber
\\&\quad \ + 2\mathbb{E}\|\nabla_y f(x_{t+1}, y_{t+1}; \psi_{t+1})-\nabla_y f(x_{t+1}, y_{t+1})\|^2 \nonumber\\
& \overset{(b)}{\leq} 2\Big(r^2_v\sigma^2_{g_{yy}}+\sigma^2_{f_y}\Big)
\end{align}
where (a) follows the definition of $R(x_t, y_t, v_t)$ in~\cref{def:R}, (b) follows from the boundedness of $v_t$ that $\|v_t\|\leq r_v$ (see line 6 in~\Cref{alg:main}),~\Cref{lm:boundofv} and Assumption~\ref{as:sf}. Moreover, for the second part of~\cref{ieq:etR2}, we have 
\begin{align}\label{eq:diffRss}
&\mathbb{E}\|\nabla_v R(x_{t+1}, y_{t+1}, v_{t+1};\psi_{t+1}) - \nabla_v R(x_{t+1}, y_{t+1}, v_{t+1}) \nonumber
\\&\quad \qquad - \nabla_v R(x_{t}, y_{t}, v_{t};\psi_{t+1}) + \nabla_v R(x_{t}, y_{t}, v_{t})\|^2 \nonumber\\
&\leq 2\mathbb{E}\|\nabla_v R(x_{t+1}, y_{t+1}, v_{t+1};\psi_{t+1}) - \nabla_v R(x_{t}, y_{t}, v_{t};\psi_{t+1})\|^2 \nonumber\\
& \quad \ + 2\mathbb{E}\|\nabla_v R(x_{t+1}, y_{t+1}, v_{t+1}) - \nabla_v R(x_{t}, y_{t}, v_{t})\|^2.
\end{align}
Next, we upper bound the second term of the right-hand side of \cref{eq:diffRss}. In specific, we have 
\begin{align}\label{eq:r_error_v}
&\mathbb{E}\|\nabla_v R(x_{t+1}, y_{t+1}, v_{t+1}) - \nabla_v R(x_{t}, y_{t}, v_{t})\|^2 \nonumber\\
&\overset{(a)}{\leq} 3\mathbb{E}\|\nabla_v R(x_{t+1}, y_{t+1}, v_{t+1}) - \nabla_v R(x_{t}, y_{t+1}, v_{t+1})\|^2 \nonumber
\\&\quad+ 3\mathbb{E}\|\nabla_v R(x_{t}, y_{t+1}, v_{t+1}) - \nabla_v R(x_{t}, y_{t}, v_{t+1})\|^2 \nonumber\\
&\quad \ + 3\mathbb{E}\|\nabla_v R(x_{t}, y_{t}, v_{t+1}) - \nabla_v R(x_{t}, y_{t}, v_{t})\|^2 \nonumber\\
& \overset{(b)}{=} 3\mathbb{E}\|\left[\nabla_{yy}^2g(x_{t+1}, y_{t+1})-\nabla_{yy}^2g(x_t, y_{t+1})\right]v_{t+1} - \left[\nabla_y f(x_{t+1}, y_{t+1})-\nabla_y f(x_t, y_{t+1})\right]\|^2\nonumber \\
&\quad \ + 3\mathbb{E}\|\left[\nabla_{yy}^2g(x_{t}, y_{t+1})-\nabla_{yy}^2g(x_t, y_{t})\right]v_{t+1} - \left[\nabla_y f(x_{t}, y_{t+1})-\nabla_y f(x_t, y_{t})\right]\|^2 \nonumber\\
&\quad \ + 3\mathbb{E}\|\nabla_v R(x_{t}, y_{t}, v_{t+1}) - \nabla_v R(x_{t}, y_{t}, v_{t})\|^2\nonumber \\
&\overset{(c)}{\leq} 6\left(L^2_{g_{yy}}\|v_{t+1}\|^2 + L^2_{f_y}\right)\mathbb{E}\|x_{t+1} - x_t\|^2\nonumber \\ 
& \quad + 6\left(L^2_{g_{yy}}\|v_{t+1}\|^2 + L^2_{f_y}\right)\mathbb{E}\|y_{t+1} - y_t\|^2 + 3L^2_g\lambda^2_t \mathbb{E}\|h_t^R\|^2 \nonumber\\
&\overset{(d)}{\leq} 6\left(L^2_{g_{yy}}r^2_v + L^2_{f_y}\right)\left(\mathbb{E}\|x_{t+1} - x_t\|^2 + \mathbb{E}\|y_{t+1} - y_t\|^2\right) + 3L^2_g\lambda^2_t \mathbb{E}\|h_t^R\|^2 \nonumber\\
& \overset{(e)}{\leq} 6\left(L^2_{g_{yy}}r^2_v + L^2_{f_y}\right)\left(\alpha_t^2\mathbb{E}\|h_t^f\|^2 + \beta_t^2\mathbb{E}\|h_t^g\|^2\right) + 6L^2_g\lambda^2_t \left(\mathbb{E}\|e_t^R\|^2 + \mathbb{E}\|\nabla_v R(x_t, y_t, v_t)\|^2\right)\nonumber \\
& \overset{(f)}{\leq} 6\left(L^2_{g_{yy}}r^2_v + L^2_{f_y}\right)\left[\alpha_t^2\mathbb{E}\|h_t^f\|^2 + 2\beta_t^2(\mathbb{E}\|e_t^g\|^2+\mathbb{E}\|\nabla_yg(x_t,y_t)\|^2)\right]
\nonumber\\&\quad+ 6L^2_g\lambda^2_t \left(\mathbb{E}\|e_t^R\|^2 
+ L_g^2\mathbb{E}\|v_t-v_t^*\|^2\right), 
\end{align}
where 
(a) uses Cauchy–Schwartz inequality,
(b) uses the definition of~\cref{def:gradR}, 
(c) follows from Assumption~\ref{as:ulf},~\ref{as:llf}, and Step 5 and 6 in~\Cref{alg:main}, 
(d) uses~\Cref{lm:boundofv};
(e) follows from Step 4 and 7 in~\Cref{alg:main} and the definition of $h_t^R$ in~\cref{def:htR}, and 
(f) follows from the definition of $h_t^g$ in~\cref{def:htg} and the fact that 
\begin{align*}
\mathbb{E}\|\nabla_v R(x_t, y_t, v_t)\|^2 &= \mathbb{E}\|\nabla_{yy}^2g(x_t, y_t)v_t - \nabla_yf(x_t, y_t)\|^2 \\
&= \mathbb{E}\|\nabla_{yy}^2 g(x_t,y_t)(v_t-v_t^*)\|^2 \leq L_g^2\mathbb{E}\|v_t-v_t^*\|^2.
\end{align*}
For the first term of the right-hand side of \cref{eq:diffRss}, we follow the same steps and get the same upper bound as in~\cref{eq:r_error_v}. Then, incorporating \cref{eq:r_error_v} into \cref{eq:diffRss} yields  
\begin{align}\label{eq:indi_mediat}
\mathbb{E}\|\nabla_v R(x_{t+1},& y_{t+1}, v_{t+1};\psi_{t+1}) - \nabla_v R(x_{t+1}, y_{t+1}, v_{t+1}) \nonumber
\\&\quad \qquad - \nabla_v R(x_{t}, y_{t}, v_{t};\psi_{t+1}) + \nabla_v R(x_{t}, y_{t}, v_{t})\|^2\nonumber
\\&\leq 24\left(L^2_{g_{yy}}r^2_v + L^2_{f_y}\right)\left[\alpha_t^2\mathbb{E}\|h_t^f\|^2 + 2\beta_t^2(\mathbb{E}\|e_t^g\|^2+\mathbb{E}\|\nabla_yg(x_t,y_t)\|^2)\right]
\nonumber\\&\quad+ 24L^2_g\lambda^2_t \left(\mathbb{E}\|e_t^R\|^2 
+ L_g^2\mathbb{E}\|v_t-v_t^*\|^2\right).
\end{align}
Then, incorporating \cref{eq:first_delta_t} and  \cref{eq:indi_mediat} into 
\cref{ieq:etR2}, we have 
\begin{align}\label{eq:etR_2ndpart}
&\mathbb{E}\|(1-\eta^R_{t+1})\nabla_v R(x_{t}, y_{t}, v_{t}) - (1-\eta^R_{t+1})\nabla_v R(x_{t}, y_{t}, v_{t};\psi_{t+1}) \nonumber\\
&\ \quad + \nabla_v R(x_{t+1}, y_{t+1}, v_{t+1};\psi_{t+1}) - \nabla_v R(x_{t+1}, y_{t+1}, v_{t+1}) \|^2 \nonumber\\
&\qquad  \leq 4(\eta_{t+1}^R)^2\big(r_v^2\sigma_{g_{yy}}^2 + \sigma_{f_y}^2\big)\nonumber\\
&\qquad \quad \ + 48(1-\eta_{t+1}^R)^2\left(L^2_{g_{yy}}r^2_v + L^2_{f_y}\right) \left[\alpha_t^2\mathbb{E}\|h_{t}^f\|^2 + 2\beta_t^2(\mathbb{E}\|e_t^g\|^2 + \mathbb{E}\|\nabla_y g(x_t,y_t)\|^2)\right] \nonumber\\
&\qquad \quad \ + 48(1-\eta_{t+1}^R)^2L^2_g\lambda^2_t \left(\mathbb{E}\|e_t^R\|^2 + L_g^2\mathbb{E}\|v_t-v_t^*\|^2\right),
\end{align}
which, incorporated into \cref{ieq:etR1}, yields
\begin{align*}
    \mathbb{E} \|e_{t+1}^R\|^2 &\leq (1-\eta_{t+1}^R)^2\left(1 + 48L_g^2\lambda_t^2 \right)\mathbb{E} \|e_{t}^R\|^2 
    + 4(\eta_{t+1}^R)^2\big(r_v^2\sigma_{g_{yy}}^2 + \sigma_{f_y}^2\big) \\ 
    &\quad \ + 48(1-\eta_{t+1}^R)^2\left(L^2_{g_{yy}}r^2_v + L^2_{f_y}\right) \left[\alpha_t^2\mathbb{E}\|h_{t}^f\|^2 + 2\beta_t^2(\mathbb{E}\|e_t^g\|^2 + \mathbb{E}\|\nabla_y g(x_t,y_t)\|^2)\right] \\
    &\quad \ +48(1-\eta_{t+1}^R)^2L^4_g\lambda^2_t \mathbb{E}\|v_t-v_t^*\|^2,
\end{align*}
which finishes the proof. 
\end{proof}

\subsection{Descent in iterates of the LS problem}
\begin{lemma}\label{lm:vv}
Under the Assumption~\ref{as:ulf},~\ref{as:llf}, the iterates of the LS problem generated according to~\Cref{alg:main} satisfy
\begin{align*}
\mathbb{E}&\|v_{t+1} - v_{t+1}^*\|^2  \\
&\leq (1+\gamma_t')\left(1+\delta_t'\right)\left[\left(1-2\lambda_t\frac{(L_g+L_g^3)\mu_g}{\mu_g + L_g} + \lambda_t^2L_g^2\right)\mathbb{E}\|v_t - v^*_t\|^2\right] \\
&\quad \ + (1+\gamma_t')\left(1+\frac{1}{\delta_t'}\right)\lambda_t^2\mathbb{E}\|e_t^R\|^2 \\
&\quad \ + (1+\frac{1}{\gamma_t'})\left( \frac{2L^2_{f_y}}{\mu^2_g} + \frac{2C_{f_y}L^2_{g_{yy}}}{\mu^4_g} \right)\left[ \alpha_t^2\mathbb{E}\|h_t^f\|^2 + \beta_t^2\left(2\mathbb{E}\|e_t^g\|^2 + 2\mathbb{E}\|\nabla_y g(x_t, y_t)\|^2 \right) \right].
\end{align*}
for all $t \in \{0, . . . , T-1\}$ with some $\gamma_t'>0$ and $ \delta_t' > 0$.
\end{lemma}
\begin{proof}
Letting $v_{k} := v(x_k, y_k)$ and $ v_{k}^* := v^*(x_k, y_k)$ and choosing the radius $r_v = \frac{C_{f_y}}{\mu_g}$ (see Step 6 in~\Cref{alg:main}), we have 
\begin{align}\label{eq:vgaps}
    \mathbb{E}\|v_{t+1} &- v_{t+1}^*\|^2 
    \overset{(a)}{\leq} (1+\gamma_t')\mathbb{E}\|v_{t+1} - v_{t}^*\|^2 + (1+\frac{1}{\gamma_t'})\mathbb{E}\|v_{t}^* - v_{t+1}^*\|^2\nonumber \\
    &\overset{(b)}{\leq} (1+\gamma_t')\mathbb{E}\|w_{t+1} - v_{t}^*\|^2 + (1+\frac{1}{\gamma_t'})\mathbb{E}\|v_{t}^* - v_{t+1}^*\|^2\nonumber \\
    &\overset{(c)}{\leq} (1+\gamma_t')\mathbb{E}\|v_t - \lambda_th_t^R - v_{t}^*\|^2 + (1+\frac{1}{\gamma_t'})\mathbb{E}\|v_{t}^* - v_{t+1}^*\|^2 \nonumber\\
    &\overset{(d)}{\leq} (1+\gamma_t')(1+\delta_t')\mathbb{E}\|v_t - \lambda_t\nabla_v R(x_t, y_t, v_t) - v_{t}^*\|^2 \nonumber\\
    &\quad \ + (1+\gamma_t')(1+\frac{1}{\delta_t'})\mathbb{E}\lambda_t^2\|h_t^R - \nabla_v R(x_t, y_t, v_t)\|^2 
     +(1+\frac{1}{\gamma_t'})\mathbb{E}\|v_{t}^* - v_{t+1}^*\|^2,
\end{align}
where (a) follows from Young’s inequality, 
(b) follows Step 6 in ~\Cref{alg:main} and \Cref{lm:boundofv} that $v^*$ is in a ball with radius $r_v$(define as $B(0,r_v^2)$), then we have $\|v_{t+1}- v_t^*\|=\|Proj_B(w_{t+1})- Proj_B(v_t^*)\|\leq\|w_{t+1}- v_t^*\|$ (this inequality is based on the nonexpansiveness of projection),
(c) follows from Step 5 in ~\Cref{alg:main}, and 
(d) uses Young’s inequality and the definition of $h_t^R$ in~\cref{def:htR}.  
For the first term of the above \cref{eq:vgaps}, we have
\begin{align}\label{eq:vtlambda}
    \mathbb{E}\|v_t - &\lambda_t\nabla_v R(x_t, y_t, v_t) - v_{t}^*\|^2 \nonumber\\
    &= \mathbb{E}\|v_t - v^*_t\|^2 + \lambda_t^2\mathbb{E}\|\nabla_v R(x_t, y_t, v_t)\|^2 - 2\lambda_t
    \mathbb{E}\big \langle \nabla_v R(x_t, y_t, v_t), v_t - v_t^* \big \rangle \nonumber\\
    &\overset{(a)}{\leq} \bigg(1-2\lambda_t\frac{\mu_gL_g}{\mu_g + L_g}\bigg)\mathbb{E}\|v_t - v^*_t\|^2 - \bigg(2\lambda_t\frac{\mu_g L_g}{\mu_g + L_g} - \lambda_t^2\bigg)\mathbb{E}\|\nabla_v R(x_t, y_t, v_t)\|^2 \nonumber\\
    &\overset{(b)}{\leq} \bigg(1-2\lambda_t\frac{(L_g+L_g^3)\mu_g}{\mu_g + L_g} + \lambda_t^2L_g^2\bigg)\mathbb{E}\|v_t - v^*_t\|^2
\end{align}
where (a) follows from the strong convexity of $R$ function(in~\cref{def:R}) that 
\begin{align*}
    \mathbb{E}\langle\nabla_v R(x_t, y_t, v_t), v_t - v_t^* \rangle \geq \frac{\mu_gL_g}{\mu_g+L_g}\mathbb{E}\|v_t - v_t^*\|^2 + \frac{1}{\mu_g+L_g}\mathbb{E}\|\nabla_v R(x_t, y_t, v_t)\|^2,
\end{align*}
and  (b) follows from \cref{eq:boundofnablaR}.
For the last term of \cref{eq:vgaps}, we have 
\begin{align}\label{eq:vt1star}
    \mathbb{E}\|&v_{t+1}^* - v_{t}^*\|^2 = \mathbb{E}\|[\nabla_{yy}^2g(x_{t+1}, y_{t+1})]^{-1}\nabla_y f(x_{t+1}, y_{t+1}) - [\nabla_{yy}^2g(x_{t}, y_{t})]^{-1}\nabla_y f(x_{t}, y_{t})\|^2 \nonumber\\
    &= \mathbb{E}\|[\nabla_{yy}^2g(x_{t+1}, y_{t+1})]^{-1}\nabla_y f(x_{t+1}, y_{t+1}) - [\nabla_{yy}^2g(x_{t+1}, y_{t+1})]^{-1}\nabla_y f(x_{t}, y_{t}) \nonumber\\
    &\quad\ \ \ + [\nabla_{yy}^2g(x_{t+1}, y_{t+1})]^{-1}\nabla_y f(x_{t}, y_{t}) - [\nabla_{yy}^2g(x_{t}, y_{t})]^{-1}\nabla_y f(x_{t}, y_{t})\|^2 \nonumber\\
    &\overset{(a)}{\leq} 2 \mathbb{E}\|[\nabla_{yy}^2g(x_{t+1}, y_{t+1})]^{-1}\left(\nabla_y f(x_{t+1}, y_{t+1}) - \nabla_y f(x_{t}, y_{t})\right)\|^2 \nonumber\\
    &\quad\ + 2\mathbb{E}\|\left([\nabla_{yy}^2g(x_{t+1}, y_{t+1})]^{-1} - [\nabla_{yy}^2g(x_{t}, y_{t})]^{-1}\right)\nabla_y f(x_{t}, y_{t}) \|^2 \nonumber\\
    &\overset{(b)}{\leq} \frac{2L^2_{f_y}}{\mu_g^2}\mathbb{E}\|(x_{t+1}, y_{t+1}) - (x_{t}, y_{t})\|^2 \nonumber\\
    &\quad\ + 2C_{f_y}\mathbb{E}\|[\nabla_{yy}^2g(x_{t}, y_{t})]^{-1}[\nabla_{yy}^2g(x_{t}, y_{t}) - \nabla_{yy}^2g(x_{t+1}, y_{t+1})][\nabla_{yy}^2g(x_{t+1}, y_{t+1})]^{-1}\|^2 \nonumber\\
    &\overset{(c)}{\leq} \left( \frac{2L^2_{f_y}}{\mu^2_g} + \frac{2C_{f_y}L^2_{g_{yy}}}{\mu^4_g} \right)\mathbb{E}\|(x_{t+1}, y_{t+1}) - (x_{t}, y_{t})\|^2 \nonumber\\
    &\overset{}{\leq} \left( \frac{2L^2_{f_y}}{\mu^2_g} + \frac{2C_{f_y}L^2_{g_{yy}}}{\mu^4_g} \right)\left( \mathbb{E}\|x_{t+1}-x_{t}\|^2 + \mathbb{E}\|y_{t+1}-y_{t}\|^2\right)\nonumber \\
    &\overset{(d)}{=} \left( \frac{2L^2_{f_y}}{\mu^2_g} + \frac{2C_{f_y}L^2_{g_{yy}}}{\mu^4_g} \right)\left( \alpha_t^2\mathbb{E}\|h_t^f\|^2 + \beta_t^2\mathbb{E}\|h_t^g\|^2 \right) \nonumber\\
    &\overset{(e)}{=} \left( \frac{2L^2_{f_y}}{\mu^2_g} + \frac{2C_{f_y}L^2_{g_{yy}}}{\mu^4_g} \right)\left[ \alpha_t^2\mathbb{E}\|h_t^f\|^2 + \beta_t^2\left(2\mathbb{E}\|e_t^g\|^2 + 2\mathbb{E}\|\nabla_y g(x_t, y_t)\|^2 \right) \right],
\end{align}
where (a) follows from Cauchy–Schwartz inequality, 
(b) follows from  Assumption~\ref{as:ulf} and~\ref{as:llf}, 
(c) follows from  Assumption~\ref{as:llf},  
(d) follows from Steps 4 and 7 in~\Cref{alg:main}, and 
(e) uses the definition of $h_t^g$ in~\cref{def:htg}.
\noindent
Finally, incorporating \cref{eq:vtlambda} and \cref{eq:vt1star} into \cref{eq:vgaps},  
we have 
\begin{align*}
\mathbb{E}\|&v_{t+1} - v_{t+1}^*\|^2  \\
&\leq (1+\gamma_t')\left(1+\delta_t'\right)\left[\left(1-2\lambda_t\frac{(L_g+L_g^3)\mu_g}{\mu_g + L_g} + \lambda_t^2L_g^2\right)\mathbb{E}\|v_t - v^*_t\|^2\right] \\
&\quad \ + (1+\gamma_t')\left(1+\frac{1}{\delta_t'}\right)\lambda_t^2\mathbb{E}\|e_t^R\|^2 \\
&\quad \ + (1+\frac{1}{\gamma_t'})\left( \frac{2L^2_{f_y}}{\mu^2_g} + \frac{2C^2_{f_y}L^2_{g_{yy}}}{\mu^4_g} \right)\left[ \alpha_t^2\mathbb{E}\|h_t^f\|^2 + \beta_t^2\left(2\mathbb{E}\|e_t^g\|^2 + 2\mathbb{E}\|\nabla_y g(x_t, y_t)\|^2 \right) \right].
\end{align*}
Then, the proof is complete. 
\end{proof}

\subsection{Descent in the Potential Function}
Define the potential function as
\begin{align}\label{def:potentialF}
V_t :=& \Phi(x_{t}) + K_1\|y_{t} - y^*(x_{t})\|^2 + K_2\|v_{t} - v^*(x_{t}, y_{t})\|^2 \nonumber
\\&+  \frac{1}{\bar{c}_{\eta_f}} \frac{\|e_{t}^f\|^2}{\alpha_{t-1}} + \frac{1}{\bar{c}_{\eta_g}} \frac{\|e_{t}^g\|^2}{\alpha_{t-1}} + \frac{1}{\bar{c}_{\eta_R}} \frac{\|e_{t}^R\|^2}{\alpha_{t-1}},
\end{align}
where the coefficients are given by
\begin{align}
K_1 &= \frac{8(L^2_{f_y} + \frac{C_{f_y}L62_{g_{xy}}}{\mu_g^2})}{c_{\beta}L_{\mu_g}}, \quad K_2 = \frac{4C_{g_{xy}}}{c_{\lambda}L_{\mu_g}}; \nonumber \\
\bar{c}_{\eta_f} &= \max\left\{96L_F^2, \ 12L_g^2c_{\lambda}^2, \ \frac{48L_{\mu_g}L^2_fc^2_{\beta}\max\{L_{\mu_g}, \mu_g+L_g\}}{L^2_{f_y}+\frac{C_{f_y}L^2_{g_{xy}}}{\mu_g^2}}, \ \frac{3}{2}L^2_{\mu_g}c^2_\lambda\right\}, \nonumber \\ 
\bar{c}_{\eta_g} &= \max\left\{256L_g^2, \ \frac{128L_{\mu_g}L^2_gc^2_{\beta}\max\{L_{\mu_g}, \mu_g+L_g\}}{L^2_{f_y}+\frac{C_{f_y}L^2_{g_{xy}}}{\mu_g^2}}\right\}, \ 
\nonumber\\
\bar{c}_{\eta_R} &= \max\bigg\{768(L^2_{g_{yy}}r_v^2 + L^2_{f_y}),\ \frac{48L_g^4c^2_{\lambda}}{C_{g_{xy}}}, \ \frac{384L_{\mu_g}(L^2_{g_{yy}}r_v^2 + L^2_{f_y})c^2_{\beta}\max\{L_{\mu_g}, \mu_g+L_g\}}{L^2_{f_y}+\frac{C_{f_y}L^2_{g_{xy}}}{\mu_g^2}}\bigg\}.\nonumber
\end{align}
\begin{lemma}\label{lm:mergeV}
Suppose Assumptions~\ref{as:ulf} ,\ref{as:llf} and \ref{as:sf} are satisfied. 
Choose the parameters of~\Cref{alg:main} as 
\begin{align*}
    &\alpha_t := \frac{1}{(w+t)^{1/3}},\  \beta_t := c_\beta \alpha_t,\ \lambda_t := c_\lambda \alpha_t;\ \eta_{t+1}^f := c_{\eta_f}\alpha_t^2,\ \eta_{t+1}^g := c_{\eta_g}\alpha_t^2,\ \eta_{t+1}^R := c_{\eta_R}\alpha_t^2,
\end{align*}
where the constants are given by 
\begin{align}\label{def:parameters}
    &c_{\beta} \geq \sqrt{\frac{512L_y^2(L^2_{f_y}+\frac{C_{f_y}L^2_{g_{xy}}}{\mu_g^2})}{L^2_{\mu_g}}}, \nonumber\\
    &c_\lambda \geq \sqrt{\max\left\{\frac{1024C_{g_{xy}}}{L^2_{\mu_g}}\Big(\frac{L^2_{f_y}}{\mu_g^2}+\frac{C_{f_y}L^2_{g_{yy}}}{\mu_g^4}\Big), \frac{128(\mu_g+L_g)C_{g_{xy}}}{L_{\mu_g}}c^2_{\beta}, 128C_{g_{xy}}c_\beta^2\right\}}\ ; \nonumber\\
    &c_{\eta_f} = \frac{1}{3L_f} + \bar{c}_{\eta_f},\  c_{\eta_g} = \frac{1}{3L_f} + 32L_g^2c_{\beta}^2 + \left[ \frac{17(L^2_{f_y}+\frac{C_{f_y}L^2_{g_{xy}}}{\mu_g^2})}{L^2_{\mu_g}}\right]\bar{c}_{\eta_g}, \nonumber\\ 
    &c_{\eta_R} = \frac{1}{3L_f} + 48L_g^2c_\lambda^2 + \left[ \frac{16C_{g_{xy}}}{L^2_{\mu_g}}\right]\bar{c}_{\eta_R};\ \ \sigma_R := \sqrt{\sigma^2_{g_{yy}}r^2_v+\sigma^2_{f_y}}, \nonumber\\
    &w \geq \left(\max\left\{c_\beta(\mu_g+L_g), \ \frac{c_\lambda(\mu_g+L_g)}{2\mu_gL_g} \right\}\right)^3 - 1. 
\end{align}
Then the iterates generated by~\Cref{alg:main} satisfy
$$\mathbb{E}[V_{t+1} - V_{t}] \leq -\frac{\alpha_t}{2}\mathbb{E}\|\nabla \Phi(x_t)\|^2 + \frac{2(\eta^f_{t+1})^2}{\bar{c}_{\eta_f} \alpha_t}\sigma_f^2 + \frac{2(\eta^g_{t+1})^2}{\bar{c}_{\eta_g} \alpha_t}\sigma_g^2 + 
\frac{4(\eta^R_{t+1})^2}{\bar{c}_{\eta_R} \alpha_t}\sigma_R^2,$$
for all $t \in \{0,1,...,T-1\}$. 
\end{lemma}

\begin{proof}
Based on the definition of $V_t$ in~\cref{def:potentialF}, it can be seen that $V_t$ contains six parts.  We next develop five important inequalities to prove \Cref{lm:mergeV}.

\noindent{\bf Step 1. Bound $\mathbb{E}\|y_{t} - y^*(x_{t})\|^2$ in \cref{def:potentialF}.} 

\noindent Based on~\Cref{lm:yy}, we have 
\begin{align}\label{eq:minusy}
    \mathbb{E}\|y_{t+1} &- y^*(x_{t+1})\|^2 - \mathbb{E}\|y_{t} - y^*(x_{t})\|^2 \nonumber\\ &\leq \left[ (1+\gamma_t)(1+\delta_t)\left(1-2\beta\frac{\mu_gL_g}{\mu_g+L_g}\right) -1\right]\mathbb{E}\|y_{t} - y^*(x_{t})]\|^2 \nonumber\\
    &\quad \ - (1+\gamma_t)(1+\delta_t)\left( \frac{2\beta_t}{\mu_g+L_g} - \beta_t^2 \right)\mathbb{E}\|\nabla_y g(x_t,y_t)\|^2 \nonumber\\
    &\quad \ + (1+\gamma_t)(1+\frac{1}{\delta_t})\beta_t^2\mathbb{E}\|e_t^g\|^2 + \left(1+\frac{1}{\gamma_t}\right)L^2_y\alpha_t^2\mathbb{E}\|h_t^f\|^2.
\end{align}
Choose 
$\gamma_t = \frac{\beta_t L_{\mu_g}/2}{1-\beta_t L_{\mu_g}}$ and $ \delta_t = \frac{\beta_tL_{\mu_g}}{1-2\beta_tL_{\mu_g}}$. 
Then,  the following equations and inequalities are satisfied. 
\begin{align}\label{eq:gammaanddelta}
(1+\gamma_t)(1+\delta_t)(1-2\beta_t L_{\mu_g}) &= 1-\frac{\beta_tL_{\mu_g}}{2}, \nonumber\\
(1+\delta_t)(1-2\beta_t L_{\mu_g}) &= 1-\beta_tL_{\mu_g}, \nonumber\\
(1+\gamma_t)(1-\beta_t L_{\mu_g}) &= 1-\frac{\beta_tL_{\mu_g}}{2}, \nonumber\\
1+\frac{1}{\delta_t} \leq \frac{1}{\beta_t L_{\mu_g}}, \quad
&1+\frac{1}{\gamma_t} \leq \frac{2}{\beta_t L_{\mu_g}},
\end{align}
where $L_{\mu_g} = \frac{\mu_gL_g}{\mu_g+L_g}$. 
Based on the selection of $w$ in \cref{def:parameters}, we have 
$(1+\gamma_t)(1+\frac{1}{\delta_t})\leq \frac{2}{\beta_t L_{\mu_g}}$.

\noindent Substituting~\cref{eq:gammaanddelta}'s bounds in~\cref{eq:minusy}, we have 
\begin{align*}
   \mathbb{E}\|y_{t+1} - &y^*(x_{t+1})\|^2 - \mathbb{E}\|y_{t} - y^*(x_{t})\|^2 \\ &\leq -\frac{\beta_t L_{\mu_g}}{2} \mathbb{E}\|y_{t} - y^*(x_{t})\|^2 
   - \left(\frac{2\beta_t}{\mu_g + L_g} - \beta^2_t\right)\mathbb{E}\|\nabla_y g(x_t,y_t)\|^2 \\
   &\quad \ +\frac{2}{\beta_t L_{\mu_g}}\beta^2_t\mathbb{E}\|e_t^g\|^2 + \frac{2}{\beta_t L_{\mu_g}}L_y^2\alpha_t^2\mathbb{E}\|h_t^f\|^2,
\end{align*}
which, 
in conjunction with our selection in \cref{def:parameters} 
,  yields 
\begin{align}\label{eq:ytplu1}
   \mathbb{E}\|y_{t+1} - &y^*(x_{t+1})\|^2 - \mathbb{E}\|y_{t} - y^*(x_{t})\|^2 \nonumber\\ &\leq -\frac{\beta_t L_{\mu_g}}{2} \mathbb{E}\|y_{t} - y^*(x_{t})\|^2 
   - \frac{\beta_t}{\mu_g + L_g}\mathbb{E}\|\nabla_y g(x_t,y_t)\|^2 \nonumber\\
   &\quad \ +\frac{2\beta_t}{ L_{\mu_g}}\mathbb{E}\|e_t^g\|^2 + \frac{2}{\beta_t L_{\mu_g}}L_y^2\alpha_t^2\mathbb{E}\|h_t^f\|^2.
\end{align}

\noindent Using $\beta_t = c_\beta\alpha_t$, and multiplying both sides of \cref{eq:ytplu1} by $K_1$, we have 
\begin{align}\label{ieq:yy}
   K_1 \mathbb{E}\big[\|y_{t+1} &- y^*(x_{t+1})\|^2 - \|y_{t} - y^*(x_{t})\|^2\big] \nonumber\\
   &\leq -4\Big(L^2_{f_x} + \frac{C_{f_y}L^2_{g_{xy}}}{\mu_g^2}\Big)\alpha_t
   \mathbb{E}\|y_{t} - y^*(x_{t})\|^2 
   - \frac{L^2_{f_y}+\frac{C_{f_y}L^2_{g_{xy}}}{\mu_g^2}}{L_{\mu_g}(\mu_g+L_g)}\alpha_t\mathbb{E}\|\nabla_y g(x_t,y_t)\|^2 \nonumber\\ 
   &\quad \ +\frac{16(L^2_{f_y}+\frac{C_{f_y}L^2_{g_{xy}}}{\mu_g^2})}{L^2_{\mu_g}}\alpha_t\mathbb{E}\|e_t^g\|^2 
   + \frac{\alpha_t}{32}\mathbb{E}\|h_t^f\|^2.
\end{align}

\noindent{\bf Step 2. Bound $\mathbb{E}\|v_{t} - v^*(x_{t})\|^2$ in \cref{def:potentialF}.} 

\noindent Next, we deal with $\mathbb{E}\|v_{t+1} - v^*_{t+1}\|^2$ in similar way. 
Based on the parameter selections in \cref{def:parameters}, we have   $\lambda_t \leq \frac{2\mu_g L_g}{\mu_g + L_g}$ , which  combined with~\Cref{lm:vv},  
yields 
\begin{align*}
\mathbb{E}\|&v_{t+1} - v_{t+1}^*\|^2 \\
&\leq (1+\gamma_t')\left(1+\delta_t'\right)\left[\left(1-2\lambda_t\frac{(L_g+L_g^3)\mu_g}{\mu_g + L_g} + \lambda_t^2L_g^2\right)\|v_t - v^*_t\|^2\right] \\
&\quad \ + (1+\gamma_t')\left(1+\frac{1}{\delta_t'}\right)\lambda_t^2\mathbb{E}\|e_t^R\|^2 \\
&\quad \ + (1+\frac{1}{\gamma_t'})\left( \frac{2L^2_{f_y}}{\mu^2_g} + \frac{2C_{f_y}L^2_{g_{yy}}}{\mu^4_g} \right)\left[ \alpha_t^2\mathbb{E}\|h_t^f\|^2 + \beta_t^2\left(2\mathbb{E}\|e_t^g\|^2 + 2\mathbb{E}\|\nabla_y g(x_t, y_t)\|^2 \right) \right]\\
&\leq (1+\gamma_t')\left(1+\delta_t'\right)\left[\left(1 - \frac{2\lambda_t L_g \mu_g}{\mu_g+L_g}\right)\mathbb{E}\|v_{t} - v_{t}^*\|^2\right] \\
&\quad \ + (1+\gamma_t')\left(1+\frac{1}{\delta_t'}\right)\lambda_t^2\mathbb{E}\|e_t^R\|^2 \\
&\quad \ + (1+\frac{1}{\gamma_t'})\left( \frac{2L^2_{f_y}}{\mu^2_g} + \frac{2C_{f_y}L^2_{g_{yy}}}{\mu^4_g} \right)\left[ \alpha_t^2\mathbb{E}\|h_t^f\|^2 + \beta_t^2\left(2\mathbb{E}\|e_t^g\|^2 + 2\mathbb{E}\|\nabla_y g(x_t, y_t)\|^2 \right) \right],
\end{align*}
Similarly to Step 1, we choose
$
\delta_t' = \frac{\lambda_t L_{\mu_g}}{1-2\lambda_t L_{\mu_g}}$ and $ \gamma_t' = \frac{\lambda_t L_{\mu_g}/2}{1-\lambda_t L_{\mu_g}}
$
which implies 
\begin{align*}
    1+\frac{1}{\delta_t'} \leq \frac{1}{\lambda_t L_{\mu_g}}, \quad 1+\frac{1}{\gamma_t'} \leq \frac{2}{\lambda_t L_{\mu_g}}, \quad
    (1+\gamma_t')(1+\frac{1}{\delta_t'})\leq \frac{2}{\lambda_t L_{\mu_g}}.
\end{align*}
Thus, we have
\begin{align}\label{eq:vsubvstar}
\mathbb{E}&\|v_{t+1} - v_{t+1}^*\|^2 \nonumber\\
&\leq \left(1-\frac{\lambda_t L_{\mu_g}}{2}\right)\mathbb{E}\|v_{t} - v_{t}^*\|^2 
+ \frac{2}{\lambda_t L_{\mu_g}}\lambda_t^2\mathbb{E}\|e_t^R\|^2\nonumber \\
&\quad \ + \frac{2}{\lambda_t L_{\mu_g}}\left( \frac{2L^2_{f_y}}{\mu^2_g} + \frac{2C_{f_y}L^2_{g_{yy}}}{\mu^4_g} \right)\left[ \alpha_t^2\mathbb{E}\|h_t^f\|^2 + \beta_t^2\left(2\mathbb{E}\|e_t^g\|^2 + 2\mathbb{E}\|\nabla_y g(x_t, y_t)\|^2 \right) \right].
\end{align}
Rearranging the above \cref{eq:vsubvstar}, we have 
\begin{align}\label{eq:vtpulsones}
    \mathbb{E}&\big[\|v_{t+1}- v_{t+1}^*\|^2 - \|v_{t} - v_{t}^*\|^2\big] \nonumber\\
    &\leq -\frac{\lambda_t L_{\mu_g}}{2}\mathbb{E}\|v_{t} - v_{t}^*\|^2 
    + \frac{2}{\lambda_t L_{\mu_g}}\lambda_t^2\mathbb{E}\|e_t^R\|^2 \nonumber\\
    &\quad \ + \frac{2}{\lambda_t L_{\mu_g}}\left( \frac{2L^2_{f_y}}{\mu^2_g} + \frac{2C_{f_y}L^2_{g_{yy}}}{\mu^4_g} \right)\Big[ \alpha_t^2\mathbb{E}\|h_t^f\|^2 + \beta_t^2\left(2\mathbb{E}\|e_t^g\|^2 + 2\mathbb{E}\|\nabla_y g(x_t, y_t)\|^2 \right) \Big].
\end{align}
Using $\lambda_t = c_\lambda\alpha_t$ with $c_\lambda$ in \cref{def:parameters}, and multiplying both sides of \cref{eq:vtpulsones} by $K_2$, we have  
\begin{align}\label{ieq:vv}
K_2\mathbb{E}&\big[\|v_{t+1} - v_{t+1}^*\|^2-\|v_{t} - v_{t}^*\|^2\big] \nonumber\\ 
&\leq -2C_{g_{xy}} \alpha_t \mathbb{E}\|v_{t} - v_{t}^*\|^2 
+ \frac{8C_{g_{xy}}^2}{L_{\mu_g}^2}\alpha_t\mathbb{E}\|e_t^R\|^2 + \frac{\alpha_t}{32}\mathbb{E}\|h_t^f\|^2 \nonumber\\
&\quad \ + \frac{L^2_{f_y}+\frac{C_{f_y}L^2_{g_{xy}}}{\mu_g^2}}{4L_{\mu_g}^2}\alpha_t\mathbb{E}\|e_t^g\|^2 + \frac{L^2_{f_y}+\frac{C_{f_y}L^2_{g_{xy}}}{\mu_g^2}}{4L_{\mu_g}(\mu_g+L_g)}\alpha_t\mathbb{E}\|\nabla_y g(x_t, y_t)\|^2.
\end{align}

\noindent{\bf Step 3. Bound $\mathbb{E}\|e_t^f\|^2$ in \cref{def:potentialF}.} 

\noindent Next, we obtain  from~\Cref{lm:errorf} that 
\begin{align}\label{eq:eftplus1}
\frac{\mathbb{E}\|e^f_{t+1}\|^2}{\alpha_{t}} - \frac{\mathbb{E}\|e^f_{t}\|^2}{\alpha_{t-1}}
&\leq \left[\frac{(1-\eta_{t+1})^2}{\alpha_t} - \frac{1}{\alpha_{t-1}}\right]\mathbb{E}\|e^f_{t}\|^2 + \frac{2(\eta^f_{t+1})^2}{\alpha_t}\sigma^2_f + 6L_F^2\alpha_t \mathbb{E}\|h_t^f\|^2 \nonumber\\
&\quad \ + \frac{12L_F^2\beta_t^2}{\alpha_t}\mathbb{E}\|e_t^g\|^2 + \frac{12L_F^2\beta_t^2}{\alpha_t}\mathbb{E}\|\nabla_y g(x_t,y_t)\|^2 \nonumber\\
&\quad \ + 12C_{g_{xy}}\frac{\lambda_t^2}{\alpha_t}\Big(\mathbb{E}\|e_t^R\|^2 + L_g^2\mathbb{E}\|v_t-v_t^*\|^2\Big),
\end{align}
where 
the inequality follows from the fact that $0 < 1-\eta_t < 1$ for all $t \in \{0,1,...,T-1\}$. 
Now considering the coefficient of the first term on the right-hand side of the above \cref{eq:eftplus1}, we have
\begin{align}\label{eq:fracalpha}
    \frac{(1-\eta^f_{t+1})^2}{\alpha_t} - \frac{1}{\alpha_{t-1}} \leq \frac{1}{\alpha_{t}} - 
    \frac{\eta^f_{t+1}}{\alpha_t} - \frac{1}{\alpha_{t-1}}.
\end{align}
Using the definition of $\alpha_t$ in \cref{def:parameters}, we have
\begin{align}\label{eq:at_at-1}
    \frac{1}{\alpha_{t}} - \frac{1}{\alpha_{t-1}} &= (w+t)^{1/3} - (w+t-1)^{1/3} \overset{(a)}{\leq} \frac{1}{3(w+t-1)^{2/3}} \overset{(b)}{\leq} \frac{1}{3(w/2+t)^{2/3}}\nonumber \\
    &=  \frac{2^{2/3}}{3(w+2t)^{2/3}} \leq \frac{2^{2/3}}{3(w+t)^{2/3}} \overset{(c)}{\leq} \frac{2^{2/3}}{3}\alpha_t^2 \overset{(d)}{\leq} \frac{\alpha_t}{3L_f}, 
\end{align}
where (a) follows from $(x+y)^{1/3} - x^{1/3} \leq y/(3x^{2/3})$, (b) 
follows because 
we choose $w \geq 2$, 
(c) follows 
from the definition of $\alpha_t$ and (d) follows because 
we choose $\alpha_t \leq 1/3L_f$. Substituting 
\cref{eq:at_at-1} 
into~\cref{eq:fracalpha_free} and using $\eta_{t+1}^f = c_{\eta_f}\alpha^2_t$, we have 
\begin{align}\label{eq:1subetaf}
    \frac{(1-\eta^f_{t+1})^2}{\alpha_t} - \frac{1}{\alpha_{t-1}} \leq \frac{\alpha_t}{3L_f} - c_{\eta_f}\alpha_t \leq - \bar{c}_{\eta_f}\alpha_t,
\end{align}
where the inequalities follow from 
$c_{\eta_f} = \frac{1}{3L_f} + \bar{c}_{\eta_f}$
with $\bar{c}_{\eta_f}$ in \cref{def:parameters}. 
Then, substituting \cref{eq:1subetaf} into \cref{eq:eftplus1} yields 
\begin{align}\label{ieq:errorf}
\frac{1}{\bar{c}_{\eta_f}} \mathbb{E}&\bigg[\frac{\|e^f_{t+1}\|^2}{\alpha_t} - \frac{\|e^f_{t}\|^2}{\alpha_{t-1}}\bigg]\nonumber
\\\leq &-\alpha_t\mathbb{E}\|e^f_{t}\|^2 + \frac{2(\eta^f_{t+1})^2}{\bar{c}_{\eta_f}\alpha_t}\sigma_f^2 + \frac{\alpha_t}{16}\mathbb{E}\|h_t^f\|^2 + \frac{L^2_{f_y}+\frac{C_{f_y}L^2_{g_{xy}}}{\mu_g^2}}{4L_{\mu_g}^2}\alpha_t\mathbb{E}\|e_t^g\|^2  \nonumber
\\
& + \frac{L^2_{f_y}+\frac{C_{f_y}L^2_{g_{xy}}}{\mu_g^2}}{4L_{\mu_g}(\mu_g+L_g)}\alpha_t\mathbb{E}\|\nabla_y g(x_t, y_t)\|^2 + \frac{8C_{g_{xy}}}{L^2_{\mu_g}}\alpha_t \mathbb{E}\|e_t^R\|^2 + C_{g_{xy}}\alpha_t\mathbb{E}\|v_t - v_t^*\|^2.
\end{align}

\noindent{\bf Step 4. Bound $\mathbb{E}\|e_t^g\|^2$ in \cref{def:potentialF}.} 

\noindent Next, from~\Cref{lm:errorg}, we have 
\begin{align}\label{eq:etplus1g}
\frac{\mathbb{E}\|e_{t+1}^g\|^2}{\alpha_t} - \frac{\mathbb{E}\|e_{t}^g\|^2}{\alpha_{t-1}}
&\leq \left[ \frac{(1-\eta^g_{t+1})^2 + 32(1-\eta^g_{t+1})^2L^2_g\beta^2_t}{\alpha_t} - \frac{1}{\alpha_{t-1}}\right]\mathbb{E}\|e_t^g\|^2 + \frac{2(\eta^g_{t+1})^2}{\alpha_t}\sigma^2_g\nonumber \\
&\quad \ +16L_g^2\alpha_t\mathbb{E}\|h_t^f\|^2 + \frac{32L_g^2\beta_t^2}{\alpha_t}\mathbb{E}\|\nabla_y g(x_t,y_t)\|^2,
\end{align}
where we use the 
fact that $0<1-\eta_t^g \leq 1$for all $t\in\{0,1,...,T-1\}$. 
Let us consider the coefficient of the first term on the right hand side 
of the above \cref{eq:etplus1g}. In specific, 
we have
\begin{align*}
\frac{(1-\eta^g_{t+1})^2 + 32(1-\eta^g_{t+1})^2L^2_g\beta^2_t}{\alpha_t} - \frac{1}{\alpha_{t-1}} &\leq \frac{(1-\eta^g_{t+1})}{\alpha_t}(1+32L^2_g\beta^2_t) - \frac{1}{\alpha_{t-1}} \\
&= \frac{1}{\alpha_t} - \frac{1}{\alpha_{t-1}} + \frac{32L^2_g\beta^2_t}{\alpha_t} - c_{\eta_g}\alpha_t(1+32L^2_g\beta^2_t),
\end{align*}
which, combined with \cref{eq:at_at-1} that $\frac{1}{\alpha_t} - \frac{1}{\alpha_{t-1}}\leq \frac{\alpha_t}{3L_f}$ and the definition of $\beta_t = c_{\beta}\alpha_t$, yields
\begin{align}\label{eq:etatplusoneg}
\frac{(1-\eta^g_{t+1})^2 + 32(1-\eta^g_{t+1})^2L^2_g\beta^2_t}{\alpha_t} - \frac{1}{\alpha_{t-1}} \leq \frac{\alpha_t}{3L_f} + 32L_g^2c_{\beta}^2\alpha_t - c_{\eta_g}\alpha_t.
\end{align}
Recall $\bar{c}_{\eta_g}$ from \cref{def:parameters} that we choose, then we have
\begin{align*}
&c_{\eta_g} = \frac{1}{3L_f} + 32L_g^2c_{\beta}^2 + \frac{17(L^2_{f_y}+\frac{C_{f_y}L^2_{g_{xy}}}{\mu_g^2})}{L^2_{\mu_g}} \bar{c}_{\eta_g},
\end{align*}
which, in conjunction with \cref{eq:etatplusoneg}, yields 
\begin{align}\label{eq:finaletas}
\frac{(1-\eta^g_{t+1})^2 + 32(1-\eta^g_{t+1})^2L^2_g\beta^2_t}{\alpha_t} - \frac{1}{\alpha_{t-1}} \leq -\frac{17(L^2_{f_y}+\frac{C_{f_y}L^2_{g_{xy}}}{\mu_g^2})}{L^2_{\mu_g}}\bar{c}_{\eta_g}\alpha_t.
\end{align}
Substituting \cref{eq:finaletas} into \cref{eq:etplus1g} yields 
\begin{align}\label{ieq:errorg}
    \frac{1}{\bar{c}_{\eta_g}}\left[\frac{\mathbb{E}\|e_{t+1}^g\|^2}{\alpha_t} - \frac{\mathbb{E}\|e_{t}^g\|^2}{\alpha_{t-1}}\right] 
    &\leq -\frac{17(L^2_{f_y}+\frac{C_{f_y}L^2_{g_{xy}}}{\mu_g^2})}{L^2_{\mu_g}}\alpha_t\mathbb{E}\|e_t^g\|^2 + \frac{2(\eta^g_{t+1})^2}{\bar{c}_{\eta_g}\alpha_t}\sigma^2_g \nonumber\\
    &\quad \ + \frac{\alpha_t}{16}\mathbb{E}\|h_t^f\|^2
    + \frac{L^2_{f_y}+\frac{C_{f_y}L^2_{g_{xy}}}{\mu_g^2}}{4L_{\mu_g}(\mu_g+L_g)}\alpha_t\mathbb{E}\|\nabla_y g(x_t, y_t)\|^2.
\end{align}

\noindent{\bf Step 5. Bound $\mathbb{E}\|e_t^R\|^2$ in \cref{def:potentialF}.} 

\noindent Next, from~\Cref{lm:errorR}, we have 
\begin{align}\label{eq:etRRs}
    \frac{\mathbb{E}\|e_{t+1}^R\|^2}{\alpha_t} - \frac{\mathbb{E}\|e_{t}^R\|^2}{\alpha_{t-1}} &\leq \Bigg[\frac{(1-\eta_{t+1}^R)^2(1 + 48L_g^2\lambda_t^2)}{\alpha_t} - \frac{1}{\alpha_{t-1}}\Bigg]\mathbb{E}\|e_{t}^R\|^2 \nonumber \\
    &\quad +\Bigg[\frac{4(\eta^R_{t+1})^2(\sigma^2_{g_{yy}}r^2_v+\sigma^2_{f_y})}{\alpha_t}\Bigg] + 48(1-\eta_{t+1}^R)^2(\sigma^2_{g_{yy}}r^2_v+\sigma^2_{f_y})\alpha_t\mathbb{E}\|h_t^f\|^2 \nonumber \\
    &\quad +48(1-\eta_{t+1}^R)^2(\sigma^2_{g_{yy}}r^2_v+\sigma^2_{f_y})c^2_\beta \alpha_t(2\mathbb{E}\|e_t^g\|^2 + 2\mathbb{E}\|\nabla_y g(x_t,y_t)\|^2) \nonumber \\
    &\quad +48(1-\eta_{t+1}^R)^2L^4_g c_\lambda^2\alpha_t \mathbb{E}\|v_t-v^*\|^2.
\end{align}
For the first term of the right-hand side of \cref{eq:etRRs}, we have 
\begin{align}\label{eq:firscoeefient}
\frac{(1-\eta_{t+1}^R)^2(1 + 48L_g^2\lambda_t^2 )}{\alpha_t} - \frac{1}{\alpha_{t-1}} 
&\leq \frac{1-\eta_{t+1}^R}{\alpha_t}(1 + 48L_g^2\lambda_t^2 ) - \frac{1}{\alpha_{t-1}} \nonumber\\
&= \frac{1}{\alpha_t} - \frac{1}{\alpha_{t-1}} - \frac{\eta_{t+1}^R}{\alpha_t} + \frac{1-\eta_{t+1}^R}{\alpha_t}\cdot48L_g^2\lambda_t^2 \nonumber\\
&\overset{(a)}{=}\frac{1}{\alpha_t} - \frac{1}{\alpha_{t-1}} - c_{\eta_R}\alpha_t + \left(\frac{1}{\alpha_t} - c_{\eta_R}\alpha_t\right)\cdot48L_g^2c_\lambda^2\alpha_t^2\nonumber\\
&\overset{(b)}{\leq} \frac{\alpha_t}{3L_f} + 48L_g^2c_\lambda^2\alpha_t - c_{\eta_R}\alpha_t,
\end{align}
where (a) follows from the definition that $\eta^R_{t+1} = c_{\eta_R}\alpha_t^2$, and (b) follows from \cref{eq:at_at-1} that $\frac{1}{\alpha_t} - \frac{1}{\alpha_{t-1}} \leq \frac{\alpha_t}{3L_f}$.
Recall $\bar{c}_{\eta_R}$ from \cref{def:parameters} that 
$$
c_{\eta_R} = \frac{1}{3L_f} + 48L_g^2c_\lambda^2 + \frac{16C_{g_{xy}}}{L_{\mu_g}^2}\bar{c}_{\eta_R},
$$
which, in conjunction with \cref{eq:firscoeefient}, yields
\begin{align}\label{eq:intermidiatResult}
\frac{(1-\eta_{t+1}^R)^2(1 + 48L_g^2\lambda_t^2 )}{\alpha_t} - \frac{1}{\alpha_{t-1}} \leq -\frac{16C_{g_{xy}}}{L_{\mu_g}^2}\bar{c}_{\eta_R}\alpha_t.
\end{align}
Incorporating \cref{eq:intermidiatResult} into \cref{eq:etRRs}, 
recalling from \cref{def:parameters} that  $\sigma_R := \sqrt{\sigma^2_{g_{yy}}r^2_v+\sigma^2_{f_y}}$, 
and multiplying both sides of \cref{eq:etRRs}
by $\frac{1}{\bar{c}_{\eta_R}}$, we have 
\begin{align}\label{ieq:errorR}
\frac{1}{\bar{c}_{\eta_R}}\mathbb{E}\left[\frac{\|e^R_{t+1}\|^2}{\alpha_t} - \frac{\|e^R_{t+1}\|^2}{\alpha_{t-1}}\right]
&\leq -\frac{16C_{g_{xy}}}{L_{\mu_g}^2} \alpha_t \mathbb{E}\|e^R_t\|^2 + \frac{4(\eta_{t+1}^R)^2}{\bar{c}_{\eta_R}\alpha_t}\sigma_R^2 
+\frac{\alpha_t}{16}\mathbb{E}\|h_{t}^f\|^2 \nonumber\\
&\ +\frac{L^2_{f_y}+\frac{C_{f_y}L^2_{g_{xy}}}{\mu_g^2}}{4L_{\mu_g}^2}\alpha_t \mathbb{E}\|e_t^g\|^2 + \frac{L^2_{f_y}+\frac{C_{f_y}L^2_{g_{xy}}}{\mu_g^2}}{4L_{\mu_g}(\mu_g+L_g)}\alpha_t\mathbb{E}\|\nabla_y g(x_t, y_t)\|^2 \nonumber\\
& \ +C_{g_{xy}}\alpha_t\mathbb{E}\|v_t-v_t^*\|^2.
\end{align}

\noindent{\bf Step 6. Merging the results of Step 1-5 to prove \cref{def:potentialF}.} 

\noindent Finally, adding~\cref{ieq:yy}, ~\cref{ieq:vv}~\cref{ieq:errorf}, ~\cref{ieq:errorg}, ~\cref{ieq:errorR} and the result of~\Cref{lm:corefunction} with $\alpha_t \leq \frac{1}{3L_f}$ yields 
\begin{align}
\mathbb{E}[V_{t+1} - V_{t}] \leq -\frac{\alpha_t}{2}\mathbb{E}\|\nabla \Phi(x_t)\|^2 + \frac{2(\eta^f_{t+1})^2}{\bar{c}_{\eta_f} \alpha_t}\sigma_f^2 + \frac{2(\eta^g_{t+1})^2}{\bar{c}_{\eta_g} \alpha_t}\sigma_g^2 + 
\frac{4(\eta^R_{t+1})^2}{\bar{c}_{\eta_R} \alpha_t}\sigma_R^2. 
\end{align}
Then, the proof is complete. 
\end{proof}

\subsection{Proof of~\Cref{th:main}}
\begin{proof}
Summing the result of~\Cref{lm:mergeV} for $t = 0$ to $T-1$, dividing by $T$ on both sides and using the definitions that  $\eta_{t+1}^f := c_{\eta_f}\alpha_t^2$, $\eta_{t+1}^g := c_{\eta_g}\alpha_t^2$, $\eta_{t+1}^R := c_{\eta_R}\alpha_t^2$, we have 
\begin{align}\label{eq:sumV}
    \frac{\mathbb{E}[V_{T} - V_{0}]}{T} \leq& -\frac{1}{T}\sum_{t=0}^{T-1}\frac{\alpha_t}{2}\mathbb{E}\|\nabla \Phi(x_t)\|^2 \nonumber
    \\&+ \frac{1}{T}\left[ \frac{2(c_{\eta_f})^2}{\bar{c}_{\eta_f}}\sigma_f^2 + \frac{2(c_{\eta_g})^2}{\bar{c}_{\eta_g}}\sigma_g^2 + \frac{4(c_{\eta_R})^2}{\bar{c}_{\eta_R}}\sigma_R^2 \right]\sum_{t=0}^{T-1}\alpha_t^3.
\end{align}
Next based on the definition of $\alpha_t$ in \cref{def:parameters}, we have 
\begin{align}\label{eq:sumalpha}
    \sum_{t=0}^{T-1}\alpha_t^3 = \sum_{t=0}^{T-1}\frac{1}{w + t}
    \overset{(a)}{\leq} \sum_{t=0}^{T-1}\frac{1}{1 + t} \leq \log(T+1) 
\end{align}
where inequality (a) results from the fact that we choose $w \geq 1$. 
By plugging ~\cref{eq:sumalpha} in~\cref{eq:sumV}, we have 
\begin{align}\label{eq:methEvt}
    \frac{\mathbb{E}[V_{T} - V_{0}]}{T} \leq -\frac{1}{T}\sum_{t=0}^{T-1}\frac{\alpha_t}{2}\mathbb{E}\|\nabla \Phi(x_t)\|^2
    + \left[ \frac{2c_{\eta_f}^2}{\bar{c}_{\eta_f}}\sigma_f^2 
    + \frac{2c_{\eta_g}^2}{\bar{c}_{\eta_g}}\sigma_g^2 + \frac{4c_{\eta_R}^2}{\bar{c}_{\eta_R}}\sigma_R^2 \right] \frac{\log(T+1)}{T}.
\end{align}
Rearrange the terms in \cref{eq:methEvt}, we have 
\begin{align*}
    \frac{1}{T}\sum_{t=0}^{T-1}\frac{\alpha_t}{2}\mathbb{E}\|\nabla \Phi(x_t)\|^2 \leq \frac{\mathbb{E}[V_0 - l^*]}{T}
    + \left[ \frac{2c_{\eta_f}^2}{\bar{c}_{\eta_f}}\sigma_f^2 
    + \frac{2c_{\eta_g}^2}{\bar{c}_{\eta_g}}\sigma_g^2 + \frac{4c_{\eta_R}^2}{\bar{c}_{\eta_R}}\sigma_R^2 \right]\frac{\log(T+1)}{T},
\end{align*}
which, in conjunction with the fact that $\alpha_t$ is decreasing w.r.t.~$t$ and multiplying by $2/\alpha_T$
on both sides, yields 
\begin{align}\label{eq:1Tephis}
    \frac{1}{T}\sum_{t=0}^{T-1}\mathbb{E}\|\nabla \Phi(x_t)\|^2 \leq \frac{2\mathbb{E}[V_0 - l^*]}{\alpha_T T}
    + \left[ \frac{4c_{\eta_f}^2}{\bar{c}_{\eta_f}}\sigma_f^2 
    + \frac{4c_{\eta_g}^2}{\bar{c}_{\eta_g}}\sigma_g^2 + \frac{8c_{\eta_R}^2}{\bar{c}_{\eta_R}}\sigma_R^2 \right]\frac{\log(T+1)}{\alpha_T T}.
\end{align}
Finally, based on the definition of the potential function, we have 
\begin{align}\label{eq:evo_ss}
    \mathbb{E}[V_0] :=& \mathbb{E} 
    \Big[ 
    \Phi(x_0) + \frac{2L}{3\sqrt{2}L_y}\|y_0 - y^*(x_0)\|^2
    + \frac{4C_B}{C_\lambda L_{\mu_g}}\|v_0 - v^*(x_0, y_0)\|^2 \nonumber
     \\
    &\quad + \frac{1}{\bar{c}_{\eta_f}}\frac{\|e_t^f\|^2}{\alpha_{t-1}}
    + \frac{1}{\bar{c}_{\eta_g}}\frac{\|e_t^g\|^2}{\alpha_{t-1}}+ \frac{1}{\bar{c}_{\eta_R}}\frac{\|e_t^R\|^2}{\alpha_{t-1}} 
    \Big] \nonumber \\
    \leq &
    \Phi(x_0) + \frac{2L}{3\sqrt{2}L_y}\|y_0 - y^*(x_0)\|^2
    + \frac{4C_B}{C_\lambda L_{\mu_g}}\|v_0 - v^*(x_0, y_0)\|^2 \nonumber
    \\&
    + \frac{1}{\bar{c}_{\eta_f}}\frac{\sigma_f^2}{\alpha_{t-1}} 
    + \frac{1}{\bar{c}_{\eta_g}}\frac{\sigma_g^2}{\alpha_{t-1}}
    + \frac{1}{\bar{c}_{\eta_R}}\frac{\sigma_R^2}{\alpha_{t-1}} 
\end{align}
where 
the inequality follows from Assumption~\ref{as:ulf},~\ref{as:llf},~\ref{as:sf},~\cref{lm:B} and the definitions of $h^f_t$, $h^g_t$ and $h^R_t$ in~\cref{def:htf},~\cref{def:htg},~\cref{def:htR}. 
Then, substituting \cref{eq:evo_ss} into \cref{eq:1Tephis} yields 
\begin{align*}
    \frac{1}{T}&\sum_{t=0}^{T-1}\|\nabla \Phi(x_t)\|^2 
    \leq
    \frac{2[\Phi(x_0) - \Phi^*]}{\alpha_t T} 
    + \frac{4L}{3\sqrt{2}L_y}\frac{\|y_0 - y^*(x_0)\|^2}{\alpha_t T}
    + \frac{8C_B}{C_{\lambda}L_{\mu_y}}\frac{\|v_0 - v^*(x_0, y_0)\|^2}{\alpha_t T} \\
    &+ \frac{2}{\alpha_{-1}\alpha_{T}T}\left( \frac{\sigma^2_f}{\bar{c}_{\eta_f}}+\frac{\sigma^2_g}{\bar{c}_{\eta_g}}+\frac{\sigma^2_R}{\bar{c}_{\eta_R}} \right) 
    + \left[ \frac{4c_{\eta_f}^2}{\bar{c}_{\eta_f}}\sigma_f^2 
    + \frac{4c_{\eta_g}^2}{\bar{c}_{\eta_g}}\sigma_g^2 + \frac{8c_{\eta_R}^2}{\bar{c}_{\eta_R}}\sigma_R^2 \right]\frac{\log(T+1)}{\alpha_T T}, 
\end{align*}
which, combined with the definitions of $\alpha_T := \frac{1}{(\omega + T)^{1/3}}$ and $\alpha_{-1} = \alpha_{0}$, yields 
\begin{align*}
    \mathbb{E}\|\nabla \Phi(x_a(T))\|^2 &\leq \widetilde{\mathcal{O}}\Bigg( \frac{\Phi(x_0) - \Phi^*}{T^{2/3}}
    + \frac{\|y_0 - y^*(x_0)\|^2}{T^{2/3}}
    + \frac{\|v_0 - v^*(x_0, y_0)\|^2}{T^{2/3}} 
    \\&\qquad\quad+ \frac{\sigma_f^2}{T^{2/3}}
    +  \frac{\sigma_g^2}{T^{2/3}}
    +  \frac{\sigma_R^2}{T^{2/3}}\Bigg) .
\end{align*}
Finally, the proof is complete. 
\end{proof}

\section{Proof of~\Cref{th:hessianfree}  {\color{blue}(FdeHBO with first order approximation)}}\label{sec:FdeHBOproof}

\subsection{Descent in the function value}
\begin{lemma}\label{lm:corefunction_free}
For non-convex and smooth $\Phi(\cdot)$,with $\widetilde{e}^f_t$ defined as: $\widetilde{e}_t^f:=\widetilde{h}_t^f - \bar{\nabla}f(x_t,y_t,v_t)$, the consecutive iterates of~\Cref{alg:main} satisfy:
\begin{align*}
    \mathbb{E}\left[\Phi(x_{t+1})\right] 
    &\leq \mathbb{E} \bigg[\Phi(x_{t})-\frac{\alpha_t}{2}\|\nabla \Phi(x_t)\|^2-\frac{\alpha_t}{2}(1-\alpha_t L_f)\|\widetilde{h}_t^f\|^2 
    \\&\quad \ + \alpha_t\|\widetilde{e}_t^f\|^2 + 4\alpha_t\Big(L_{f_x}^2 + \frac{C_{f_y}L_{g_{xy}}^2}{\mu_g^2}\Big)\|y_t - y_t^*\|^2 + 2\alpha_tC_{g_{xy}}\|v_t - v_t^*\|^2\bigg]
\end{align*} 
for all $t \in \{0,1,...,T-1\}$.
\end{lemma}
\begin{proof}
    The proof follows the same steps  as in  \Cref{lm:corefunction}. 
\end{proof}

\subsection{Descent in the iterates of the lower level function}
\begin{lemma}\label{lm:yy_free}
Define $e^g_t := h^g_t -\nabla_y g(x_t,y_t)$. Then the iterates of solving the lower-level  problem generated by~\Cref{alg:main_free} satisfy
\begin{align*}
    &\mathbb{E}\|y_{t+1} - y^*_{t+1}\|^2 \nonumber
    \\
    &\leq (1+\gamma_t)(1+\delta_t)\left(1-2\beta_t\frac{\mu_g Lg}{\mu_g + Lg}\right)\mathbb{E}\|y_t-y^*(x_t)\|^2 + \left(1+\frac{1}{\gamma_t}\right)L_y^2\alpha_t^2\mathbb{E}\|\widetilde{h}_t^f\|^2 \nonumber
    \\
    &\quad - (1+\gamma_t)(1+\delta_t)\left(\frac{2\beta_t}{\mu_g + Lg}-\beta_t^2 \right)\mathbb{E}\|\nabla_y g(x_t,y_t)\|^2 + (1+\gamma_t)(1+\frac{1}{\delta_t})\beta^2_t\mathbb{E}\|e_t^g\|^2
\end{align*}
for all $t \in \{0, . . . , T-1\}$ with some $\gamma_t, \delta_t > 0$.
\end{lemma}
\begin{proof}
    The proof follows the same steps as in  Lemma C.2 in \cite{khanduri2021near}. 
\end{proof}

\subsection{Descent in the gradient estimation error of the upper function}
\begin{lemma}[Restatement of \Cref{prop:HFetf}]\label{lm:errorf_free}
For any $\xi$, define $\widetilde{e}^f_t :=\widetilde{h}^f_t - \bar \nabla f(x_t,y_t,v_t)$, $\widetilde{e}_t^J:= \widetilde{J}(x_t, y_t, v_t, \delta_{\epsilon}; \xi) - \nabla^2_{xy}g(x_t,y_t; \xi)v_t$, and $\widetilde{e}_t^R := \widetilde{h}_t^R - \nabla_v R(x_t, y_t, v_t)$. 
Under Assumption~\ref{as:sf}, the iterates of the outer problem generated by~\Cref{alg:main_free} satisfy
\begin{align*}
    \mathbb{E}\|\widetilde{e}^f_{t+1}\|^2 
    &\leq \Big[(1-\eta^f_{t+1})^2+4L_{g_{xy}}r_v^2\delta_{\epsilon}\Big]\mathbb{E}\|\widetilde{e}_t^f\|^2 + 4(\eta^f_{t+1})^2\sigma_f^2 + \big(4L_{g_{xy}}r_v^2\delta_{\epsilon} + 16L^2_{g_{xy}}r_v^4\delta_{\epsilon}^2\big) \\
    &\quad \ +6(1-\eta^f_{t+1})^2\Big[L_F^2\alpha_t^2\mathbb{E}\|\widetilde{h}_t^f\|^2+2L_F^2\beta_t^2\big(\mathbb{E}\|e_t^g\|^2 + \|\nabla_y g(x_t, y_t)\|^2\big) \\
    &\qquad \qquad \qquad \qquad + 2C_{g_{xy}}\lambda^2_t\big(\mathbb{E}\|\widetilde{e}_t^R\|^2 + L_g^2\mathbb{E}\|v_t - v_t^*\|^2\big)\Big]    
\end{align*}
for all $t \in \{0, . . . , T-1\}$ with $L_F$ 
in~\Cref{lm:boundofgradf}.
\end{lemma}
\begin{proof}
Based on the definition of $\bar{\nabla} f(x_t, y_t, v_t;\xi)$ in \cref{def:fffform}, the definition of $\widetilde{\nabla} f(x_t, y_t, v_t;\xi)$ in \Cref{sec:hfbo} and the definition of  $\widetilde{e}_t^J$ above, we have 
\begin{align}\label{eq:etjinf}
    \bar{\nabla} f(x_t, y_t, v_t; \xi) = \widetilde{\nabla} f(x_t, y_t, v_t,\delta_\epsilon; \xi) + \widetilde{e}_t^J
\end{align}
for any data sample $\xi$.
From the definition of $\widetilde{e}_t^f$, we have 
\begin{align}\label{eq:etf_free}
    &\mathbb{E}\|\widetilde{e}_{t+1}^f\|^2 \nonumber\\
    &= \mathbb{E}\|\widetilde{h}^f_{t+1} - \bar \nabla f(x_{t+1},y_{t+1},v_{t+1})\|^2 \nonumber\\
    &\overset{(a)}{=}\mathbb{E}\|\eta^f_{t+1}\widetilde{\nabla}f(x_{t+1},y_{t+1},v_{t+1},\delta_{\epsilon}; \xi_{t+1}) + (1-\eta^f_{t+1})\big(\widetilde{h}_t^f+\widetilde{\nabla}f(x_{t+1},y_{t+1},v_{t+1},\delta_{\epsilon}; \xi_{t+1}) \nonumber\\
    &\quad \quad \ \ - \widetilde{\nabla}f(x_{t},y_{t},v_{t},\delta_{\epsilon}; \xi_{t+1})\big) - \bar{\nabla}f(x_{t+1},y_{t+1},v_{t+1})\|^2 \nonumber\\
    &\overset{(b)}{=}\mathbb{E}\|\eta^f_{t+1}\bar{\nabla}f(x_{t+1},y_{t+1},v_{t+1}; \xi_{t+1}) + (1-\eta^f_{t+1})\big(\widetilde{h}_t^f+\bar{\nabla}f(x_{t+1},y_{t+1},v_{t+1}; \xi_{t+1}) \nonumber\\
    &\quad - \bar{\nabla}f(x_{t},y_{t},v_{t}; \xi_{t+1})\big) - \bar{\nabla}f(x_{t+1},y_{t+1},v_{t+1}) - \big(\widetilde{e}^J_{t+1} - (1-\eta^f_{t+1})\widetilde{e}^J_t\big)\|^2 \nonumber\\
    &\overset{(c)}{=}\mathbb{E}\|(1-\eta^f_{t+1})\widetilde{e}_t^f + \eta^f_{t+1}\big(\bar{\nabla}f(x_{t+1},y_{t+1},v_{t+1}; \xi_{t+1}) - \bar{\nabla}f(x_{t+1},y_{t+1},v_{t+1}) \nonumber\\
    &\quad \quad \ \ \ +(1-\eta^f_{t+1})\Big(\big(\bar{\nabla}f(x_{t+1},y_{t+1},v_{t+1}; \xi_{t+1}) - \bar{\nabla}f(x_{t+1},y_{t+1},v_{t+1})\big) \nonumber\\
    &\quad \quad \ \ \ -\big(\bar{\nabla}f(x_{t},y_{t},v_{t}; \xi_{t+1}) - \bar{\nabla}f(x_{t},y_{t},v_{t}) \big)\Big)
    - \big(\widetilde{e}^J_{t+1} - (1-\eta^f_{t+1})\widetilde{e}^J_t\big)\|^2
    \nonumber\\
    &\overset{(d)}{\leq}\Big[(1-\eta^f_{t+1})^2+2L_{g_{xy}}r_v^2\delta_{\epsilon}\Big]\mathbb{E}\|\widetilde{e}_t^f\|^2 \nonumber
\\&\quad\ +\mathbb{E}\|\eta^f_{t+1}\big(\bar{\nabla}f(x_{t+1},y_{t+1},v_{t+1}; \xi_{t+1}) - \bar{\nabla}f(x_{t+1},y_{t+1},v_{t+1})\big) \nonumber\\
    &\quad\ +(1-\eta^f_{t+1})\Big(\big(\bar{\nabla}f(x_{t+1},y_{t+1},v_{t+1}; \xi_{t+1}) - \bar{\nabla}f(x_{t+1},y_{t+1},v_{t+1}) \big) \nonumber\\
    &\quad\  -\big(\bar{\nabla}f(x_{t},y_{t},v_{t}; \xi_{t+1}) - \bar{\nabla}f(x_{t},y_{t},v_{t}) \big)\Big) - \big(\widetilde{e}^J_{t+1} - (1-\eta^f_{t+1})\widetilde{e}^J_t\big)\|^2 + 2L_{g_{xy}}r_v^2\delta_{\epsilon} \nonumber\\
    &\overset{(e)}{\leq} \Big[(1-\eta^f_{t+1})^2+2L_{g_{xy}}r_v^2\delta_{\epsilon}\Big]\mathbb{E}\|\widetilde{e}_t^f\|^2 + 2L_{g_{xy}}r_v^2\delta_{\epsilon} + 4(\eta^f_{t+1})^2\sigma_f^2 + 4\mathbb{E}\|\widetilde{e}^J_{t+1} - (1-\eta^f_{t+1})\widetilde{e}^J_t\|^2
    \nonumber\\ 
    &\quad \ + 2(1-\eta^f_{t+1})^2\mathbb{E}\|\big(\bar{\nabla}f(x_{t+1},y_{t+1},v_{t+1}; \xi_{t+1}) - \bar{\nabla}f(x_{t+1},y_{t+1},v_{t+1})\big) \nonumber
    \\&\quad\ -\big(\bar{\nabla}f(x_{t},y_{t},v_{t}; \xi_{t+1}) - \bar{\nabla}f(x_{t},y_{t},v_{t}) \big)\|^2 \nonumber\\
    &\leq \Big[(1-\eta^f_{t+1})^2+4L_{g_{xy}}r_v^2\delta_{\epsilon}\Big]\mathbb{E}\|\widetilde{e}_t^f\|^2 + 4L_{g_{xy}}r_v^2\delta_{\epsilon} + 4(\eta^f_{t+1})^2\sigma_f^2 + 16L^2_{g_{xy}}r_v^4\delta_{\epsilon}^2
    \nonumber\\ 
    &\quad \ + 2(1-\eta^f_{t+1})^2\mathbb{E}\|\big(\bar{\nabla}f(x_{t+1},y_{t+1},v_{t+1}; \xi_{t+1}) - \bar{\nabla}f(x_{t+1},y_{t+1},v_{t+1})\big) \nonumber
    \\&\quad\ -\big(\bar{\nabla}f(x_{t},y_{t},v_{t}; \xi_{t+1}) - \bar{\nabla}f(x_{t},y_{t},v_{t})\big)\|^2
\end{align}
where 
(a) uses the definition of $\widetilde{h}^f_{t+1}$ in~\cref{def:htf_free}, 
(b) uses \cref{eq:etjinf}, 
(c) uses the definition that $\widetilde{e}^f_t :=\widetilde{h}^f_t - \bar{\nabla} f(x_t,y_t,v_t)$,
(d) follows because for $\Sigma_{t+1} = \sigma\{y_0, x_0,v_0,...,y_t, x_t, v_t,  y_{t+1}, x_{t+1}, v_{t+1}\}$,
\begin{align*}
    &\mathbb{E}\bigg\langle (1-\eta^f_{t+1})\widetilde{e}_t^f, \big(\bar{\nabla}f(x_{t+1},y_{t+1},v_{t+1}; \xi_{t+1}) - \bar{\nabla}f(x_{t+1},y_{t+1},v_{t+1}) \big)\nonumber\\
    &\quad \quad \ \ \ -\big(\bar{\nabla}f(x_{t},y_{t},v_{t}; \xi_{t+1}) - \bar{\nabla}f(x_{t},y_{t},v_{t}) \big)\Big)
    - \big(\widetilde{e}^J_{t+1} - (1-\eta^f_{t+1})\widetilde{e}^J_t\big)\|^2 | \Sigma_{t+1}\bigg\rangle\\
    &\quad=\mathbb{E}\bigg\langle (1-\eta^f_{t+1})\widetilde{e}_t^f, \mathbb{E}\Big[\big(\bar{\nabla}f(x_{t+1},y_{t+1},v_{t+1}; \xi_{t+1}) - \bar{\nabla}f(x_{t+1},y_{t+1},v_{t+1}) \big)\nonumber\\
    &\quad \quad\ -\big(\bar{\nabla}f(x_{t},y_{t},v_{t}; \xi_{t+1}) - \bar{\nabla}f(x_{t},y_{t},v_{t}) \big)\Big)
    - \big(\widetilde{e}^J_{t+1} - (1-\eta^f_{t+1})\widetilde{e}^J_t\big)\|^2\Big] | \Sigma_{t+1}\bigg\rangle \\
    &\quad=\mathbb{E}\Big\langle (1-\eta^f_{t+1})\widetilde{e}_t^f, -\big(\widetilde{e}^J_{t+1} - (1-\eta^f_{t+1})\widetilde{e}^J_t\big)\Big\rangle \\
    &\quad\leq \sqrt{\mathbb{E}\|\widetilde{e}_t^f\|^2}\cdot\sqrt{\mathbb{E}\|\widetilde{e}^J_{t+1} - (1-\eta^f_{t+1})\widetilde{e}^J_t\|^2} \\
    &\quad\leq \max\{1, \mathbb{E}\|\widetilde{e}_t^f\|^2\}\cdot\sqrt{\mathbb{E}\|\widetilde{e}^J_{t+1} - (1-\eta^f_{t+1})\widetilde{e}^J_t\|^2} \\
    &\quad\leq (1+ \mathbb{E}\|\widetilde{e}_t^f\|^2)\cdot\sqrt{2\mathbb{E}\|\widetilde{e}^J_{t+1}\|^2 + 2\mathbb{E}\|\widetilde{e}^J_t\|^2} \\
    &\quad\leq (1+ \mathbb{E}\|\widetilde{e}_t^f\|^2)\cdot\big(2L_{g_{xy}}r_v^2\delta_{\epsilon}\big)\\
    &\quad = \big(2L_{g_{xy}}r_v^2\delta_{\epsilon}\big)\mathbb{E}\|\widetilde{e}_t^f\|^2 + 2L_{g_{xy}}r_v^2\delta_{\epsilon}, 
\end{align*}
which follows from~\Cref{lm:boundofeJH}, and 
(e) follows from~\cref{lm:B}.

\noindent Next, we bound the last term of~\cref{eq:etf_free} as 
\begin{align}\label{eq:2ndpartofetf_free}
    &2(1-\eta^f_{t+1})^2\mathbb{E}\|\big(\bar{\nabla}f(x_{t+1},y_{t+1},v_{t+1}; \xi_{t+1}) - \bar{\nabla}f(x_{t},y_{t},v_{t}; \xi_{t+1})\big) \nonumber
    \\&\qquad\quad\- \big(\bar{\nabla}f(x_{t+1},y_{t+1},v_{t+1}) - \bar{\nabla}f(x_{t},y_{t},v_{t})\big)\|^2 \nonumber \\
    &\quad \ \overset{(a)}{\leq}2(1-\eta^f_{t+1})^2\mathbb{E}\|\bar{\nabla}f(x_{t+1},y_{t+1},v_{t+1}; \xi_{t+1}) - \bar{\nabla}f(x_{t},y_{t},v_{t}; \xi_{t+1})\|^2 \nonumber \\
    &\quad \ \leq 6(1-\eta^f_{t+1})^2\mathbb{E}\|\bar{\nabla}f(x_{t+1},y_{t+1},v_{t+1}; \xi_{t+1}) - \bar{\nabla}f(x_{t},y_{t+1},v_{t+1}; \xi_{t+1})\|^2 \nonumber \\
    &\quad \ \quad \ + 6(1-\eta^f_{t+1})^2\mathbb{E}\|\bar{\nabla}f(x_{t},y_{t+1},v_{t+1}; \xi_{t+1}) - \bar{\nabla}f(x_{t},y_{t},v_{t+1}; \xi_{t+1})\|^2 \nonumber \\
    &\quad \ \quad \ + 6(1-\eta^f_{t+1})^2\mathbb{E}\|\bar{\nabla}f(x_{t},y_{t},v_{t+1}; \xi_{t+1}) - \bar{\nabla}f(x_{t},y_{t},v_{t}; \xi_{t+1})\|^2 \nonumber \\
    &\quad \ \overset{(b)}{\leq} 6(1-\eta^f_{t+1})^2L_F^2\mathbb{E}\|x_{t+1} - x_t\|^2 + 6(1-\eta^f_{t+1})^2L_F^2\mathbb{E}\|y_{t+1} - y_t\|^2 \nonumber
    \\&\quad\ \quad\ + 6(1-\eta^f_{t+1})^2\mathbb{E}\|\nabla^2_{xy} g(x_t, y_t)(v_{t+1} - v_t)\|^2 \nonumber \\
    &\quad \ \overset{(c)}{\leq} 6(1-\eta^f_{t+1})^2L_F^2\mathbb{E}\|x_{t+1} - x_t\|^2 + 6(1-\eta^f_{t+1})^2L_F^2\mathbb{E}\|y_{t+1} - y_t\|^2 \nonumber
    \\&\quad\ \quad\ +6(1-\eta^f_{t+1})^2C_{g_{xy}}\mathbb{E}\|v_{t+1} - v_t\|^2 \nonumber \\
    &\quad \ \overset{(d)}{\leq} 6(1-\eta^f_{t+1})^2L_F^2\mathbb{E}\|x_{t+1} - x_t\|^2 + 6(1-\eta^f_{t+1})^2L_F^2\mathbb{E}\|y_{t+1} - y_t\|^2 \nonumber
    \\& \quad\ \quad\ + 6(1-\eta^f_{t+1})^2C_{g_{xy}}\mathbb{E}\|w_{t+1} - v_t\|^2 \nonumber \\
    &\quad \ \overset{(e)}{\leq} 6(1-\eta^f_{t+1})^2\Big[L_F^2\big(\alpha_t^2\mathbb{E}\|\widetilde{h}_t^f\|^2+\beta_t^2\mathbb{E}\|h_t^g\|^2\big) + C_{g_{xy}}\lambda^2_t\mathbb{E}\|\widetilde{h}_t^R\|^2\Big] \nonumber \\
    &\quad \ \overset{(f)}{\leq}6(1-\eta^f_{t+1})^2\Big[L_F^2\alpha_t^2\mathbb{E}\|\widetilde{h}_t^f\|^2+2L_F^2\beta_t^2\big(\mathbb{E}\|e_t^g\|^2 + \|\nabla_y g(x_t, y_t)\|^2\big) \nonumber
    \\&\quad\ \quad\ + 2C_{g_{xy}}\lambda^2_t\big(\mathbb{E}\|\widetilde{e}_t^R\|^2 + \mathbb{E}\|\nabla R_v(x_t, y_t, v_t)\|^2\big)\Big] \nonumber \\
    &\quad \ \overset{(g)}{\leq}6(1-\eta^f_{t+1})^2\Big[L_F^2\alpha_t^2\mathbb{E}\|\widetilde{h}_t^f\|^2+2L_F^2\beta_t^2\big(\mathbb{E}\|e_t^g\|^2 + \|\nabla_y g(x_t, y_t)\|^2\big)\nonumber
    \\&\quad\ \quad\ + 2C_{g_{xy}}\lambda^2_t\big(\mathbb{E}\|\widetilde{e}_t^R\|^2 + \mathbb{E}\|v_t - v_t^*\|^2\big)\Big],
\end{align}
where (a) follows from the mean variance inequality: For a random variable $Z$ we have $\mathbb{E}\|Z - \mathbb{E}[Z]\|^2 \leq \mathbb{E}\|Z\|^2$ with $Z$ defined as $Z := \bar{\nabla}f(x_{t+1},y_{t+1},v_{t+1}; \xi_{t+1}) - \bar{\nabla}f(x_{t},y_{t},v_{t}; \xi_{t+1})$, (b) follows from~\Cref{lm:boundofgradf} and~\cref{def:gradR}, (c) uses Assumption~\ref{as:llf}, (d) uses the nonexpansiveness of projection, (e) uses the definition of $\widetilde{h}_t^f$ in~\cref{def:htf_free}, $h_t^g$ in~\cref{def:htg} and $\widetilde{h}_t^R$ in~\cref{def:htR_free}, (f) follows from the definition that $e_t^g := g_t^R - \nabla g(x_t, y_t)$ and $\widetilde{e}_t^R := \widetilde{h}_t^R - \nabla_v R(x_t, y_t, v_t)$, (g) uses the result of~\cref{eq:boundofnablaR}.

\noindent Finally, substituting~\cref{eq:2ndpartofetf_free} in~\cref{eq:etf_free}, we finish the proof.  
\end{proof}

\subsection{Descent in the gradient estimation error of the inner function}
\begin{lemma}\label{lm:errorg_free}
Define $e_t^g := h_t^g - \nabla_y g(x_t, y_t)$. Under Assumption~\ref{as:llf} and~\ref{as:sf}, the iterates generated from~\Cref{alg:main_free} satisfy 
\begin{align*}
    \mathbb{E} \|e_{t+1}^g\|^2 &\leq \left((1-\eta_{t+1}^g)^2 \mathbb{E} \|e_{t}^g\|^2 + 32(1-\eta_{t+1}^g)^2 L_g^2\beta_t^2 \right)\mathbb{E} \|e_{t}^g\|^2 + 2(\eta_{t+1}^g)^2\sigma_g^2 \nonumber\\ &+ 16(1-\eta_{t+1}^g)^2L_g^2 \alpha_t^2\mathbb{E} \|\widetilde{h}_{t}^f\|^2 + 32(1-\eta_{t+1}^g)^2L_g^2 \beta^2_t \mathbb{E}\|\nabla_y g(x_t, y_t)\|^2
\end{align*}
for all $t\in \{0,1,...,T-1\}$.
\end{lemma}
\begin{proof}
    The proof follows the same steps as in  \Cref{lm:errorg}.
\end{proof}

\subsection{Descent in the gradient estimation error of the R function}
\begin{lemma}[Restatement of \Cref{prop:hfetR}]\label{lm:errorR_free}
For any $\psi$, define $\widetilde{e}_t^R := \widetilde{h}_t^R - \nabla_v R(x_t, y_t, v_t)$ and  $\widetilde{e}_t^H:= \widetilde{H}(x_t, y_t, v_t, \delta_{\epsilon}; \psi) - \nabla^2_{yy}g(x_t,y_t; \psi)v_t$. Under Assumption~\ref{as:ulf},~\ref{as:llf},~\ref{as:sf}, the iterates generated by~\Cref{alg:main_free} satisfy 
\begin{align*}
\mathbb{E}\|\widetilde{e}^R_{t+1}\|^2 
&\leq \big[(1-\eta^R_{t+1})^2(1+96L_g^2\lambda^2_t) + 4L_{g_{yy}}r^2_v\delta_{\epsilon}\big] \mathbb{E}\|\widetilde{e}_t^R\|^2
+ \big(4L_{g_{yy}}r^2_v\delta_{\epsilon} + 8L^2_{g_{yy}}r_v^4\delta_{\epsilon}^2\big) \\ 
&\quad \ + 96(1-\eta_{t+1}^R)^2\big(L^2_{g_{yy}}r^2_v + L^2_{f_y}\big) \left[\alpha_t^2\mathbb{E}\|\widetilde{h}_{t}^f\|^2 + 2\beta_t^2(\mathbb{E}\|e_t^g\|^2 + \mathbb{E}\|\nabla_y g(x_t,y_t)\|^2)\right] \\
&\quad \ + 96(1-\eta_{t+1}^R)^2L^2_g\lambda^2_t \left(\mathbb{E}\|\widetilde{e}_t^R\|^2 + L_g^2\mathbb{E}\|v_t-v_t^*\|^2\right) + 8(\eta_{t+1}^R)^2(\sigma^2_{g_{yy}}r_v^2+ \sigma^2_{f_y}),
\end{align*}
for all $t\in \{0,1,...,T-1\}$. 
\end{lemma}
\begin{proof}
From the definition of $\nabla_v R(x_t, y_t, v_t; \psi_t)$ in~\cref{def:gradR}, the definition of $\widetilde{\nabla}_v R(x_t, y_t, v_t,\delta_{\epsilon}; \psi_t)$ in~\cref{def:gradR_free} and the definition of $\widetilde{e}_t^H$, we have
\begin{align}\label{eq:gradR_diff}
    \widetilde{\nabla}_v R&(x_t, y_t, v_t,\delta_{\epsilon};\psi) - \nabla_v R(x_t, y_t, v_t;\psi) \nonumber
    \\&= \widetilde{H}(x_t, y_t, v_t, \delta_{\epsilon};\psi) - \nabla^2_{yy}g(x_t,y_t;\psi)v_t = \widetilde{e}_t^H
\end{align}
for any data sample $\psi$.
For the gradient estimation error of the R function, we have  
\begin{align}\label{ieq:etR1_free}
&\mathbb{E}\|\widetilde{e}^R_{t+1}\|^2 \nonumber\\
& = \mathbb{E}\|\widetilde{h}^R_{t+1} - \nabla_v R(x_{t+1}, y_{t+1}, v_{t+1})\|^2 \nonumber\\
& \overset{(a)}{=} \mathbb{E}\|\widetilde{\nabla}_v R(x_{t+1}, y_{t+1}, v_{t+1},\delta_{\epsilon}; \psi_{t+1}) + (1-\eta^R_{t+1})\widetilde{h}_t^R - (1-\eta^R_{t+1})\widetilde{\nabla}_v R(x_{t}, y_{t}, v_{t},\delta_{\epsilon}; \psi_{t+1}) \nonumber\\
& \qquad \ \ - \nabla_v R(x_{t+1}, y_{t+1}, v_{t+1})\|^2 \nonumber\\
& \overset{(b)}{=} \mathbb{E}\|\nabla_v R(x_{t+1}, y_{t+1}, v_{t+1}; \psi_{t+1}) + (1-\eta^R_{t+1})\widetilde{h}_t^R - (1-\eta^R_{t+1})\nabla_v R(x_{t}, y_{t}, v_{t}; \psi_{t+1}) \nonumber\\
& \qquad \ \ - \nabla_v R(x_{t+1}, y_{t+1}, v_{t+1}) + \big(\widetilde{e}_{t+1}^H - (1-\eta^R_{t+1})\widetilde{e}_{t}^H\big)\|^2 \nonumber\\
& \overset{(c)}{=} \mathbb{E}\|(1-\eta^R_{t+1})\widetilde{e}_t^R + (1-\eta^R_{t+1})\nabla_v R(x_{t}, y_{t}, v_{t})+ \nabla_v R(x_{t+1}, y_{t+1}, v_{t+1}; \psi_{t+1}) \nonumber\\
& \qquad \ \  - (1-\eta^R_{t+1}) \nabla_v R(x_{t}, y_{t}, v_{t}; \psi_{t+1}) - \nabla_v R(x_{t+1}, y_{t+1}, v_{t+1}) + \big(\widetilde{e}_{t+1}^H - (1-\eta^R_{t+1})\widetilde{e}_{t}^H\big)\|^2 \nonumber\\
& \overset{(d)}{\leq}\big[(1-\eta^R_{t+1})^2 + 4L_{g_{yy}}r^2_v\delta_{\epsilon}\big] \mathbb{E}\|\widetilde{e}_t^R\|^2
+ 4L_{g_{yy}}r^2_v\delta_{\epsilon} + \mathbb{E}\|(1-\eta^R_{t+1})\nabla_v R(x_{t}, y_{t}, v_{t}) \nonumber\\ 
& \quad \ - (1-\eta^R_{t+1})\nabla_v R(x_{t}, y_{t}, v_{t};\psi_{t+1}) 
+ \nabla_v R(x_{t+1}, y_{t+1}, v_{t+1};\psi_{t+1}) - \nabla_v R(x_{t+1}, y_{t+1}, v_{t+1}) \nonumber\\ 
& \quad \ + \big(\widetilde{e}_{t+1}^H - (1-\eta^R_{t+1})\widetilde{e}_{t}^H\big)\|^2 \nonumber\\ 
& \leq \big[(1-\eta^R_{t+1})^2 + 4L_{g_{yy}}r^2_v\delta_{\epsilon}\big] \mathbb{E}\|\widetilde{e}_t^R\|^2
+ 4L_{g_{yy}}r^2_v\delta_{\epsilon} + 2\mathbb{E}\|(1-\eta^R_{t+1})\nabla_v R(x_{t}, y_{t}, v_{t}) \nonumber\\ 
& \ - (1-\eta^R_{t+1})\nabla_v R(x_{t}, y_{t}, v_{t};\psi_{t+1}) 
+ \nabla_v R(x_{t+1}, y_{t+1}, v_{t+1};\psi_{t+1}) - \nabla_v R(x_{t+1}, y_{t+1}, v_{t+1})\|^2 \nonumber\\ 
& \quad \ + 2\mathbb{E}\|\widetilde{e}_{t+1}^H - (1-\eta^R_{t+1})\widetilde{e}_{t}^H\|^2 \nonumber\\ 
& \overset{(e)}{\leq} \big[(1-\eta^R_{t+1})^2 + 4L_{g_{yy}}r^2_v\delta_{\epsilon}\big] \mathbb{E}\|\widetilde{e}_t^R\|^2
+ 4L_{g_{yy}}r^2_v\delta_{\epsilon} + 2\mathbb{E}\|(1-\eta^R_{t+1})\nabla_v R(x_{t}, y_{t}, v_{t}) \nonumber\\ 
&  \ - (1-\eta^R_{t+1})\nabla_v R(x_{t}, y_{t}, v_{t};\psi_{t+1}) 
+ \nabla_v R(x_{t+1}, y_{t+1}, v_{t+1};\psi_{t+1}) - \nabla_v R(x_{t+1}, y_{t+1}, v_{t+1})\|^2 \nonumber\\ 
& \quad \ + 8L^2_{g_{yy}}r_v^2\delta_{\epsilon}^2,
\end{align} 
where (a) follows from the definition of $\widetilde{h}_t^R$ in~\cref{def:htR_free},
(b) follows from~\cref{eq:gradR_diff},
(c) uses the definition of $\widetilde{e}_t^R := \widetilde h_t^R - \nabla_v R(x_t, y_t, v_t)$,
(d) follows from the fact that  
\begin{align*}
    &\mathbb{E}\Big\langle (1-\eta^R_{t+1})\widetilde{e}_t^R, \big(\nabla_v R(x_{t+1}, y_{t+1}, v_{t+1}; \psi_{t+1}) - \nabla_v R(x_{t+1}, y_{t+1}, v_{t+1}) \\
    &\qquad\quad -(1-\eta_{t+1}^R)\big(\nabla_v R(x_{t}, y_{t}, v_{t}; \psi_{t+1}) - \nabla_v R(x_{t}, y_{t}, v_{t})\big)+ \big(\widetilde{e}_{t+1}^H - (1-\eta^R_{t+1})\widetilde{e}_{t}^H\big)\Big\rangle \\
    &\quad \ \ =\mathbb{E}\Big\langle (1-\eta^R_{t+1})\widetilde{e}_t^R, \mathbb{E}\big[\big(\nabla_v R(x_{t+1}, y_{t+1}, v_{t+1}; \psi_{t+1}) - \nabla R(x_{t+1}, y_{t+1}, v_{t+1})\big) \\
    &\quad\quad  -(1-\eta_{t+1}^R)\big(\nabla_v R(x_{t}, y_{t}, v_{t}; \psi_{t+1}) - \nabla_v R(x_{t}, y_{t}, v_{t}) \big)+ \big(\widetilde{e}_{t+1}^H - (1-\eta^R_{t+1})\widetilde{e}_{t}^H\big)|\Sigma_{t+1}\big]\Big\rangle \\
    &\quad \ \ =\mathbb{E}\Big\langle (1-\eta^R_{t+1})\widetilde{e}_t^R, \widetilde{e}_{t+1}^H - (1-\eta^R_{t+1})\widetilde{e}_{t}^H\Big\rangle \\
    &\quad\ \ \leq \sqrt{\mathbb{E}\|\widetilde{e}_t^R\|^2}\cdot\sqrt{\mathbb{E}\|\widetilde{e}^H_{t+1} - (1-\eta^R_{t+1})\widetilde{e}^H_t\|^2} \\
    &\quad\ \ \leq \max\{1, \mathbb{E}\|\widetilde{e}_t^R\|^2\}\cdot\sqrt{\mathbb{E}\|\widetilde{e}^H_{t+1} - (1-\eta^R_{t+1})\widetilde{e}^H_t\|^2} \\
    &\quad\ \ \leq (1+ \mathbb{E}\|\widetilde{e}_t^R\|^2)\cdot\sqrt{2\mathbb{E}\|\widetilde{e}^H_{t+1}\|^2 + 2\mathbb{E}\|\widetilde{e}^H_t\|^2} \\
    &\quad\ \ \leq (1+ \mathbb{E}\|\widetilde{e}_t^R\|^2)\cdot\big(2L_{g_{yy}}r_v^2\delta_{\epsilon}\big)\\
    &\quad\ \ \leq \big(2L_{g_{yy}}r_v^2\delta_{\epsilon}\big)\mathbb{E}\|\widetilde{e}_t^R\|^2 + 2L_{g_{yy}}r_v^2\delta_{\epsilon}, 
\end{align*}
for $\Sigma_{t+1} = \sigma\{y_0, v_0, x_0,...,y_t, v_t, x_t, y_{t+1}, v_{t+1}, x_{t+1}\}$.  Incorporating~\cref{eq:etR_2ndpart} into~\cref{ieq:etR1_free}, we have 
\begin{align*}
\mathbb{E}\|\widetilde{e}^R_{t+1}\|^2 
&\leq \big[(1-\eta^R_{t+1})^2(1+96L_g^2\lambda^2_t) + 4L_{g_{yy}}r^2_v\delta_{\epsilon}\big] \mathbb{E}\|\widetilde{e}_t^R\|^2
+ 4L_{g_{yy}}r^2_v\delta_{\epsilon} + 8L^2_{g_{yy}}r_v^4\delta_{\epsilon}^2   \\ 
&\quad \ + 96(1-\eta_{t+1}^R)^2\left(L^2_{g_{yy}}r^2_v + L^2_{f_y}\right) \left[\alpha_t^2\mathbb{E}\|\widetilde{h}_{t}^f\|^2 + 2\beta_t^2(\mathbb{E}\|e_t^g\|^2 + \mathbb{E}\|\nabla_y g(x_t,y_t)\|^2)\right] \\
&\quad \ + 96(1-\eta_{t+1}^R)^2L^2_g\lambda^2_t \left(\mathbb{E}\|\widetilde{e}_t^R\|^2 + L_g^2\mathbb{E}\|v_t-v_t^*\|^2\right) + 8(\eta_{t+1}^R)^2(\sigma^2_{g_{yy}}r_v^2+ \sigma^2_{f_y}),
\end{align*}
which finishes the proof. 
\end{proof}

\subsection{Descent in iterates of the LS problem}
\begin{lemma}\label{lm:vv_free}
Define $\widetilde{e}_t^R := \widetilde{h}_t^R - \nabla_v R(x_t, y_t, v_t)$.  
Under the Assumption~\ref{as:ulf},~\ref{as:llf}, the iterates of the LS problem generated according to~\Cref{alg:main_free} satisfy
\begin{align*}
\mathbb{E}\|&v_{t+1} - v_{t+1}^*\|^2  \\
&\leq (1+\gamma_t')\left(1+\delta_t'\right)\left[\left(1-2\lambda_t\frac{(L_g+L_g^3)\mu_g}{\mu_g + L_g} + \lambda_t^2L_g^2\right)\mathbb{E}\|v_t - v^*_t\|^2\right] \\
&\quad \ + (1+\gamma_t')\left(1+\frac{1}{\delta_t'}\right)\lambda_t^2\mathbb{E}\|\widetilde{e}_t^R\|^2 \\
&\quad \ + (1+\frac{1}{\gamma_t'})\left( \frac{2L^2_{f_y}}{\mu^2_g} + \frac{2C^2_{f_y}L^2_{g_{yy}}}{\mu^4_g} \right)\left[ \alpha_t^2\mathbb{E}\|\widetilde{h}_t^f\|^2 + \beta_t^2\left(2\mathbb{E}\|e_t^g\|^2 + 2\mathbb{E}\|\nabla_y g(x_t, y_t)\|^2 \right) \right].
\end{align*}
for all $t \in \{0, . . . , T-1\}$ with some $\gamma_t'>0$ and $ \delta_t' > 0$.
\end{lemma}
\begin{proof}
From~\cref{eq:vgaps_free} we have that there exist $\gamma'_t \geq 0$, $\delta'_t \geq 0$ such that
\begin{align}\label{eq:vgaps_free}
    \mathbb{E}\|v_{t+1} -& v_{t+1}^*\|^2 \leq (1+\gamma_t')(1+\delta_t')\mathbb{E}\|v_t - \lambda_t\nabla_v R(x_t, y_t, v_t) - v_{t}^*)\|^2 \nonumber\\
    &+ (1+\gamma_t')(1+\frac{1}{\delta_t'})\lambda_t^2\mathbb{E}\|\widetilde{h}_t^R - \nabla_v R(x_t, y_t, v_t)\|^2 
     +(1+\frac{1}{\gamma_t'})\mathbb{E}\|v_{t}^* - v_{t+1}^*\|^2. 
\end{align}
Incorporating~\cref{eq:vtlambda} and~\cref{eq:vt1star} into~\cref{eq:vgaps_free}, similarly to~\Cref{lm:vv_free}, we have
\begin{align*}
\mathbb{E}&\|v_{t+1} - v_{t+1}^*\|^2  \\
&\leq (1+\gamma_t')\left(1+\delta_t'\right)\left[\left(1-2\lambda_t\frac{(L_g+L_g^3)\mu_g}{\mu_g + L_g} + \lambda_t^2L_g^2\right)\mathbb{E}\|v_t - v^*_t\|^2\right] \\
&\quad \ + (1+\gamma_t')\left(1+\frac{1}{\delta_t'}\right)\lambda_t^2\mathbb{E}\|\widetilde{e}_t^R\|^2 \\
&\quad \ + (1+\frac{1}{\gamma_t'})\left( \frac{2L^2_{f_y}}{\mu^2_g} + \frac{2C_{f_y}L^2_{g_{yy}}}{\mu^4_g} \right)\left[ \alpha_t^2\mathbb{E}\|\widetilde{h}_t^f\|^2 + \beta_t^2\left(2\mathbb{E}\|e_t^g\|^2 + 2\mathbb{E}\|\nabla_y g(x_t, y_t)\|^2 \right) \right].
\end{align*}
which finishes the proof.
\end{proof}

\subsection{Descent in the Potential Function}
Define the potential function as
\begin{align}\label{def:potentialF_free}
\widetilde{V}_t :=& \Phi(x_{t}) + K_1\|y_{t} - y^*(x_{t})\|^2 + K_2\|v_{t} - v^*(x_{t}, y_{t})\|^2 +  \frac{1}{\bar{c}_{\eta_f}} \frac{\|\widetilde{e}_{t}^f\|^2}{\alpha_{t-1}} + \frac{1}{\bar{c}_{\eta_g}} \frac{\|e_{t}^g\|^2}{\alpha_{t-1}} + \frac{1}{\bar{c}_{\eta_R}} \frac{\|\widetilde{e}_{t}^R\|^2}{\alpha_{t-1}},
\end{align}
where the coefficients are given by
\begin{align}
K_1 &= \frac{8(L^2_{f_y} + \frac{C_{f_y}L^2_{g_{xy}}}{\mu_g^2})}{c_{\beta}L_{\mu_g}}, \quad K_2 = \frac{4C_{g_{xy}}}{c_{\lambda}L_{\mu_g}}; \nonumber \\
\bar{c}_{\eta_f} &= \max\left\{96L_F^2, \ 12L_g^2c_{\lambda}^2, \ \frac{48L_{\mu_g}L^2_fc^2_{\beta}\max\{L_{\mu_g}, \mu_g+L_g\}}{L^2_{f_y}+\frac{C_{f_y}L^2_{g_{xy}}}{\mu_g^2}}, \ \frac{3}{2}L^2_{\mu_g}c^2_\lambda\right\}, \nonumber \\ 
\bar{c}_{\eta_g} &= \max\left\{256L_g^2, \ \frac{128L_{\mu_g}L^2_gc^2_{\beta}\max\{L_{\mu_g}, \mu_g+L_g\}}{L^2_{f_y}+\frac{C_{f_y}L^2_{g_{xy}}}{\mu_g^2}}\right\}, \ 
\nonumber\\
\bar{c}_{\eta_R} &= \max\bigg\{1536(L^2_{g_{yy}}r_v^2 + L^2_{f_y}),\ \frac{96L_g^4c^2_{\lambda}}{C_{g_{xy}}}, \ \frac{768L_{\mu_g}(L^2_{g_{yy}}r_v^2 + L^2_{f_y})c^2_{\beta}\max\{L_{\mu_g}, \mu_g+L_g\}}{L^2_{f_y}+\frac{C_{f_y}L^2_{g_{xy}}}{\mu_g^2}}\bigg\}.\nonumber
\end{align}
\begin{lemma}\label{lm:mergeV_free}
Suppose Assumptions~\ref{as:ulf} ,\ref{as:llf} and \ref{as:sf} are satisfied. 
Choose the parameters of~\Cref{alg:main_free} as 
\begin{align*}
    &\alpha_t := \frac{1}{(w+t)^{1/3}},\  \beta_t := c_\beta \alpha_t,\ \lambda_t := c_\lambda \alpha_t;\ \eta_{t+1}^f := c_{\eta_f}\alpha_t^2,\ \eta_{t+1}^g := c_{\eta_g}\alpha_t^2,\ \eta_{t+1}^R := c_{\eta_R}\alpha_t^2,
\end{align*}
where the constants are given by 
\begin{align}\label{def:parameters_free}
    &c_{\beta} \geq \sqrt{\frac{512L_y^2(L^2_{f_y}+\frac{C_{f_y}L^2_{g_{xy}}}{\mu_g^2})}{L^2_{\mu_g}}}, \nonumber\\
    &c_\lambda \geq \sqrt{\max\left\{\frac{1024C_{g_{xy}}}{L^2_{\mu_g}}\Big(\frac{L^2_{f_y}}{\mu_g^2}+\frac{C^2_{f_y}L^2_{g_{yy}}}{\mu_g^4}\Big), \frac{128(\mu_g+L_g)C_{g_{xy}}}{L_{\mu_g}}c^2_{\beta}, 128C_{g_{xy}}c_\beta^2\right\}}\ ; \nonumber\\
    &c_{\eta_f} = \frac{2}{3L_f} + 2\bar{c}_{\eta_f},\  c_{\eta_g} = \frac{1}{3L_f} + 32L_g^2c_{\beta}^2 + \left[ \frac{17(L^2_{f_y}+\frac{C_{f_y}L^2_{g_{xy}}}{\mu_g^2})}{L^2_{\mu_g}}\right]\bar{c}_{\eta_g}, \nonumber\\ 
    &c_{\eta_R} = \frac{2}{3L_f} + 192L_g^2c_\lambda^2 + \left[ \frac{32C_{g_{xy}}}{L^2_{\mu_g}}\right]\bar{c}_{\eta_R};\ \ \sigma_R := \sqrt{\sigma^2_{g_{yy}}r^2_v+\sigma^2_{f_y}}, \nonumber\\
    &w \geq \left(\max\left\{c_\beta(\mu_g+L_g), \ \frac{c_\lambda(\mu_g+L_g)}{2\mu_gL_g} \right\}\right)^3 - 1,  \nonumber\\
    &\delta_{\epsilon} \leq \min \left\{\frac{c_{\eta_f}}{8(L_{g_{xy}}r_v^2(w+T-1)^{2/3})}, \frac{c_{\eta_R}}{8(L_{g_{xy}}r_v^2(w+T-1)^{2/3})} \right\}.
\end{align}
for all $t \in \{0,1,...,T-1\}$.
Then the iterates generated by~\Cref{alg:main_free} satisfy
\begin{align*}
\mathbb{E}[\widetilde{V}_{t+1} - \widetilde{V}_{t}] &\leq -\frac{\alpha_t}{2}\mathbb{E}\|\nabla \Phi(x_t)\|^2 + \frac{2(\eta^f_{t+1})^2}{\bar{c}_{\eta_f} \alpha_t}\sigma_f^2 + \frac{2(\eta^g_{t+1})^2}{\bar{c}_{\eta_g} \alpha_t}\sigma_g^2 + 
\frac{4(\eta^R_{t+1})^2}{\bar{c}_{\eta_R} \alpha_t}\sigma_R^2 \\
& \quad \ + \frac{1}{\alpha_t\bar{c}_{\eta_f}}\big(4L_{g_{xy}}r_v^2\delta_{\epsilon} + 16L^2_{g_{xy}}r_v^4\delta_{\epsilon}^2\big) + \frac{1}{\alpha_t\bar{c}_{\eta_R}}\big(4L_{g_{yy}}r_v^2\delta_{\epsilon} + 8L^2_{g_{yy}}r_v^4\delta_{\epsilon}^2\big),
\end{align*}
for all $t \in \{0,1,...,T-1\}$. 
\end{lemma}

\begin{proof}
Based on the definition of $\widetilde{V}_t$ in~\cref{def:potentialF_free}, it can be seen that $\widetilde{V}_t$ contains six parts.  We next develop five important inequalities to prove \Cref{lm:mergeV_free}.

\noindent{\bf Step 1. Bound $\mathbb{E}\|y_{t} - y^*(x_{t})\|^2$ in \cref{def:potentialF_free}.}

\noindent
Same as Step 1 in \Cref{lm:mergeV}, by \Cref{lm:yy_free}, we have
\begin{align}\label{ieq:yy_free}
   K_1 \mathbb{E}\big[\|y_{t+1} &- y^*(x_{t+1})\|^2 - \|y_{t} - y^*(x_{t})\|^2\big] \nonumber\\
   &\leq -4\Big(L^2_{f_x} + \frac{C_{f_y}L^2_{g_{xy}}}{\mu_g^2}\Big)\alpha_t
   \mathbb{E}\|y_{t} - y^*(x_{t})\|^2 
   - \frac{L^2_{f_y}+\frac{C_{f_y}L^2_{g_{xy}}}{\mu_g^2}}{L_{\mu_g}(\mu_g+L_g)}\alpha_t\mathbb{E}\|\nabla_y g(x_t,y_t)\|^2 \nonumber\\ 
   &\quad \ +\frac{16(L^2_{f_y}+\frac{C_{f_y}L^2_{g_{xy}}}{\mu_g^2})}{L^2_{\mu_g}}\alpha_t\mathbb{E}\|e_t^g\|^2 
   + \frac{\alpha_t}{32}\mathbb{E}\|h_t^f\|^2.
\end{align} 

\noindent{\bf Step 2. Bound $\mathbb{E}\|v_{t} - v^*(x_{t})\|^2$ in \cref{def:potentialF_free}.} 

\noindent 
Same as Step 1 in \Cref{lm:mergeV}, by \Cref{lm:vv_free}, we have
\begin{align}\label{ieq:vv_free}
K_2\mathbb{E}&\big[\|v_{t+1} - v_{t+1}^*\|^2-\|v_{t} - v_{t}^*\|^2\big] \nonumber\\ 
&\leq -2C_{g_{xy}} \alpha_t \mathbb{E}\|v_{t} - v_{t}^*\|^2 
+ \frac{8C_{g_{xy}}^2}{L_{\mu_g}^2}\alpha_t\mathbb{E}\|\widetilde{e}_t^R\|^2 + \frac{\alpha_t}{32}\mathbb{E}\|h_t^f\|^2 \nonumber\\
&\quad \ + \frac{L^2_{f_y}+\frac{C_{f_y}L^2_{g_{xy}}}{\mu_g^2}}{4L_{\mu_g}^2}\alpha_t\mathbb{E}\|e_t^g\|^2 + \frac{L^2_{f_y}+\frac{C_{f_y}L^2_{g_{xy}}}{\mu_g^2}}{4L_{\mu_g}(\mu_g+L_g)}\alpha_t\mathbb{E}\|\nabla_y g(x_t, y_t)\|^2.
\end{align} 

\noindent{\bf Step 3. Bound $\mathbb{E}\|\widetilde{e}_t^f\|^2$ in \cref{def:potentialF_free}. }

\noindent Next, we obtain  from~\Cref{lm:errorf_free} that 
\begin{align}\label{eq:eftplus1_free}
\frac{\mathbb{E}\|\widetilde{e}^f_{t+1}\|^2}{\alpha_{t}} - &\frac{\mathbb{E}\|\widetilde{e}^f_{t}\|^2}{\alpha_{t-1}}
\leq \left[\frac{(1-\eta_{t+1})^2}{\alpha_t} - \frac{1}{\alpha_{t-1}} + \frac{4L_{g_{xy}}r_v^2\delta_{\epsilon}}{\alpha_t}\right]\mathbb{E}\|\widetilde{e}^f_{t}\|^2 + \frac{4(\eta^f_{t+1})^2}{\alpha_t}\sigma^2_f  \nonumber\\
&\quad \ + 6L_F^2\alpha_t \mathbb{E}\|\widetilde{h}_t^f\|^2 + \frac{12L_F^2\beta_t^2}{\alpha_t}\mathbb{E}\|e_t^g\|^2 + \frac{12L_F^2\beta_t^2}{\alpha_t}\mathbb{E}\|\nabla_y g(x_t,y_t)\|^2 \nonumber\\
&\quad \ + 12C_{g_{xy}}\frac{\lambda_t^2}{\alpha_t}\Big(\mathbb{E}\|\widetilde{e}_t^R\|^2 + L_g^2\mathbb{E}\|v_t-v_t^*\|^2\Big) +\big( 4L_{g_{xy}}r_v^2\delta_{\epsilon}+16L^2_{g_{xy}}r_v^4\delta_{\epsilon}^2\big),
\end{align}
where 
the inequality follows from the fact that $0 < 1-\eta_t < 1$ for all $t \in \{0,1,...,T-1\}$. 
Now considering the coefficient of the first term on the right hand side of the above \cref{eq:eftplus1_free} and recalling the $\delta_{\epsilon}$ in \cref{def:parameters_free}, we have
\begin{align}\label{eq:fracalpha_free}
    \frac{(1-\eta^f_{t+1})^2}{\alpha_t} - \frac{1}{\alpha_{t-1}} + \frac{4L_{g_{xy}}r_v^2\delta_{\epsilon}}{\alpha_t}\leq \frac{1}{\alpha_{t}} - 
    \frac{\eta^f_{t+1}}{2\alpha_t} - \frac{1}{\alpha_{t-1}}.
\end{align}
Using the definition of $\alpha_t$ in \cref{def:parameters_free}, we have
\begin{align}\label{eq:at_at-1_free}
    \frac{1}{\alpha_{t}} - \frac{1}{\alpha_{t-1}} &= (w+t)^{1/3} - (w+t-1)^{1/3} \overset{(a)}{\leq} \frac{1}{3(w+t-1)^{2/3}} \overset{(b)}{\leq} \frac{1}{3(w/2+t)^{2/3}}\nonumber \\
    &=  \frac{2^{2/3}}{3(w+2t)^{2/3}} \leq \frac{2^{2/3}}{3(w+t)^{2/3}} \overset{(c)}{\leq} \frac{2^{2/3}}{3}\alpha_t^2 \overset{(d)}{\leq} \frac{\alpha_t}{3L_f}, 
\end{align}
where (a) follows from $(x+y)^{1/3} - x^{1/3} \leq y/(3x^{2/3})$, (b) 
follows because 
we choose $w \geq 2$ hence $1 \leq w/2$, 
(c) follows 
from the definition of $\alpha_t$ and (d) follows because 
we choose $\alpha_t \leq 1/3L_f$. Substituting 
\cref{eq:at_at-1_free} 
into~\cref{eq:fracalpha_free} and using $\eta_{t+1}^f = c_{\eta_f}\alpha^2_t$, we have 
\begin{align}\label{eq:1subetaf_free}
    \frac{(1-\eta^f_{t+1})^2}{\alpha_t} - \frac{1}{\alpha_{t-1}} \leq \frac{\alpha_t}{3L_f} - \frac{c_{\eta_f}\alpha_t}{2} \leq - \bar{c}_{\eta_f}\alpha_t,
\end{align}
where the inequalities follow from 
$c_{\eta_f} = \frac{2}{3L_f} + 2\bar{c}_{\eta_f}$
with $\bar{c}_{\eta_f}$ in \cref{def:parameters_free}.
Then substituting \cref{eq:1subetaf_free} into \cref{eq:eftplus1_free} yields 
\begin{align}\label{ieq:errorf_free}
\frac{1}{\bar{c}_{\eta_f}} \mathbb{E}&\bigg[\frac{\|\widetilde{e}^f_{t+1}\|^2}{\alpha_t} - \frac{\|\widetilde{e}^f_{t}\|^2}{\alpha_{t-1}}\bigg]\nonumber
\\\leq &-\alpha_t\mathbb{E}\|\widetilde{e}^f_{t}\|^2 + \frac{2(\eta^f_{t+1})^2}{\bar{c}_{\eta_f}\alpha_t}\sigma_f^2 + \frac{\alpha_t}{16}\mathbb{E}\|\widetilde{h}_t^f\|^2 + \frac{L^2_{f_y}+\frac{C_{f_y}L^2_{g_{xy}}}{\mu_g^2}}{4L_{\mu_g}^2}\alpha_t\mathbb{E}\|e_t^g\|^2  \nonumber\\
& + \frac{L^2_{f_y}+\frac{C_{f_y}L^2_{g_{xy}}}{\mu_g^2}}{4L_{\mu_g}(\mu_g+L_g)}\alpha_t\mathbb{E}\|\nabla_y g(x_t, y_t)\|^2 + \frac{8C_{g_{xy}}}{L^2_{\mu_g}}\alpha_t \mathbb{E}\|\widetilde{e}_t^R\|^2 + C_{g_{xy}}\alpha_t\mathbb{E}\|v_t - v_t^*\|^2.
\end{align}

\noindent{\bf Step 4. Bound $\mathbb{E}\|e_t^g\|^2$ in \cref{def:potentialF_free}.} 

\noindent Next, from~\Cref{lm:errorg_free}, we have 
\begin{align}\label{eq:etplus1g_free}
\frac{\mathbb{E}\|e_{t+1}^g\|^2}{\alpha_t} - \frac{\mathbb{E}\|e_{t}^g\|^2}{\alpha_{t-1}}
&\leq \left[ \frac{(1-\eta^g_{t+1})^2 + 32(1-\eta^g_{t+1})^2L^2_g\beta^2_t}{\alpha_t} - \frac{1}{\alpha_{t-1}}\right]\mathbb{E}\|e_t^g\|^2 + \frac{2(\eta^g_{t+1})^2}{\alpha_t}\sigma^2_g\nonumber \\
&\quad \ +16L_g^2\alpha_t\mathbb{E}\|\widetilde{h}_t^f\|^2 + \frac{32L_g^2\beta_t^2}{\alpha_t}\mathbb{E}\|\nabla_y g(x_t,y_t)\|^2,
\end{align}
where we use the 
fact that $0<1-\eta_t^g \leq 1$for all $t\in\{0,1,...,T-1\}$. 
Let us consider the coefficient of the first term on the right hand side 
of the above \cref{eq:etplus1g_free}. In specific, 
we have
\begin{align*}
\frac{(1-\eta^g_{t+1})^2 + 32(1-\eta^g_{t+1})^2L^2_g\beta^2_t}{\alpha_t} - \frac{1}{\alpha_{t-1}} &\leq \frac{(1-\eta^g_{t+1})}{\alpha_t}(1+32L^2_g\beta^2_t) - \frac{1}{\alpha_{t-1}} \\
&= \frac{1}{\alpha_t} - \frac{1}{\alpha_{t-1}} + \frac{32L^2_g\beta^2_t}{\alpha_t} - c_{\eta_g}\alpha_t(1+32L^2_g\beta^2_t),
\end{align*}
which, combined with \cref{eq:at_at-1_free} that $\frac{1}{\alpha_t} - \frac{1}{\alpha_{t-1}}\leq \frac{\alpha_t}{3L_f}$ and the definition of $\beta_t = c_{\beta}\alpha_t$, yields
\begin{align}\label{eq:etatplusoneg_free}
\frac{(1-\eta^g_{t+1})^2 + 32(1-\eta^g_{t+1})^2L^2_g\beta^2_t}{\alpha_t} - \frac{1}{\alpha_{t-1}} \leq \frac{\alpha_t}{3L_f} + 32L_g^2c_{\beta}^2\alpha_t - c_{\eta_g}\alpha_t.
\end{align}
Recall from \cref{def:parameters_free} that we choose 
$$
c_{\eta_g} = \frac{1}{3L_f} + 32L_g^2c_{\beta}^2 + \frac{17(L^2_{f_y}+\frac{C_{f_y}L^2_{g_{xy}}}{\mu_g^2})}{L^2_{\mu_g}}\bar{c}_{\eta_g},
$$
which, in conjunction with \cref{eq:etatplusoneg_free}, yields 
\begin{align}\label{eq:finaletas_free}
\frac{(1-\eta^g_{t+1})^2 + 32(1-\eta^g_{t+1})^2L^2_g\beta^2_t}{\alpha_t} - \frac{1}{\alpha_{t-1}} \leq - \frac{17(L^2_{f_y}+\frac{C_{f_y}L^2_{g_{xy}}}{\mu_g^2})}{L^2_{\mu_g}}\bar{c}_{\eta_g}\alpha_t.
\end{align}
Substituting \cref{eq:finaletas_free} into \cref{eq:etplus1g_free} yields 
\begin{align}\label{ieq:errorg_free}
    \frac{1}{\bar{c}_{\eta_g}}\left[\frac{\mathbb{E}\|e_{t+1}^g\|^2}{\alpha_t} - \frac{\mathbb{E}\|e_{t}^g\|^2}{\alpha_{t-1}}\right] 
    &\leq - \frac{17(L^2_{f_y}+\frac{C_{f_y}L^2_{g_{xy}}}{\mu_g^2})}{L^2_{\mu_g}}\alpha_t\mathbb{E}\|e_t^g\|^2 + \frac{2(\eta^g_{t+1})^2}{\bar{c}_{\eta_g}\alpha_t}\sigma^2_g \nonumber\\
    &\quad \ + \frac{\alpha_t}{16}\mathbb{E}\|h_t^f\|^2
    + \frac{L^2_{f_y}+\frac{C_{f_y}L^2_{g_{xy}}}{\mu_g^2}}{4L_{\mu_g}(\mu_g+L_g)}\alpha_t\mathbb{E}\|\nabla_y g(x_t, y_t)\|^2.
\end{align}

\noindent{\bf Step 5. Bound $\mathbb{E}\|\widetilde{e}_t^R\|^2$ in \cref{def:potentialF_free}.} 

\noindent Next, from~\Cref{lm:errorR_free}, we have 
\begin{align}\label{eq:etRRs_free}
    \frac{\mathbb{E}\|\widetilde{e}_{t+1}^R\|^2}{\alpha_t} - & \frac{\mathbb{E}\|\widetilde{e}_{t}^R\|^2}{\alpha_{t-1}} 
    \leq \Bigg[\frac{(1-\eta_{t+1}^R)^2(1 + 96L_g^2\lambda_t^2)}{\alpha_t} - \frac{1}{\alpha_{t-1}} + 4L_{g_{yy}}r_v^2\delta_{\epsilon}\Bigg]\mathbb{E}\|\widetilde{e}_{t}^R\|^2 \nonumber \\
    &\quad \ +\Bigg[\frac{8(\eta^R_{t+1})^2(\sigma^2_{g_{yy}}r^2_v+\sigma^2_{f_y})}{\alpha_t}\Bigg] + 96(1-\eta_{t+1}^R)^2(L^2_{g_{yy}}r^2_v+L^2_{f_y})\alpha_t\mathbb{E}\|\widetilde{h}_t^f\|^2 \nonumber \\
    &\quad \ +96(1-\eta_{t+1}^R)^2(L^2_{g_{yy}}r^2_v+L^2_{f_y})c^2_\beta \alpha_t(2\mathbb{E}\|e_t^g\|^2 + 2\mathbb{E}\|\nabla_y g(x_t,y_t)\|^2) \nonumber \\
    &\quad \ +96(1-\eta_{t+1}^R)^2L^4_g c_\lambda^2\alpha_t \mathbb{E}\|v_t-v^*\|^2 + \frac{1}{\alpha_t}\big(4L_{g_{yy}}r_v^2\delta_{\epsilon} + 8L^2_{g_{yy}}r_v^4\delta_{\epsilon}^2\big).
\end{align}

\noindent For the first term of right hand side of \cref{eq:etRRs_free}, recalling the $\delta_{\epsilon}$ in \cref{def:parameters_free}, we have
\begin{align}\label{eq:firscoeefient_free}
\frac{(1-\eta_{t+1}^R)^2(1 + 96L_g^2\lambda_t^2 )}{\alpha_t}& - \frac{1}{\alpha_{t-1}} + \frac{4L_{g_{yy}}r_v^2\delta_{\epsilon}}{\alpha_t} \nonumber
\\&\leq \frac{1-\eta_{t+1}^R}{\alpha_t}(1 + 96L_g^2\lambda_t^2 ) - \frac{1}{\alpha_{t-1}} + \frac{4L_{g_{yy}}r_v^2\delta_{\epsilon}}{\alpha_t} \nonumber\\
& = \frac{1}{\alpha_t} - \frac{1}{\alpha_{t-1}} - \frac{\eta_{t+1}^R}{2\alpha_t} + \frac{1-\eta_{t+1}^R}{\alpha_t}\cdot96L_g^2\lambda_t^2 \nonumber\\
&\overset{(a)}{=}\frac{1}{\alpha_t} - \frac{1}{\alpha_{t-1}} - \frac{c_{\eta_R}\alpha_t}{2} + \left(\frac{1}{\alpha_t} - c_{\eta_R}\alpha_t\right)\cdot96L_g^2c_\lambda^2\alpha_t^2\nonumber\\
&\overset{(b)}{\leq} \frac{\alpha_t}{3L_f} + 96L_g^2c_\lambda^2\alpha_t - \frac{c_{\eta_R}\alpha_t}{2},
\end{align}
where (a) follows from the definition that $\eta^R_{t+1} = c_{\eta_R}\alpha_t^2$, and (b) follows from \cref{eq:at_at-1_free} that $\frac{1}{\alpha_t} - \frac{1}{\alpha_{t-1}} \leq \frac{\alpha_t}{3L_f}$.
Recalling $\bar{c}_{\eta_R}$ from \cref{def:parameters_free} that 
$$
c_{\eta_R} = \frac{2}{3L_f} + 192L_g^2c_\lambda^2 + \frac{32C_{g_{xy}}}{L_{\mu_g}^2}\bar{c}_{\eta_R}
$$
which, in conjunction with \cref{eq:firscoeefient_free}, yields
\begin{align}\label{eq:intermidiatResult_free}
\frac{(1-\eta_{t+1}^R)^2(1 + 96L_g^2\lambda_t^2 )}{\alpha_t} - \frac{1}{\alpha_{t-1}} + \frac{4L_{g_{yy}}r_v^2\delta_{\epsilon}}{\alpha_t}v\leq -\frac{16C_{g_{xy}}}{L_{\mu_g}^2}\bar{c}_{\eta_R}.
\end{align}
Incorporating \cref{eq:intermidiatResult_free} into \cref{eq:etRRs_free}, 
recalling from \cref{def:parameters_free} that  $\sigma_R := \sqrt{\sigma^2_{g_{yy}}r^2_v+\sigma^2_{f_y}}$, 
and multiplying both sides of \cref{eq:etRRs_free}
by $\frac{1}{\bar{c}_{\eta_R}}$, we have 
\begin{align}\label{ieq:errorR_free}
\frac{1}{\bar{c}_{\eta_R}}\mathbb{E}\left[\frac{\|\widetilde{e}^R_{t+1}\|^2}{\alpha_t} - \frac{\|\widetilde{e}^R_{t+1}\|^2}{\alpha_{t-1}}\right]
&\leq -\frac{16C_{g_{xy}}}{L_{\mu_g}^2}\alpha_t \|\widetilde{e}^R_t\|^2 + \frac{8(\eta_{t+1}^R)^2}{\bar{c}_{\eta_R}\alpha_t}\sigma_R^2 
+\frac{\alpha_t}{16}\mathbb{E}\|\widetilde{h}_{t}^f\|^2 \nonumber\\
&\ +\frac{L^2_{f_y}+\frac{C_{f_y}L^2_{g_{xy}}}{\mu_g^2}}{4L_{\mu_g}^2}\alpha_t \mathbb{E}\|e_t^g\|^2 + \frac{L^2_{f_y}+\frac{C_{f_y}L^2_{g_{xy}}}{\mu_g^2}}{4L_{\mu_g}(\mu_g+L_g)}\alpha_t\mathbb{E}\|\nabla_y g(x_t, y_t)\|^2 \nonumber\\
&\quad \ +C_{g_{xy}}\alpha_t\|v_t-v_t^*\|^2 + \frac{1}{\alpha_t\bar{c}_{\eta_R}}\big(4L_{g_{yy}}r_v^2\delta_{\epsilon} + 8L^2_{g_{yy}}r_v^4\delta_{\epsilon}^2\big).
\end{align}

\noindent{\bf Step 6. Merging the results of Step 1-5 to prove \cref{def:potentialF_free}.} 

\noindent Finally, adding~\cref{ieq:yy_free}, ~\cref{ieq:vv_free}~\cref{ieq:errorf_free}, ~\cref{ieq:errorg_free}, ~\cref{ieq:errorR_free} and the result of~\Cref{lm:corefunction_free} with $\alpha_t \leq \frac{1}{3L_f}$ yields 
\begin{align}\label{ieq:Vtplus1_Vt_free}
\mathbb{E}[\widetilde{V}_{t+1} - \widetilde{V}_{t}] &\leq -\frac{\alpha_t}{2}\mathbb{E}\|\nabla \Phi(x_t)\|^2 + \frac{4(\eta^f_{t+1})^2}{\bar{c}_{\eta_f} \alpha_t}\sigma_f^2 + \frac{2(\eta^g_{t+1})^2}{\bar{c}_{\eta_g} \alpha_t}\sigma_g^2 + 
\frac{8(\eta^R_{t+1})^2}{\bar{c}_{\eta_R} \alpha_t}\sigma_R^2 \nonumber\\
& \quad \ + \frac{1}{\alpha_t\bar{c}_{\eta_f}}\big(4L_{g_{xy}}r_v^2\delta_{\epsilon} + 16L^2_{g_{xy}}r_v^4\delta_{\epsilon}^2\big) + \frac{1}{\alpha_t\bar{c}_{\eta_R}}\big(4L_{g_{yy}}r_v^2\delta_{\epsilon} + 8L^2_{g_{yy}}r_v^4\delta_{\epsilon}^2\big).
\end{align}
Then, the proof is complete. 
\end{proof}

\subsection{Proof of~\Cref{th:hessianfree}}
\begin{proof}
Summing up the result of~\Cref{lm:mergeV_free} for $t = 0$ to $T-1$, dividing by $T$ on both sides and using the definitions that  $\eta_{t+1}^f := c_{\eta_f}\alpha_t^2$, $\eta_{t+1}^g := c_{\eta_g}\alpha_t^2$, $\eta_{t+1}^R := c_{\eta_R}\alpha_t^2$, we have 
\begin{align}\label{eq:sumV_free}
   & \frac{\mathbb{E}[\widetilde{V}_{T} - \widetilde{V}_{0}]}{T} \leq -\frac{1}{T}\sum_{t=0}^{T-1}\frac{\alpha_t}{2}\mathbb{E}\|\nabla \Phi(x_t)\|^2 + \frac{1}{T}\left[ \frac{4(c_{\eta_f})^2}{\bar{c}_{\eta_f}}\sigma_f^2 + \frac{2(c_{\eta_g})^2}{\bar{c}_{\eta_g}}\sigma_g^2 + \frac{8(c_{\eta_R})^2}{\bar{c}_{\eta_R}}\sigma_R^2 \right]\sum_{t=0}^{T-1}\alpha_t^3 \nonumber\\ 
    &\quad \ + \sum_{t=0}^{T-1}\frac{1}{\alpha_t\bar{c}_{\eta_f}}\big(4L_{g_{xy}}r_v^2\delta_{\epsilon} + 16L^2_{g_{xy}}r_v^4\delta_{\epsilon}^2\big) + \sum_{t=0}^{T-1}\frac{1}{\alpha_t\bar{c}_{\eta_R}}\big(4L_{g_{yy}}r_v^2\delta_{\epsilon} + 8L^2_{g_{yy}}r_v^4\delta_{\epsilon}^2\big).
\end{align}
Next based on the definition of $\alpha_t$ in \cref{def:parameters_free}, we have 
\begin{align}\label{eq:sumalpha_free}
    \sum_{t=0}^{T-1}\alpha_t^3 &= \sum_{t=0}^{T-1}\frac{1}{w + t}
    \overset{(a)}{\leq} \sum_{t=0}^{T-1}\frac{1}{1 + t} \leq \log(T+1) \nonumber \\
    \sum_{t=0}^{T-1}\frac{1}{\alpha_t} &= \sum_{t=0}^{T-1}\frac{1}{(w + t)^{1/3}} \overset{(a)}{\leq}\sum_{t=0}^{T-1} \frac{1}{(1 + t)^{1/3}} \leq \frac{3}{2}T^{2/3}.
\end{align}
where inequality (a) results from the fact that we choose $w \geq 1$. 

By plugging ~\cref{eq:sumalpha_free} in~\cref{eq:sumV_free}, we have 
\begin{align}\label{eq:methEvt_free}
    \frac{\mathbb{E}[\widetilde{V}_{T} - \widetilde{V}_{0}]}{T} &\leq -\frac{1}{T}\sum_{t=0}^{T-1}\frac{\alpha_t}{2}\mathbb{E}\|\nabla \Phi(x_t)\|^2
    + \left[ \frac{4c_{\eta_f}^2}{\bar{c}_{\eta_f}}\sigma_f^2 
    + \frac{2c_{\eta_g}^2}{\bar{c}_{\eta_g}}\sigma_g^2 + \frac{8c_{\eta_R}^2}{\bar{c}_{\eta_R}}\sigma_R^2 \right] \frac{\log(T+1)}{T} \nonumber \\
    &\quad \ + T^{2/3}\big(6L_{g_{xy}}r_v^2\delta_{\epsilon} + 24L^2_{g_{xy}}r_v^4\delta_{\epsilon}^2\big) + T^{2/3}\big(6L_{g_{yy}}r_v^2\delta_{\epsilon} + 12L^2_{g_{yy}}r_v^4\delta_{\epsilon}^2\big).
\end{align}
Rearrange the terms in \cref{eq:methEvt_free}, we have
\begin{align*}
    \frac{1}{T}\sum_{t=0}^{T-1}\frac{\alpha_t}{2}\mathbb{E}\|\nabla \Phi(x_t)&\|^2 \leq \frac{\mathbb{E}[\widetilde{V}_0 - \Phi^*]}{T}
    + \left[ \frac{4c_{\eta_f}^2}{\bar{c}_{\eta_f}}\sigma_f^2 
    + \frac{2c_{\eta_g}^2}{\bar{c}_{\eta_g}}\sigma_g^2 + \frac{8c_{\eta_R}^2}{\bar{c}_{\eta_R}}\sigma_R^2 \right]\frac{\log(T+1)}{T} \\
    &\quad \ + T^{2/3}\big(6L_{g_{xy}}r_v^2\delta_{\epsilon} + 24L^2_{g_{xy}}r_v^4\delta_{\epsilon}^2\big) + T^{2/3}\big(6L_{g_{yy}}r_v^2\delta_{\epsilon} + 12L^2_{g_{yy}}r_v^4\delta_{\epsilon}^2\big),
\end{align*}
which, in conjunction with the fact that $\alpha_t$ is decreasing w.r.t.~$t$ and multiplying by $2/\alpha_T$
on both sides, yields 
\begin{align}\label{eq:1Tephis_free}
    \frac{1}{T}\sum_{t=0}^{T-1}\mathbb{E}\|\nabla \Phi&(x_t)\|^2 \leq \frac{2\mathbb{E}[\widetilde{V}_0 - \Phi^*]}{\alpha_T T}
    + \left[ \frac{8c_{\eta_f}^2}{\bar{c}_{\eta_f}}\sigma_f^2 
    + \frac{4c_{\eta_g}^2}{\bar{c}_{\eta_g}}\sigma_g^2 + \frac{16c_{\eta_R}^2}{\bar{c}_{\eta_R}}\sigma_R^2 \right]\frac{\log(T+1)}{\alpha_T T} \nonumber\\
    &\quad \ + \frac{T^{2/3}}{\alpha_t}\big(6L_{g_{xy}}r_v^2\delta_{\epsilon} + 24L^2_{g_{xy}}r_v^4\delta_{\epsilon}^2\big) + \frac{T^{2/3}}{\alpha_t}\big(6L_{g_{yy}}r_v^2\delta_{\epsilon} + 12L^2_{g_{yy}}r_v^4\delta_{\epsilon}^2\big).
\end{align}
Finally, based on the definition of the potential function, we have 
\begin{align}\label{eq:evo_ss_free}
    \mathbb{E}[\widetilde{V}_0] :=& \mathbb{E} 
    \Big[ 
    \Phi(x_0) + \frac{2L}{3\sqrt{2}L_y}\|y_0 - y^*(x_0)\|^2
    + \frac{4C_B}{C_\lambda L_{\mu_g}}\|v_0 - v^*(x_0, y_0)\|^2 \nonumber
     \\
    &\quad + \frac{1}{\bar{c}_{\eta_f}}\frac{\|\widetilde{e}_t^f\|^2}{\alpha_{t-1}}
    + \frac{1}{\bar{c}_{\eta_g}}\frac{\|e_t^g\|^2}{\alpha_{t-1}}+ \frac{1}{\bar{c}_{\eta_R}}\frac{\|\widetilde{e}_t^R\|^2}{\alpha_{t-1}} 
    \Big] \nonumber \\
    \leq &
    \Phi(x_0) + \frac{2L}{3\sqrt{2}L_y}\|y_0 - y^*(x_0)\|^2
    + \frac{4C_B}{C_\lambda L_{\mu_g}}\|v_0 - v^*(x_0, y_0)\|^2 \nonumber
    \\&
    + \frac{1}{\bar{c}_{\eta_f}}\frac{\sigma_f^2}{\alpha_{t-1}} 
    + \frac{1}{\bar{c}_{\eta_g}}\frac{\sigma_g^2}{\alpha_{t-1}}
    + \frac{1}{\bar{c}_{\eta_R}}\frac{\sigma_R^2}{\alpha_{t-1}} 
\end{align}
where 
the inequality follows from Assumption~\ref{as:ulf},~\ref{as:llf},~\ref{as:sf},~\cref{lm:B} and the definitions of $\widetilde{h}^f_t$, $h^g_t$ and $\widetilde{h}^R_t$ in~\cref{def:htf_free},~\cref{def:htg},~\cref{def:htR_free}. 
Then, substituting \cref{eq:evo_ss_free} into \cref{eq:1Tephis_free} yields 
\begin{align}\label{eq:almostthere}
    \frac{1}{T}&\sum_{t=0}^{T-1}\|\nabla \Phi(x_t)\|^2 
    \leq
    \frac{2[\Phi(x_0) - \Phi^*]}{\alpha_t T} 
    + \frac{4L}{3\sqrt{2}L_y}\frac{\|y_0 - y^*(x_0)\|^2}{\alpha_t T}
    + \frac{8C_B}{C_{\lambda}L_{\mu_y}}\frac{\|v_0 - v^*(x_0, y_0)\|^2}{\alpha_t T} \nonumber\\
    &\quad \ + \frac{2}{\alpha_{-1}\alpha_{T}T}\left( \frac{\sigma^2_f}{\bar{c}_{\eta_f}}+\frac{\sigma^2_g}{\bar{c}_{\eta_g}}+\frac{\sigma^2_R}{\bar{c}_{\eta_R}} \right) 
    + \left[ \frac{8c_{\eta_f}^2}{\bar{c}_{\eta_f}}\sigma_f^2 
    + \frac{4c_{\eta_g}^2}{\bar{c}_{\eta_g}}\sigma_g^2 + \frac{16c_{\eta_R}^2}{\bar{c}_{\eta_R}}\sigma_R^2 \right]\frac{\log(T+1)}{\alpha_T T} \nonumber\\
    &\quad \ + \frac{T^{2/3}}{\alpha_t}\big(6L_{g_{xy}}r_v^2\delta_{\epsilon} + 24L^2_{g_{xy}}r_v^4\delta_{\epsilon}^2 + 6L_{g_{yy}}r_v^2\delta_{\epsilon} + 12L^2_{g_{yy}}r_v^4\delta_{\epsilon}^2\big).
\end{align}
Recalling $\delta_{\epsilon}$ in \cref{def:parameters_free}, we choose our $\delta_{\epsilon}$ as 
{\small
\begin{align}\label{def:delta_epsilon}
    \delta_{\epsilon} \leq \min \Bigg\{&\frac{c_{\eta_f}}{8(L_{g_{xy}}r_v^2(w+T-1)^{2/3})}, \frac{c_{\eta_R}}{8(L_{g_{xy}}r_v^2(w+T-1)^{2/3})}, \frac{(w+T)^{\frac{1}{3}}}{12T^{\frac{4}{3}}L_{g_{xy}}r_v^2}, \frac{(w+T)^{\frac{1}{3}}}{12T^{\frac{4}{3}}L_{g_{yy}}r_v^2} \Bigg\}.
\end{align}}
Substituting \cref{def:delta_epsilon} and the definitions of $\alpha_T := \frac{1}{(\omega + T)^{1/3}}$ and $\alpha_{-1} = \alpha_{0}$ into \cref{eq:almostthere}, yields 
\begin{align*}
    \mathbb{E}\|\nabla \Phi\big(x_a(T)\big)\|^2 &\leq \widetilde{\mathcal{O}}\Bigg( \frac{\Phi(x_0) - \Phi^*}{T^{2/3}}
    + \frac{\|y_0 - y^*(x_0)\|^2}{T^{2/3}}
    + \frac{\|v_0 - v^*(x_0, y_0)\|^2}{T^{2/3}} \\
    &\qquad \quad + \frac{1}{T^{2/3}} + \frac{\sigma_f^2}{T^{2/3}}
    +  \frac{\sigma_g^2}{T^{2/3}}
    +  \frac{\sigma_R^2}{T^{2/3}}\Bigg) .
\end{align*}
Finally, the proof is complete. 
\end{proof}

\vspace*{1cm}

\end{document}